\documentclass{daj}
\newcommand{\mb}{\mathbf}
\newcommand{\st}{\operatorname{st}}
\newcommand{\lcirc}{\leftidx^{\circ}\mathbf}
\usepackage{amsmath,centernot, amssymb,amsthm,mathrsfs,leftidx,bm}

\newtheorem{theorem}[subsection]{Theorem}
\newtheorem{lemma}[subsection]{Lemma}
\newtheorem{proposition}[subsection]{Proposition}
\newtheorem{corollary}[subsection]{Corollary}
\newtheorem*{claim}{Claim}
\numberwithin{equation}{section}

\theoremstyle{definition}

\newtheorem{definition}[subsection]{Definition}
\newtheorem{notation}[subsection]{Notation}
\newtheorem{observation}[subsection]{Observation}

\newtheorem{remark}[subsection]{Remark}
\newtheorem{example}[subsection]{Example}

\dajAUTHORdetails{%
  title = {Semicontinuity of Structure for Small Sumsets in Compact Abelian Groups}, 
  author = {John T. Griesmer},
  plaintextauthor = {John T. Griesmer},
  runningtitle = {Semicontinuity for small sumsets},
  keywords = {sumset, inverse theorem, ultraproduct},
}   

\dajEDITORdetails{%
   year={2019},
   number={18},
   received={27 August 2018},   
   revised={20 May 2019},    
   published={29 November 2019},  
   doi={10.19086/da.11089},       
}   

\begin{document}

\begin{frontmatter}[classification=text]

\title{Semicontinuity of Structure for Small Sumsets in Compact Abelian Groups} 

\author[john]{John T. Griesmer}

\begin{abstract}
We study pairs of subsets $A, B$ of a compact abelian group $G$ where the sumset $A+B:=\{a+b: a\in A, b\in B\}$ is small.  Let $m$ and $m_{*}$ be Haar measure and inner Haar measure on $G$, respectively. Given $\varepsilon>0$, we classify all pairs $A,B$ of Haar measurable subsets of $G$ satisfying $m(A), m(B)>\varepsilon$ and $m_{*}(A+B)\leq m(A)+m(B)+\delta$ where $\delta=\delta(\varepsilon)$ is small.  We also study the case where the $\delta$-popular sumset $A+_{\delta}B:=\{t\in G: m(A\cap (t-B))>\delta\}$ is small. We prove that for all $\varepsilon>0$, there is a $\delta>0$ such that if $A$ and $B$ are subsets of a compact abelian group $G$ having $m(A), m(B)>\varepsilon$ and $m(A+_{\delta}B)\leq m(A)+m(B)+\delta$, then there are sets $S, T\subseteq G$ such that $m(A\triangle S)+m(B\triangle T)<\varepsilon$ and $m(S+T)\leq m(S)+m(T)$.  Appealing to known results, the latter inequality yields strong structural information on $S$ and $T$, and therefore on $A$ and $B$.
\end{abstract}
\end{frontmatter}

\section{Sumsets in compact abelian groups}\label{sec:Intro}

If $G$ is an abelian group and $A$ and $B$ are subsets of $G$, let $A+B$ denote the \emph{sumset}  $\{a+b:a\in A, b\in B\}$  and $A-B$ denote the   \emph{difference set} $\{a-b:a\in A, b\in B\}$.  If $t\in G$ write $A+t$ for the \emph{translate} $\{a+t:a\in A\}$.  Let $-A$ denote the set $\{-a:a\in A\}$. In this article all topological groups are assumed to be Hausdorff, and all finite groups are endowed with the discrete topology.

If $G$ is a compact abelian group with Haar probability measure $m$, let $m_{*}$ denote the corresponding \emph{inner Haar measure}: $m_{*}(A)=\sup\{m(C):C\subseteq A,C\text{ is compact}\}$.  We say that a set $A\subseteq G$ is \emph{Haar measurable} (or \emph{$m$-measurable}) if it lies in the completion of the Borel $\sigma$-algebra of $G$ with respect to $m$.  We are forced to consider inner Haar measure, as the sumset of two $m$-measurable sets can fail to be $m$-measurable. The structure  of Haar measurable sets $A,B\subseteq G$ satisfying $m(A),m(B)>0$ and
\begin{equation}\label{eqn:Critical}
m_{*}(A+B)\leq m(A)+m(B)
\end{equation}
is well understood; see
\S\ref{sec:LCAreview} for a brief (and incomplete) summary or  \cite{GriesmerLCA, GrynkKemperman,GrynkStep, Kneser56,TaoInverse}  for comprehensive exposition.

For $\delta\geq 0$, let $A+_{\delta}B:=\{t\in G: m(A\cap (t-B))>\delta\}$ be the \emph{$\delta$-popular sumset}.  The function $t\mapsto m(A\cap (t-B))$ is continuous, so  $A+_{\delta}B$ is an open subset of $A+B$. Thus the hypothesis $m(A+_{\delta}B)\leq m(A)+m(B)+\delta$ is weaker than Inequality (\ref{eqn:Critical}).  Our first result is the following theorem, which says that for very small $\delta$, pairs of sets satisfying this weaker inequality resemble pairs satisfying (\ref{eqn:Critical}). This resolves the special case of Conjecture 5.1 in \cite{TaoInverse} with the additional assumption that $G$ is abelian.  Here $A\triangle B$ denotes the symmetric difference $(A\setminus B)\cup (B\setminus A)$.

\begin{theorem}\label{thm:Popular}
	For all $\varepsilon>0$ there exists $\delta>0$ such that for every compact abelian group $G$ with Haar probability measure $m$ and all $m$-measurable sets $A,B\subseteq G$ with $m(A), m(B)>\varepsilon$ and $m(A+_{\delta}B)\leq m(A)+m(B)+\delta$, there exist $m$-measurable sets $S,T\subseteq G$ such that $m(A\triangle S)+m(B\triangle T)<\varepsilon$, $S+T$ is $m$-measurable, and $m(S+T)\leq m(S)+m(T)$.
\end{theorem}
Note: $\delta$ depends only on $\varepsilon$ and not on $G$ in Theorem \ref{thm:Popular}.

Theorem \ref{thm:Popular} generalizes Theorem 1.5 of \cite{TaoInverse}, which imposes the additional assumption that $G$ is connected. A quantitative version of this result in the case where $G$ is a cyclic group of prime order is proved in \cite{MazurZp}. See also \cite{MazurInZ} for a version in $\mathbb Z$, and  more recently \cite{ShaoXu}. As a byproduct of the proof of Theorem \ref{thm:Popular}, we provide in Theorem \ref{thm:Precise} a classification of sets satisfying $m_{*}(A+B)\leq m(A)+m(B)+\delta$, where $\delta$ is very small compared to $m(A)$ and $m(B)$.  This generalizes Theorem 1.3 of \cite{TaoInverse}, which assumes connectedness of $G$.  Theorems \ref{thm:Popular} and \ref{thm:Precise} are new for every infinite compact disconnected abelian group, such as $\mathbb T\times (\mathbb Z/N \mathbb Z)$.

Before stating Theorem \ref{thm:Precise} we introduce some terminology and notation.

\begin{notation}
	Let $(X,\mu)$ be a measure space and $A, B\subseteq X$. We write $A\sim_{\mu} B$ if $\mu(A\triangle B)=0$ and write $A\subset_{\mu} B$ if $\mu(A\setminus B)=0$.  If $f$ and $g$ are functions on $X$, we write $f\equiv_{\mu} g$ if $f(x)=g(x)$ for $\mu$-almost every $x$.
\end{notation}
For any given set $X$, we will consider only one measure $\mu$ on $X$, and hence only one $\sigma$-algebra of measurable sets; we will not write the $\sigma$-algebra explicitly.

If $G$ is an abelian group, a measure $\mu$ on $G$ is \emph{translation invariant} if for every $\mu$-measurable set $A\subseteq G$ and every $t\in G$, $A+t$ is $\mu$-measurable and $\mu(A+t)=\mu(A)$.  For such a measure and $f, g\in L^{2}(\mu)$, we define the \emph{$\mu$-convolution} $f*_{\mu}g$ by $f*_{\mu}g(x):=\int f(t)g(x-t) \,d\mu(t)$.   In the sequel the measure $\mu$ may be understood from context, and we may write $f*g$ in place of $f*_{\mu}g$. Note that if $A:=\{x:f(x)>0\}$ and $B:=\{x:g(x)>0\}$,  then $\{x: f*_{\mu}g(x)>0\}\subseteq A+B$.   Furthermore, the $\delta$-popular sumset defined above can be written in terms of convolution as $A+_{\delta} B=\{x\in G: 1_{A}*_{m}1_{B}(x)>\delta \}$.  Theorem \ref{thm:Popular} thereby yields information on the supports of $f$ and $g$ under a strong assumption on a level set of $f*g$.  For connected groups, the recent article \cite{ChristIlioupoulou} proves more quantitative and detailed results on the relationship between $f$, $g$, and $f*g$.

If $H$ is a compact abelian group and $\pi:G\to H$ is a homomorphism, we say that $\pi$ is \emph{$\mu$-measurable} if $\pi^{-1}(A)$ is $\mu$-measurable for every open set $A\subseteq H$.  If $m$ is Haar probability measure on $H$, we say that $\pi$ \emph{preserves $\mu$} if $\mu(\pi^{-1}(A))=m(A)$ for every $m$-measurable set $A\subseteq H$.  Note that $\pi$ preserves $\mu$ if and only if $\int f\circ \pi \,d\mu= \int f \,dm$ for all $f\in L^{1}(m)$.

If the measures $\mu$ and $m$ are clear from context, we may say that ``$\pi$ is measure preserving'' instead of ``$\pi$ preserves $\mu$''.

\begin{lemma}\label{lem:MeasureIsPreserved} Let $G$ be an abelian group and $\mu$ a translation invariant probability measure on $G$.   If $H$ is a compact abelian group with Haar probability measure $m$ and $\pi:G\to H$ is a surjective $\mu$-measurable homomorphism, then $\pi$ preserves $\mu$.  Furthermore, if $f,g\in L^{2}(m)$ then $(f\circ \pi)*_{\mu}(g\circ \pi) = (f*_{m}g)\circ \pi$.
\end{lemma}

\begin{proof}
	Define a measure $\eta$ on the $\sigma$-algebra of Borel subsets of $H$ by $\eta(A):=\mu(\pi^{-1}(A))$.  It is easy to check that $\eta$ is a translation invariant probability measure, so uniqueness of Haar measure implies $\eta(A)=m(A)$ for all $m$-measurable subsets of $H$.  Thus $\pi$ preserves $\mu$.  The last assertion is straightforward to verify from the definition of convolution and the fact that $\pi$ preserves $\mu$.
\end{proof}

Our results relate subsets of a given compact abelian group $G$ to subsets of quotients of $G$ which are isomorphic to $\mathbb T$ or to a finite group; the following definition helps describe such sets.

\begin{definition}[Bohr intervals and cyclic progressions]\label{def:PreInterval}
	Let $\mathbb T:=\mathbb R/\mathbb Z$ with the usual topology.  The elements of $\mathbb T$ are   $a+\mathbb Z$, where $a\in \mathbb R$, so every subset of $\mathbb T$ may be identified with a set $E+\mathbb Z$, where $E\subseteq [0,1]\subseteq \mathbb R$.  An \emph{interval} in $\mathbb T$ is a set of the form $[a,b]+\mathbb Z$, where $a, b\in \mathbb R$ and  $0<b-a\leq 1$.    We will abuse notation and write $[a,b]$ for the set $[a,b]+\mathbb Z$ contained in $\mathbb T$.  Unless otherwise specified, the intervals we consider are closed, as indicated in the definition.
	
	Let $\lambda$ denote Haar probability measure on $\mathbb T$. Note that if $I$ and $J$ are intervals in $\mathbb T$ then $I+J$ is an interval, and $\lambda(I+J)=\min\{1,\lambda(I)+\lambda(J)\}$.

	Let $G$ be an abelian group and $\mu$ a translation invariant probability measure on $G$.  We say that $A\subseteq G$ is a \emph{Bohr interval} if there is a surjective $\mu$-measurable homomorphism $\tau:G\to \mathbb T$ and an interval $I\subseteq \mathbb T$ such that $A=\tau^{-1}(I)$.
	
	If $A, B\subseteq G$ are Bohr intervals, we say that $A$ and $B$ are \emph{parallel} if there is a single surjective $\mu$-measurable homomorphism $\tau:G\to \mathbb T$ and intervals $I, J\subseteq \mathbb T$ such that $A=\tau^{-1}(I)$ and $B=\tau^{-1}(J)$.

	If $N\in\mathbb N$, we say that $A\subseteq G$ is an \emph{$N$-cyclic progression}\footnote{This terminology is not standard, but it satisfies our desire to include the cardinality of the group $\mathbb Z/N\mathbb Z$ in the terminology, as a bound on $N$ is stated in Theorem \ref{thm:Precise}.} if there is a surjective $\mu$-measurable homomorphism $\tau:G\to \mathbb Z/N\mathbb Z$ and an arithmetic progression $P\subseteq \mathbb Z/N\mathbb Z$ of common difference $1$ such that $A=\tau^{-1}(P)$.
	
	If $A$ and $B$ are $N$-cyclic progressions, we say that $A$ and $B$ are \emph{parallel} if there is a single surjective $\mu$-measurable homomorphism $\tau:G\to \mathbb Z/N\mathbb Z$ such that $A=\tau^{-1}(P)$ and $B=\tau^{-1}(Q)$ for arithmetic progressions $P, Q\subseteq \mathbb Z/N\mathbb Z$ of common difference $1$.

	If $K\leq G$ is a  $\mu$-measurable finite index subgroup of $G$, $C'\subseteq K$ is a Bohr interval in $K$, and $a\in G$, we say that $a+C'$ is a \emph{relative Bohr interval} in $G$.  A \emph{relative $N$-cyclic progression}
	is defined similarly.
\end{definition}

\begin{remark}  An $N$-cyclic progression is not an arithmetic progression in $G$, unless $G$ itself is a cyclic group of order $N$.
\end{remark}

\begin{notation}\label{not:CN} We write $C_{N}$ for the subgroup $\{\frac{j}{N}:j\in \mathbb Z\}\subseteq \mathbb T$, so that $C_{N}$ is isomorphic to $\mathbb Z/N\mathbb Z$.  If $\tau:G\to \mathbb T$ is a $\mu$-measurable homomorphism such that $\tau(G)=C_{N}$ and $I\subseteq \mathbb T$ is an interval, then $\tau^{-1}(I)$ is an $N$-cyclic progression.  The $N$-cyclic progressions we encounter in our proofs will have this form, and the next lemma estimates their measure.\end{notation}

\begin{lemma}\label{lem:CyclicMeetInterval}
	If $I\subseteq \mathbb T$ is an interval, then $N\lambda(I)-1\leq |C_{N}\cap I|\leq N\lambda(I)+1$.  Consequently, if $\tau:G\to C_{N}$ is a surjective $\mu$-measurable homomorphism, then $\tilde{I}:=\tau^{-1}(I)$ is an $N$-cyclic progression satisfying $\lambda(I)-\frac{1}{N}\leq \mu(\tilde{I})\leq \lambda(I)+\frac{1}{N}$.
\end{lemma}

\begin{proof}
	Haar measure on $C_{N}$ is normalized counting measure, so by Lemma \ref{lem:MeasureIsPreserved} it suffices to estimate the cardinality of $I\cap C_{N}$. To do this we identify $I\cap C_{N}$ with a set of the form $[a,b]\cap \{\frac{j}{N}: j\in \mathbb Z\}$ contained in $\mathbb R$, where $b-a=\lambda(I)$.  The last intersection has the same cardinality as $[Na,Nb]\cap \mathbb Z$, so we use the estimate $d-c-1 \leq |[c,d]\cap \mathbb Z|\leq d-c+1$, which holds for all $c\leq d\in \mathbb R$.  Thus $N(b-a)-1\leq|I\cap C_{N}|\leq N(b-a)+1$, which is equivalent to the estimate stated in the lemma.
\end{proof}

\begin{remark} If $A$ and $B$ are parallel Bohr intervals such that $A=~\tau^{-1}(I)$ and $B=\tau^{-1}(J)$ where $I, J$ are intervals in $\mathbb T$, then Lemma \ref{lem:MeasureIsPreserved} implies $\mu(A)=\lambda(I)$, $\mu(B)=\lambda(J)$, and $\mu(A+B)=\min\{1,\mu(A)+\mu(B)\}$.  Similarly, if $A$ and $B$ are parallel $N$-cyclic progressions then $\mu(A+B)=\min\{1,\mu(A)+\mu(B)-\frac{1}{N}\}$.
\end{remark}

\begin{remark}
	Pairs of parallel Bohr intervals are called ``parallel Bohr sets'' in  \cite{TaoInverse}.  These form a special case of the \emph{Sturmian pairs} considered in \cite{BjorklundCompact,BFKneser}, which deal with the abelian and nonabelian settings.
\end{remark}

\begin{example}[Bohr intervals]
	Let $I\subseteq \mathbb T$ be the interval corresponding to $[0,\frac{1}{3}]$ and define $\tau:\mathbb T\to \mathbb T$ by $\tau(x)=2x$.  Then $\tilde{I}:=\tau^{-1}(I)$ is a Bohr interval.  Here $\tilde{I}$ corresponds to $[0,\frac{1}{6}]\cup [\frac{1}{2}, \frac{2}{3}]$, so a Bohr interval in $\mathbb T$ is not necessarily an interval.

	Let $I\subseteq \mathbb T$ be the interval corresponding to $[0,\frac{1}{3}]$ and  define a homomorphism $\tau:\mathbb T\times \mathbb T\to \mathbb T$ by $\tau(x,y)=x-2y$.  Then  $\tilde{I}:=\tau^{-1}(I)=\{(x,y)\in \mathbb T\times \mathbb T: x-2y\in I\}$ is a Bohr interval in $\mathbb T\times \mathbb T$.
\end{example}

\begin{example}[$N$-cyclic progressions]  If $N, a\in \mathbb N$, with $a\leq N$, then the residues $\{0, 1, \dots, a-1\}$ in $\mathbb Z/N\mathbb Z$ form an $N$-cyclic progression.  If $p$ is prime and  $A\subseteq \mathbb Z/p\mathbb Z$ is an arithmetic progression, i.e.~a set of the form $\{a, a+d, a+2d,\dots, a+kd\}$, where $a, d\in \mathbb Z/p\mathbb Z$, $d\neq 0$, $k\in \mathbb N$, then $A$ is a $p$-cyclic progression.
\end{example}

For the following definitions, fix an abelian group $G$ with a translation invariant probability measure $\mu$, and let $\mu_{*}$ be the associated inner measure.

\begin{definition}[Periodicity]\label{def:Periodic} Given a subgroup $K\leq G$, a set $A\subseteq G$ is a union of cosets of $K$ if and only if $A=A+K$.  The latter equation will  abbreviate the assertion ``$A$ is a union of cosets of $K$''.   We may call such a set $A$ \emph{periodic with respect to} $K$.
	
	If $A\sim_{\mu} A+K$, we say that $A$ is \emph{essentially periodic with respect to} $K$.  When $A$ is essentially periodic with respect to $K$, then $A$ is $\mu$-similar to a union of cosets of $K$, but the converse does not hold. For example, if $G$ is an infinite compact abelian group with Haar probability measure $\mu$, $K\leq G$ is a proper open subgroup, $c\notin K$, and $A=K\cup \{c\}$, then $A\sim_{\mu} K$ but $A\nsim_{\mu} A+K$.
\end{definition}

The following definition describes sets which are close to being periodic.  It is stronger than merely insisting that a given set has small symmetric difference with a periodic set.

\begin{definition}\label{def:EpsPeriodic}
	If $K\leq G$ is a $\mu$-measurable subgroup with $\mu(K)>0$ and $\varepsilon\geq 0$, a set $A\subseteq G$ is \emph{$\varepsilon$-periodic with respect to $K$} if $\mu(A+K)-\mu_{*}(A)\leq \varepsilon\mu(K)$.
\end{definition}

The term ``quasi-periodic'' is introduced in \cite{GrynkKemperman} and \cite{GrynkStep} to describe pairs of sets satisfying $|A+B|\leq |A|+|B|$.  Here is the natural generalization of the term to a group with a translation invariant measure.

\begin{definition}[Quasi-periodicity]\label{def:QP}
	Let $K\leq G$ be a $\mu$-measurable subgroup with $\mu(K)>0$.  We say that $A\subseteq G$ is \emph{quasi-periodic} with respect to $K$ if $A$ can be partitioned into two sets $A_{1}$ and $A_{0}$, where $A_{0}\neq \varnothing$, $(A_{1}+K)\cap A_{0}=\varnothing$, $A_{1}\sim_{\mu} A_{1}+K$, and $A_{0}$ is contained in a coset of $K$.  For such $A_{1}, A_{0}$,  the expression $A_{1}\cup A_{0}$ is called a \emph{quasi-periodic decomposition of $A$ with respect to $K$}, or simply a \emph{quasi-periodic decomposition}, if the group $K$ is clear from context.
	
	If $\varepsilon>0$, we say that $A\subseteq G$ is \emph{$\varepsilon$-quasi periodic with respect to $K$} if $A=A_{1}\cup A_{0}$, where $A_{0}\neq \varnothing$, $(A_{1}+K)\cap A_{0}=\varnothing$, $A_{1}$ is $\varepsilon$-periodic with respect to $K$, and $A_{0}$ is contained in a coset of $K$.
\end{definition}

We allow $A_{1}$ to be empty in the definitions of ``quasi-periodic'' and ``$\varepsilon$-quasi-periodic'', but we do not allow $A_{0}$ to be empty.

If we say that ``$A=A_{1}\cup A_{0}$ is a quasi-periodic (or $\varepsilon$-quasi-periodic) decomposition of $A$ with respect to $K$'', then $A_{1}$ and $A_{0}$ satisfy the conditions listed in Definition \ref{def:QP}.

For a compact abelian group $G$ with Haar measure $m$, Theorem \ref{thm:Precise}  provides detailed structural information about pairs of sets $A, B\subseteq G$ when $\delta$ is very small and $m_{*}(A+B)\leq m(A)+m(B)+\delta$.  Our methods do not provide a specific $\delta$ for which our conclusions hold.  Note that an $m$-measurable subgroup $K\leq G$ has finite index if and only if $K$ is compact and open, if and only if $m(K)>0$.

\begin{theorem}\label{thm:Precise}
	For all $\varepsilon >0$, there exists $\delta>0$ and $d\in \mathbb N$ such that if $G$ is a compact abelian group with Haar probability measure $m$ and $A, B\subseteq G$ have $m(A)>\varepsilon$, $m(B)>\varepsilon$, and   $m_{*}(A+B)\leq m(A)+m(B)+\delta$, then there is a compact open subgroup $K\leq G$ (possibly $K=G$) having index at most $d$ such that exactly one of the following holds.
	\begin{enumerate}
		\item[(I)] $A+B$ is $\varepsilon$-periodic with respect to $K$ and $m(A+B+K) \leq m(A+K)+m(B+K)$.

		\item[(II)]  $A+B$ is not $\varepsilon$-periodic with respect to $K$,
		while $A$ and $B$ have $\varepsilon$-quasi-periodic decompositions $A=A_{1}\cup A_{0}$, $B=B_{1}\cup B_{0}$ with respect to $K$, where at least one of $A_{1}, B_{1}$ is nonempty and \[m_{*}(A_{0}+B_{0})\leq m(A_{0})+m(B_{0})+\varepsilon m(K),\]
		 or

		\item[(III)] $m_{*}(A+B)<(1-\varepsilon)m(K)$, and there are $A', B'\subseteq K$ such that  $m(A')<m(A)+\varepsilon$, $m(B')<m(B)+\varepsilon$ and $a, b\in G$ such that $A\subseteq a+A'$ and $B\subseteq b+B'$, where
		\begin{enumerate}
			\item[(III.1)] $A', B'$ are parallel Bohr intervals in $K$, or
			\item[(III.2)]   $A', B'$ are parallel $N$-cyclic progressions in $K$ for some $N> d$.
	\end{enumerate}  \end{enumerate}
	
\noindent Note: $\delta$ and $d$ are independent of the group $G$.
\end{theorem}

\begin{remark}
	Theorem \ref{thm:Precise} does not assume $m_{*}(A+B)<1-\varepsilon$ or even $m_{*}(A+B)<1$; the possibility $m_{*}(A+B)\approx 1$ is accounted for in conclusion (I), where the group $K$ may be equal to $G$.  See \S\ref{sec:Niveau} for elaboration.
\end{remark}

Connected groups have no proper open subgroups, so the special  case of Theorem \ref{thm:Precise} where $G$ is connected recovers the  following result from \cite{TaoInverse}.

\begin{theorem}[\cite{TaoInverse}, Theorem 1.3]\label{thm:TaoInverse1}
	For all $\varepsilon>0$, there exists $\delta>0$ such that if $G$ is a connected compact abelian group and $A, B\subseteq G$ have $m(A)>\varepsilon$, $m(B)>\varepsilon$,  $m(A)+m(B)<1-\varepsilon$ and $m_{*}(A+B)\leq m(A)+m(B)+\delta$, then there are parallel Bohr intervals $A', B'\subseteq G$ such that $A\subseteq A'$, $B\subseteq B'$, and $m(A'\setminus A)+m(B'\setminus B)< \varepsilon$.
\end{theorem}

The following corollary of Theorem \ref{thm:Precise} is a weak version of Theorem 21.8 of \cite{GrynkBook} and Theorem 1.3 of \cite{SerraZemor}.  We obtain it by specializing Theorem \ref{thm:Precise} to the case where $G$ is a cyclic group of prime order and adding the hypothesis $m(A)+m(B)< 1-\varepsilon$, which rules out conclusion (I).  

\begin{corollary}\label{cor:PrimeOrder}
	For all $\varepsilon>0$, there exists $N\in \mathbb N$, $\delta>0$, such that for every prime $p>N$ and all sets $A, B\subseteq \mathbb Z/p\mathbb Z$ satisfying $|A|,|B|\geq \varepsilon p$, $|A|+|B|<(1-\varepsilon)p$, and $|A+B|\leq |A|+|B|+\delta p$, there are arithmetic progressions $I\supseteq A$, $J\supseteq  B$ in $\mathbb Z/p\mathbb Z$ both having common difference $k\neq 0$ such that $|I|\leq (1+\varepsilon)|A|$ and $|J|\leq (1+\varepsilon) |B|$.
\end{corollary}
Corollary \ref{cor:PrimeOrder} follows immediately from Theorem \ref{thm:Precise}: when $\varepsilon>0$ is fixed, $\delta$ and $d$ are as in the conclusion, and $p>d$, there is no proper subgroup of $\mathbb Z/p\mathbb Z$ having index at most $d$, so only conclusion (III.2) can hold, and the subgroup $K$ therein is equal to $G$.

Theorem \ref{thm:Precise} is new for every infinite disconnected compact abelian group, such as $\mathbb T\times (\mathbb Z/N\mathbb Z)$. The special case of Theorem \ref{thm:Precise} where $G$ is an arbitrary finite abelian group is apparently also new.

\subsection{Acknowledgement}  Two anonymous referees for \href{https://discreteanalysisjournal.com/}{\emph{Discrete Analysis}} contributed many corrections and improvements to this article.

\section{Outline of proofs}\label{sec:ProofOutline}

We prove Theorems \ref{thm:Popular} and \ref{thm:Precise}  in \S\ref{sec:PopularProof} and \S\ref{sec:PreciseProof}, respectively.  The proofs are carried out in three parts, which we outline here. We assume familiarity with ultraproducts and Loeb measure; the reader can refer to Appendix \ref{sec:Ultraproducts} for a synopsis and references.  We now fix a nonprincipal ultrafilter $\mathcal U$ on $\mathbb N$.

\smallskip

\noindent \textbf{Part 1.}   Classify all pairs of $\mu$-measurable sets $A, B\subseteq \mb G$ satisfying $\mu(A)>0$, $\mu(B)>0$, and $\mu_{*}(A+B)\leq \mu(A)+\mu(B)$, where $\mb G$ is an ultraproduct $\prod_{n\to \mathcal U} G_{n}$ of compact abelian groups $G_{n}$ with Haar probability measure $m_{n}$, and $\mu$, $\mu_{*}$ are the corresponding Loeb measure and inner Loeb measure on $\mb G$.  This is done in Propositions \ref{prop:LiftSubcritical} and \ref{prop:Reduction}, whose proofs we outline here.

To prove Propositions \ref{prop:LiftSubcritical} and \ref{prop:Reduction}, we use Corollary \ref{cor:ProjectAndPullBack} to model an arbitrary sumset in $\mb G$ by a sumset in a compact quotient of $\mb G$.  We then apply known inverse theorems in compact abelian groups to these models and transfer the information to the original group $\mb G$. Given $\mu$-measurable sets $A, B\subseteq \mb G$, Corollary \ref{cor:ProjectAndPullBack} provides:
\begin{enumerate}
	\item[$\bullet$]  a compact abelian group $G$ with Haar probability measure $m$,
	
	\item[$\bullet$] a surjective $\mu$-measurable (and measure preserving) homomorphism $\pi:\mb G\to G$, and
	
	\item[$\bullet$] Borel sets $C, D\subseteq G$ such that $C+D$ is Borel,
	\begin{equation}\label{eqn:ModelContainments}
	A\subset_{\mu} A':=\pi^{-1}(C), \quad B\subset_{\mu} B':=\pi^{-1}(D), \quad \text{and} \quad A'+B'\subset_{\mu} A+B.
	\end{equation}
\end{enumerate}
Assuming $A, B\subseteq \mb G$ satisfy $\mu_{*}(A+B)\leq \mu(A)+\mu(B)$, the containments in (\ref{eqn:ModelContainments}) and the fact that $\pi$ preserves $\mu$ imply $m(C+D)\leq m(C)+m(D)$. This inequality allows us to apply Theorems \ref{thm:Satz1} and \ref{thm:LCAInverse} to classify $C$ and $D$.  Some ad hoc arguments then lead to the desired classification of $A$ and $B$, based on the structure of $C$ and $D$ and the  hypothesis $\mu_{*}(A+B)\leq \mu(A)+\mu(B)$.  Some of these arguments are carried out in \S \ref{sec:ReducibleGeneral}, which generalizes some  lemmas from \cite{GriesmerLCA} while providing more efficient proofs.

\smallskip

\noindent \textbf{Part 2.} Prove that if an \emph{internal} set $\mb A=\prod_{n\to \mathcal U} A_{n}$ in $\mb G$ is highly structured (i.e. $\mb A$ is quasi-periodic, or a Bohr interval, etc.), then the constituent sets $A_{n}$ are themselves highly structured.  For example, if $\mb A$ is a $\mu$-measurable internal finite index subgroup of $\mb G$, we will see that for $\mathcal U$-many $n$, the set $A_{n}$ is also a finite index subgroup of $G_{n}$ with the same index as $\mb A$.  This is done in \S\ref{sec:IntervalsInG} and \S\ref{sec:FiniteIndex}.

Parts 1 and 2 both rely on Lemma \ref{lem:InternalCharacter} and Proposition \ref{prop:SzegedyQuotient}, two of the main results of \cite{SzegedyLimits}.  Lemma \ref{lem:InternalCharacter} characterizes the $\mu$-measurable homomorphisms $\chi:\mb G\to \mathcal S^{1}$, and Proposition \ref{prop:SzegedyQuotient} exploits this characterization to find compact quotients of $\mb G$ which are useful for studying convolutions and sumsets.

\smallskip

\noindent \textbf{Part 3.}  Prove Theorem \ref{thm:Precise}  by contradiction, following the strategy of \cite{TaoInverse}.   Assuming Theorem \ref{thm:Precise} is false, then for some $\varepsilon>0$ and each $n\in \mathbb N$ there is a compact abelian group $G_{n}$ with Haar probability measure $m_{n}$ and sets $A_{n}, B_{n}\subseteq G_{n}$ with $m_{n}(A_{n}), m_{n}(B_{n})>\varepsilon$, while $A_{n}+B_{n}$ has inner Haar measure at most $m_{n}(A_{n})+m_{n}(B_{n})+\frac{1}{n}$, and none of the conclusions (I)-(III) in Theorem \ref{thm:Precise} are satisfied by $A_{n}, B_{n}$, and a compact open subgroup $K_{n}\leq G_{n}$ of index at most $n$.  We consider the ultraproduct $\mb G=\prod_{n\to\mathcal U} G_{n}$ with Loeb measure $\mu$ based on $m_{n}$ and the internal sets $\mb A:=\prod_{n\to \mathcal U} A_{n}$ and $\mb B:=\prod_{n\to \mathcal U}B_{n}\subseteq \mb G$.  The definition of Loeb measure implies $\mu(\mb A), \mu(\mb B)\geq \varepsilon$, and $\mu_{*}(\mb A+\mb B)\leq \mu(\mb A)+\mu(\mb B)$, so  we may appeal to Propositions \ref{prop:LiftSubcritical} and \ref{prop:Reduction} to identify the structure of $\mb A$ and $\mb B$. The results of \S\ref{sec:IntervalsInG} and \S\ref{sec:FiniteIndex} then derive the desired structure of $A_{n}$ and $B_{n}$.  This produces a fixed $d\in \mathbb N$ and compact open subgroups $K_{n}\leq G_{n}$ of index at most $d$ such that $A_{n}, B_{n}$, and $K_{n}$ satisfy the conclusion of Theorem \ref{thm:Precise} for $\mathcal U$-many $n$, contradicting the assumption to the contrary.

In \S\ref{sec:ConnectedProof} we provide a separate proof of Theorem \ref{thm:TaoInverse1},  where the main ideas  of our appoach appear and connectedness prevents the profusion of cases seen in the general setting.

The proof of Theorem \ref{thm:Popular} follows an outline similar to the proof of Theorem \ref{thm:Precise}: we assume, to get a contradiction, that for some $\varepsilon>0$ and each $n\in \mathbb N$ there are  $A_{n}, B_{n}\subseteq G_{n}$ such that $m_{n}(A_{n}), m_{n}(B_{n})>\varepsilon$ and $m_{n}(A_{n}+_{\frac{1}{n}} B_{n})<m_{n}(A_{n})+m_{n}(B_{n})+\frac{1}{n}$, while $A_{n}, B_{n}$ do not satisfy the conclusion of Theorem \ref{thm:Precise}.  In the ultraproduct $\mb G$, this produces a pair $\mb A, \mb B$ where $\mu(\mb A+_{0}\mb B)\leq \mu(\mb A)+\mu(\mb B)$. Proposition \ref{prop:Kronecker}, Lemma \ref{lem:ConvolutionToSumset}, and some ad hoc arguments then provide sets $A'$ and $B'$ such that $A'\sim_{\mu} \mb A$ and $B'\sim_{\mu} \mb B$ while $\mu(A'+B')\leq \mu(A')+\mu(B')$.  Lemma \ref{lem:FirstHalf}, a consequence of Propositions \ref{prop:LiftSubcritical} and \ref{prop:Reduction}, then provides sets $S_n$, $T_n\subseteq G_{n}$ having $m_{n}(S_n+T_n)\leq m_{n}(S_n)+m_{n}(T_n)$ and $m_{n}(A_{n}\triangle S_n)+m_{n}(B_{n}\triangle T_n)<\varepsilon$, contradicting our assumption to the contrary.

\section{Summary of known inverse theorems}\label{sec:LCAreview}

We summarize some known inverse theorems for sumsets in a locally compact abelian (LCA) group $G$.  Our proofs proceed by transferring these results to ultraproducts of compact abelian groups.

\begin{definition}\label{def:Stabilizer}
	Let $G$ be an abelian group and $A\subseteq G$.  The group $H(A):=\{g\in G: A+g=A\}$ is called the \emph{stabilizer} of $A$.
\end{definition}
Note that $A+H(A)=A$, and $H(A)$ is the largest set $S$ such that $A+S=A$.  Consequently, if $H(A)$ is open, then $A$ is open as well.

The next three results are due to M.~Kneser.

\begin{theorem}[\cite{Kneser56}, Satz 1]\label{thm:Satz1}
	Let $G$ be a locally compact abelian group with Haar measure $m$.  If $A, B\subseteq G$ are $m$-measurable sets satisfing $m_{*}(A+B)<m(A)+m(B)$   then $H:=H(A+B)$ is compact and open (so $A+B$ is clopen and therefore measurable), and
	\begin{equation}\label{eqn:Kneser}
	m(A+B) = m(A+H)+m(B+H)-m(H).
	\end{equation}
\end{theorem}
Equation (\ref{eqn:Kneser}) plays a crucial role in our arguments; some of its consequences are developed in \S \ref{sec:ReducibleGeneral}.

If $G$ is a connected LCA group then $G$ has no proper open subgroup, so the next corollary follows from Theorem \ref{thm:Satz1}.

\begin{corollary}\label{cor:ConnectedExpansion}
	If $G$ is a connected LCA group with Haar measure $m$ and $A, B\subseteq  G$ are $m$-measurable sets, then $m_{*}(A+B)\geq \min\{m(G), m(A)+m(B)\}$.
\end{corollary}

When $G$ is compact, the following theorem classifies the pairs of $m$-nonnull sets where equality holds in Corollary \ref{cor:ConnectedExpansion}.  We use terminology from Definition \ref{def:PreInterval} to state it.

\begin{theorem}[\cite{Kneser56}, Satz 2]\label{thm:Satz2}
	If $G$ is a connected compact abelian group with Haar measure $m$ and  $A, B\subseteq G$  are $m$-measurable sets satisfying $m(A), m(B)>0$ and
	\[m_{*}(A+B)=m(A)+m(B)<1,\]  then there are parallel Bohr intervals $A', B'\subseteq G$ with $m(A')=m(A)$, $m(B')=m(B)$, and  $A\subseteq A'$, $B\subseteq B'$.
\end{theorem}

The next result extends Theorem \ref{thm:Satz2} to disconnected groups. It is a consequence of Theorem 1.3 in \cite{GriesmerLCA}, classifying the pairs of subsets $A, B$ of a compact abelian group where $m_{*}(A+B)=m(A)+m(B)$.  We could apply those results from \cite{GriesmerLCA} directly, but we find it easier to use the following theorem, which presents the same result in a more convenient organization. Its derivation from the results of \cite{GriesmerLCA} is explained in Remark \ref{rem:Reformulation}. Here we use terminology and notation from \S\ref{sec:Intro}.

\begin{theorem}\label{thm:LCAInverse}
	Let $G$ be a compact abelian group with Haar probability measure $m$ and let $A, B\subseteq G$ be $m$-measurable sets with $m(A), m(B)>0$ and $m_{*}(A+B)=m(A)+m(B)$.  Then there is a compact open subgroup $K\leq G$ (possibly $K=G$) such that exactly one of the following holds.
	\begin{enumerate}
		\item[(I)]  $A+B\sim_{m} A+B+K$,
		
		\item[(II)]  $A+B\nsim_{m} A+B+K$, and  $A$ and $B$ have quasi-periodic decompositions $A=A_{1}\cup A_{0}$, $B=B_{1}\cup B_{0}$ with respect to $K$, where at least one of $A_{1}$ or $B_{1}$ is nonempty and $m_{*}(A_{0}+B_{0})=m(A_{0})+m(B_{0})$, or
		
		\item[(III)] $A+B\nsim_{m} A+B+K$, and there are parallel Bohr intervals $A', B'\subseteq K$ having $m(A')=m(A)$, $m(B')=m(B)$ and  $a, b\in G$  such that $A\subseteq a+A'$ and $B\subseteq b+B'$.\end{enumerate}
\end{theorem}

While the possibilities (I)-(III) are mutually exclusive for a given subgroup $K$, a pair $A, B$ may satisfy (I) with a given group $K$ and satisfy (II) with a different subgroup $K'$ in place of $K$.

\begin{remark}
	The equation $m_{*}(A_{0}+B_{0})=m(A_{0})+m(B_{0})$ in conclusion (II) leads to Corollary \ref{cor:Iterate}, which seems indispensable to our approach.
\end{remark}

\begin{remark}\label{rem:Reformulation}
	Theorem 1.3 of \cite{GriesmerLCA} has the same hypothesis as Theorem \ref{thm:LCAInverse} here. To see how Theorem \ref{thm:LCAInverse} follows, we list here the four conclusions (P), (E), (K), and (QP) from \cite{GriesmerLCA} and explain how each implies one of (I)-(III) in Theorem \ref{thm:LCAInverse}.
	
	\begin{enumerate}
		\item[(P)]  \textit{There is a compact open subgroup $K\leq G$ with $A\sim_{m}A+K$ and $B\sim_{m}B+K$.}
		
		If $A$ and $B$ satisfy (P) then $A$ and $B$ satisfy (I) in Theorem \ref{thm:LCAInverse}.  To see this, note that if $A\sim_m A+K$, then for every coset $K'$ of $K$ we have $A\cap K'\sim_m K'$ or $A\cap K'=\varnothing$.  Thus $A+B\sim_m A+B+K$, meaning (I) is satisfied in Theorem \ref{thm:LCAInverse}.

		\item[(E)]  \textit{There are measurable sets $A'\supseteq A$ and $B'\supseteq B$ such that}
		\[
		m(A')+m(B') > m(A) + m(B), \textit{ and } m_{*}(A'+B')=m_{*}(A+B).
		\]
		
		In this case $m_{*}(A'+B')<m(A')+m(B')$, so Theorem \ref{thm:Satz1} implies the stabilizer $K:=H(A'+B')$ is compact and open.  The containment $A+B\subseteq A'+B'$ then implies $A+B\sim_{m} A+B+K$, meaning $A, B$, and $K$ satisfy (I) in Theorem \ref{thm:LCAInverse}.

		\item[(K)]  \textit{There is a compact open subgroup $K\leq G$, parallel Bohr intervals $\tilde{I},\tilde{J}\subseteq G$, and $a, b\in G$ such that $A\subseteq a+\tilde{I}$, $B\subseteq b+ \tilde{J}$, while $m(A)=m(\tilde{I})$, $m(B)=m(\tilde{J})$.}
		
		If a pair $A, B$ satisfies (K) and does not satisfy conclusion (I) of Theorem \ref{thm:LCAInverse}, then it is easy to verify that $A, B$ satisfies conclusion (III).

		\item[(QP)] \textit{ There is a compact open subgroup $K\leq G$ and partitions $A=A_{1}\cup A_{0}$, $B=B_{1}\cup B_{0}$ such that $A_{0}\neq \varnothing$, $B_{0}\neq \varnothing$, at least one of $A_{1}\neq \varnothing$, $B_{1}\neq \varnothing$, and
			\begin{enumerate}
				\item[(QP.1)]  $A_{1}+K\sim_{m} A_{1}$, $B_{1}+K\sim_{m} B_{1}$, while $A_{0}$ and $B_{0}$ are each contained in a coset of $K$ and $(A_{1}+K)\cap A_{0}=(B_{1}+K)\cap B_{0}=\varnothing$;
				\item[(QP.2)]  $A_{0}+B_{0}+K$ is a unique expression element of $A+B+K$ in $G/K$;
				\item[(QP.3)] $m_{*}(A_{0}+B_{0})=m(A_{0})+m(B_{0})$.
		\end{enumerate}}
		
		If $A,B$ satisfies (QP) and $A+B\nsim_{m} A+B+K$, then it is easy to check that $A,B$ satisfies conclusion (II) of Theorem \ref{thm:LCAInverse} here.

	\end{enumerate}

\end{remark}

\begin{remark}The pairs of subsets of a discrete abelian group satisfying $|A+B|<|A|+|B|$ are classified in \cite{KempermanAbelian} (see \cite{GrynkKemperman} for exposition).  The pairs satisfying $|A+B|=|A|+|B|$ are classified in \cite{GrynkStep}.  These results can be combined with Theorem \ref{thm:Precise} to yield additional detail, as the inequality $m(A+B+K)\leq m(A+K)+m(B+K)$ in conclusion (I) reduces the study of the pair $A+K, B+K$ to the study of sets $A', B'\subseteq G/K$ satisfying $|A'+B'|\leq |A'|+|B'|$.
\end{remark}

\section{Characters of ultraproduct groups}\label{sec:Characters}
We will use ultraproducts and Loeb measure throughout the remainder of this article; see Appendix \ref{sec:Ultraproducts} for a summary and references.  Here we briefly summarize notation.

If $\mathcal U$ is a nonprincipal ultrafilter on $\mathbb N$ and $(X_{n})_{n\in \mathbb N}$ is a sequence of sets, we write $\prod_{n\to \mathcal U} X_{n}$ for the corresponding ultraproduct, which we denote by $\mathbf X$. Ultraproducts and internal sets will be denoted with boldface letters.  If $Y$ is a compact Hausdorff topological space, $(y_{n})_{n\in \mathbb N}$ is a sequence of elements of $Y$, and $\mathcal U$ is an ultrafilter on $\mathbb N$, we write $\lim_{n\to \mathcal U} y_{n}$ for the unique element $y\in Y$ such that for every neighborhood $V$ of $y$, $\{n\in \mathbb N: y_{n}\in V\}\in \mathcal U$.   For a sequence of functions $f_{n}:X_{n}\to Y$, we denote by $\lim_{n\to \mathcal U} f_{n}$ the function from $\prod_{n\to \mathcal U} X_{n}$ to $Y$ defined by $f(\mathbf x)=\lim_{n\to \mathcal U} f_{n}(x_{n})$, where $(x_{n})_{n\in \mathbb N}$ is a representative of $\mathbf x$.

Note that some authors use ``$\lim_{n\to \mathcal U}$'' for what we call ``$\prod_{n\to \mathcal U}$''.

For this section, fix a sequence of compact abelian groups $G_{n}$ with Haar probability measure $m_{n}$ and a nonprincipal ultrafilter $\mathcal U$ on $\mathbb N$.   Let $\mb G$ be the ultraproduct $\prod_{n\to \mathcal U} G_{n}$ and $\mu$  the Loeb measure corresponding to $(m_{n})_{n\in \mathbb N}$.  As usual $\mathcal S^{1}$ denotes the circle group $\{z\in \mathbb C: |z|=1\}$ with the group operation of multiplication and the usual topology.

We assume familiarity with basic harmonic analysis on compact abelian groups, as introduced in \cite{FollandAbstract} or  \cite{RudinFourier}.  If $G$ is a compact abelian group, $\widehat{G}$ denotes its Pontryagin dual, the group of continuous homomorphisms $\chi:G\to \mathcal S^{1}$ with the discrete topology.

\begin{definition}\label{def:Characters} Call $f:\mb G\to \mathbb C$ a \emph{strong character} of $\mathbf G$ if there are characters $\chi_{n}\in \widehat{G}_{n}$ such that $f=\lim_{n\to \mathcal U} \chi_{n}$. Equivalently, $f$ is a strong character of $\mb G$ if $f=\leftidx^\circ(\prod_{n\to \mathcal U} \chi_{n})$, where $\chi_{n}\in \widehat{G}_{n}$ for $\mathcal U$-many $n$.  Every strong character is $\mu$-measurable by Part (ii) of Proposition \ref{prop:LoebApprox}.  One can easily verify that every strong character is a homomorphism from $\mb G$ to $\mathcal S^{1}$. Write $\widehat{\mb G}$ for the set of strong characters of $\mb G$, and note that $\widehat{\mb G}$ is a group under pointwise multiplication.\footnote{We do not refer to a topology on $\mb G$ when we define $\widehat{\mb G}$, in contrast to the dual $\widehat{G}$ of a locally compact abelian group $G$, which (by definition) grants $\widehat{G}$ the topology of uniform convergence on compact subsets of $G$.}

	A \emph{weak character} of $\mb G$ is a $\mu$-measurable homomorphism $\rho:\mb G\to \mathcal S^{1}$.  Of course every strong character is also weak character, and Lemma \ref{lem:InternalCharacter} says that every weak character of $\mb G$ is a strong character of $\mb G$, motivating our choice of the notation $\widehat{\mb G}$.
\end{definition}

\begin{definition}\label{def:Translate}
	If $G$ is a group, $f$ is a function on $G$, and $t\in G$, we write $f_{t}$ for the \emph{translate} of $f$ by $t$: the function defined by $f_{t}(x):=f(x-t)$.
\end{definition}

\begin{definition}\label{def:FourierAlgebra}Let $\mathcal F(\mathbf G)$ be the uniform closure of the linear span of the strong characters of $\mathbf G$. Note that $\mathcal F(\mb G)$ is closed under pointwise multiplication and complex conjugation, and so is an algebra of functions. Following \cite{SzegedyLimits} we call $\mathcal F(\mathbf G)$ the \emph{Fourier algebra} of $\mathbf G$.
\end{definition}

\begin{definition}\label{def:WM}
	Let $L^{2}(\mu)_{wm}$ be the set of $f\in L^{2}(\mu)$ such that
	\[\int \Bigl|\int f \bar{g}_{\mb t} \,d\mu\Bigr|^{2} \,d\mu(\mb t)=0 \quad \text{for all } g\in L^{2}(\mu).\]  Equivalently, $f\in L^{2}(\mu)_{wm}$ if  $f*g\equiv_{\mu}0$ for all $g\in L^{2}(\mu)$.
\end{definition}
To see the equivalence asserted in the definition, observe that $\int f\bar{g}_{\mb t} \,d\mu = f*h(\mb t)$, where $h(\mb x):=\overline{g(-\mb x)}$, so $f\in L^{2}(\mu)_{wm}$ if and only if $f*g\equiv_{\mu}0$ for all $g\in L^{2}(\mu)$. The subscript ``\emph{wm}'' is to indicate an analogy with weak mixing in dynamics.

The next lemma is  Lemma 6.3 in \cite{SzegedyLimits}.  Readers familiar with \cite{GowersSz} or \cite{GowersQRG} will recognize it as a limiting version of an equivalence between two definitions of uniformity.    Here we write $f\perp \mathcal F(\mathbf G)$ to mean $\int f\bar{g}\, d\mu=0$ for every $g\in \mathcal F(\mb G)$.
\begin{lemma}\label{lem:pseudorandom}
	Let $f:\mb G\to \mathbb C$ be bounded  and $\mu$-measurable.  Then $f\perp \mathcal F(\mathbf G)$ if and only if $f\in L^{2}(\mu)_{wm}$. \end{lemma}

The next lemma is part of Proposition 6.1 in \cite{SzegedyLimits}.

\begin{lemma}\label{lem:InternalCharacter}
	If $\chi:\mb G\to \mathcal S^{1}$ is a $\mu$-measurable homomorphism, then $\chi$ is a strong character of $\mb G$.  Consequently, if $\tau: \mb G\to \mathbb T$ is a $\mu$-measurable homomorphism, then there is a sequence of continuous homomorphisms $\tau_{n}:G_{n}\to \mathbb T$ such that $\tau=\lim_{n\to \mathcal U}\tau_{n}$.
\end{lemma}

The second assertion in Lemma \ref{lem:InternalCharacter} follows from the first as $\mathbb T$ is isomorphic to $\mathcal S^{1}$, and with this isomorphism characters of a compact group $G$ correspond to continuous homomorphisms from $G$ to $\mathbb T$.

We say a character $\chi$ of a group $G$ is \emph{trivial} if $\chi(g)=1$ for all $g\in G$.

\begin{lemma}\label{lem:TrivialCharacter}
	If $\chi_{n}\in \widehat{G}_{n}$ for all $n$, then $\chi:=\lim_{n\to \mathcal U} \chi_{n}$ is trivial if and only if $\chi_{n}$ is trivial for $\mathcal U$-many $n$.
\end{lemma}

\begin{proof}
	If $\chi_{n}$ is trivial for $\mathcal U$-many $n$, then clearly $\chi$ is trivial. To prove the converse, assume $\chi_{n}$ is nontrivial for $\mathcal U$-many $n$.  For these $n$, the image  of $\chi_{n}$ is a nontrivial subgroup of $\mathcal S^{1}$.  Thus there is an $x_{n}\in G_{n}$ such that the real part of $\chi_{n}(x_{n})$ is not positive, meaning $|\chi_{n}(x_{n})-1|\geq \sqrt{2}$.  It follows that $\lim_{n\to \mathcal U} \chi_{n}(x_{n})\neq 1$, so $\chi$ is nontrivial.
\end{proof}

\begin{corollary}\label{cor:IdentifyGhat}
	The group $\widehat{\mb G}$ is isomorphic to the group $\prod_{n\to \mathcal U} \widehat{G}_{n}$.
\end{corollary}
Note: here we are only considering the objects as groups, so the isomorphism is simply a group isomorphism.

\begin{proof}
	By Lemma \ref{lem:InternalCharacter}, a surjective homomorphism $\Phi:  \prod_{n\to \mathcal U} \widehat{G}_{n} \to \widehat{\mb G}$ may be defined by $\Phi(\bm \chi):=\lim_{n\to \mathcal U} \chi_{n}$, where $(\chi_{n})_{n\in \mathbb N}$ is a representative of $\bm \chi$.  Lemma \ref{lem:TrivialCharacter} implies that the kernel of $\Phi$ consists only of a single element, so $\Phi$ is one-to-one.  Thus $\Phi$ is an isomorphism.
\end{proof}

Corollary \ref{cor:IdentifyGhat} allows us to identify the ultraproducts $\mb G$ of compact groups where $\widehat{\mb G}$ is torsion free.

\begin{lemma}\label{lem:GHatTorsion}
	For a fixed $k$, the group $\widehat{\mb G}$ has an element of order $k$ if and only if for $\mathcal U$-many $n$, $\widehat{G}_{n}$ has an element of order $k$.  Consequently, $\widehat{\mb G}$ is torsion free if and only if for all $k\in \mathbb N\setminus \{1\}$ and $\mathcal U$-many $n$, $\widehat{G}_{n}$ does not have an element of order $k$.
\end{lemma}
\begin{proof}
	This follows immediately from the definition of  $\prod_{n\to \mathcal U} \widehat{G}_{n}$ and Corollary \ref{cor:IdentifyGhat}.
\end{proof}

\begin{corollary}\label{cor:FiniteIndexInternal}
	Let $H$ be a finite group and $\tau: \mathbf G\to H$ a surjective $\mu$-measurable homomorphism.  Then there are continuous homomorphisms $\tau_{n}: G_{n} \to H$ such that $\tau = \lim_{n\to \mathcal U} \tau_{n}$.  If $K_{n}=\ker \tau_{n}$, then for $\mathcal U$-many $n$, we have $G_{n}/K_{n} \cong H$.  Furthermore $\ker \tau=\prod_{n\to \mathcal U} K_{n}$.
\end{corollary}

Corollary \ref{cor:FiniteIndexInternal} says that every $\mu$-measurable finite index subgroup of $\mb G$ is internal.

\begin{proof}
	Every finite abelian group is a product of finite cyclic groups, so it suffices to prove the special case of the corollary where $H$ is a finite cyclic group.  In this case we consider $H$ as the subgroup of $\mathcal S^{1}$ generated by $\exp(2\pi i/N)$, where $N$ is the order of $H$.  Note that $H$ is the unique subgroup of $\mathcal S^{1}$ having order $N$, and that if $z\in \mathcal S^{1}$ satisfies $z^{N}=1$, then $z\in H$.  We consider $\tau$ as a homomorphism into $\mathcal S^{1}$ with image $H$.  Lemma \ref{lem:InternalCharacter} then implies $\tau=\lim_{n\to \mathcal U} \chi_{n}$ for some $\chi_{n}\in \widehat{G}_{n}$.  Let $\tau^{N}$ be the homomorphism defined by $\tau^{N}(\mb g):=\tau(\mb g)^{N}$, and similarly define $\chi_{n}^{N}.$  Continuity of the map $z\mapsto z^{N}$ implies $\tau^{N} = \lim_{n\to \mathcal U} \chi_{n}^{N}$, and $\tau^{N}$ is the trivial homomorphism, as the image of $\tau$ is generated by an element of order $N$.  Lemma \ref{lem:TrivialCharacter} then implies $\chi_{n}^{N}$ is trivial for $\mathcal U$-many $n$, meaning $\chi_{n}(G_{n})\subseteq H$ for $\mathcal U$-many $n$. The hypothesis $\tau(\mb G)=H$ then implies $|\tau_{n}(G_{n})|=|H|$ for $\mathcal U$-many $n$, and therefore $\tau_{n}(G_{n})=H$ for $\mathcal U$-many $n$.  Thus $G_{n}/K_{n} \cong H$ for $\mathcal U$-many $n$.

	To see that $\ker \tau=\prod_{n\to \mathcal U} K_{n}$, let $\mb g\in \mb G$, choose a representative $(g_{n})_{n\in \mathbb N}$ of $\mb g$, and note that the finiteness of $H$ implies $\lim_{n\to \mathcal U} \tau_{n}(g_{n})=0$ if and only if $\tau_{n}(g_{n})=0$ for $\mathcal U$-many $n$.
\end{proof}

The next definition provides a useful splitting of $L^{2}(\mu)$.

\begin{definition}\label{def:FC} Let $L^{2}(\mu)_{c}$ denote the closure of $\mathcal F(\mb G)$ in $L^{2}(\mu)$.  If $f\in L^{2}(\mu)$, write $f^{(c)}$ for the orthogonal projection of $f$ onto $L^{2}(\mu)_{c}$, and $f^{(wm)}$ for $f-f^{(c)}$.  Thus $f^{(wm)}\perp \mathcal F(\mb G)$.  Of course $f^{(c)}$ and $f^{(wm)}$ are defined only up to equality on $\mu$-null sets, but for our purposes this ambiguity can be resolved by choosing a genuine function to represent $f^{(c)}$.
\end{definition}

The ``$c$'' in the notation above is to indicate that the translates $\{f_{\mb t}: \mb t\in \mb G\}$ of a bounded function  form a precompact subset of $L^{2}(\mu)$ if and only if $f\in L^{2}(\mu)_{c}$.

The following lemma is a collection of common facts about the orthogonal projection of a function onto a space spanned by an algebra of functions, but they are often not stated explicitly for our setting.

\begin{lemma}\label{lem:ProjectOnF}
	Let $f\in L^{2}(\mu)$.
	\begin{enumerate}
		\item[(i)]  If $f(\mb x)\geq 0$ for $\mu$-almost every $\mb x$, then $f^{(c)}(\mb x)\geq 0$ for $\mu$-almost every $\mb x$.
		\item[(ii)]  If $f:\mb G\to [0,1]$, then $f^{(c)}(\mb x)\in [0,1]$ for $\mu$-almost every $\mb x$.
		\item[(iii)]  $\int f^{(c)} \,d\mu = \int f \,d\mu$.
		\item[(iv)]  If $A\subseteq \mb G$ is $\mu$-measurable, then $A \subset_{\mu} \{\mb x: 1_{A}^{(c)}(\mb x)> 0\}$.
		\item[(v)] If $f:\mb G\to [0,1]$ then $\{\mb x :f(\mb x)>0\}\subset_{\mu}\{\mb x: f^{(c)}(\mb x)>0\}$.
	\end{enumerate}
\end{lemma}

\begin{proof}
	(i)  Suppose $f$ is a nonnegative element of $L^{2}(\mu)$.  We will prove that there is a sequence of nonnegative elements of $\mathcal F(\mb G)$ converging in $L^{2}(\mu)$ to $f^{(c)}$, which implies $f^{(c)}$ is $\mu$-almost everywhere nonnegative.  Let $h_{n}$ be any sequence of elements of $\mathcal F(\mb G)$ converging to $f^{(c)}$ in $L^{2}(\mu)$. Since $\mathcal F(\mb G)$ is a uniformly closed algebra of functions containing the constant functions $0$ and $1_{\mb G}$, we have that for all $h\in \mathcal F(\mb G)$, $h^{+}:=\sup\{h,0\}\in \mathcal F(\mb G)$, by Lemma 8 in Chapter 6 of \cite{Bollobas}.  Since $f$ is nonnegative we have $\|f-h_{n}^{+}\|_{L^{2}(\mu)}\leq \|f-h_{n}\|_{L^{2}(\mu)}$ for each $n$, so $h_{n}^{+}$ converges to $f^{(c)}$, as $f^{(c)}$ is the unique element $h$ in $L^{2}(\mu)_{c}$ which minimizes $\|f-h\|_{L^{2}(\mu)}$.
	
	Part (ii) follows from Part (i) and the nonnegativity of the functions $1_{\mb G}-f$ and $f$.
	
	Part (iii) is a consequence of the fact that $\mathcal F(\mb G)$ contains the constant functions on $\mb G$.
	
	To prove Part (iv), we use the identity $\int f\cdot g^{(c)} \,d\mu=\int f^{(c)}\cdot  g^{(c)} \,d\mu$, which follows immediately from the fact that the map $f\mapsto f^{(c)}$ is an orthogonal projection. Let $E=A\setminus \{\mb x: 1_{A}^{(c)}(\mb x)> 0\}$.  Then $1_{E}1_{A}^{(c)}=0$, so that $\int 1_{E}^{(c)}1_{A}^{(c)} \,d\mu=\int 1_{E}1_{A}^{(c)} \,d\mu=0$.  Now $0\leq 1_{E}\leq1_{A}\leq 1_{\mb G}$, so Part (i) implies $0\leq 1_{E}^{(c)}\leq 1_{A}^{(c)}$.  Combining this last inequality with the integral $\int  1_{E}^{(c)}1_{A}^{(c)} \,d\mu=0$, we get that $\int 1_{E}^{(c)} \,d\mu=0$, which by Part (iii) implies $\mu(E)=0$.
	
	Part (v) follows from Part (iv), the linearity of the map $f\mapsto f^{(c)}$, and bounding $f$ below by linear combinations of characteristic  functions supported on $\{x:f(x)>0\}$.
\end{proof}

\begin{lemma}\label{lem:fcgc}
	If $f, g:\mb G\to \mathbb C$ are bounded $\mu$-measurable functions, then $f*g\equiv_{\mu} f^{(c)}*g^{(c)}$.
\end{lemma}

\begin{proof}
	Write $f$ as $f^{(c)}+f^{(wm)}$.  Lemma \ref{lem:ProjectOnF} implies $f^{(c)}$ is bounded, so Lemma \ref{lem:pseudorandom} implies  $f^{(wm)}$ is a bounded element of $L^{2}(\mu)_{wm}$, and we get that $f^{(wm)}*g\equiv_{\mu}0$ by Definition \ref{def:WM}. Thus $f*g\equiv_{\mu} f^{(c)}*g$, and by symmetry we get that $f^{(c)}*g\equiv_{\mu}f^{(c)}*g^{(c)}$.
\end{proof}

The conclusion of Lemma \ref{lem:fcgc} cannot be improved to assert $f*g=f^{(c)}*g^{(c)}$. See Remark \ref{rem:AEnotE} for elaboration.

The following proposition is part of Theorem 6 of \cite{SzegedyLimits}, cosmetically modified to deal with two functions on $\mb G$ instead of one.

\begin{proposition}\label{prop:SzegedyQuotient}
	If $f, g\in L^{2}(\mu)_{c}$ are bounded functions, then there is a compact metrizable abelian group $G$, a surjective measure preserving quotient map $\pi:\mb G\to G$, and bounded Borel functions $\tilde{f}, \tilde{g}$ on $G$ such that $f\equiv_{\mu} \tilde{f}\circ \pi$ and $g\equiv_{\mu}\tilde{g}\circ \pi$.
\end{proposition}

One can prove Proposition \ref{prop:SzegedyQuotient} by imitating the proof in \cite{SzegedyLimits}, with the following minor modification: where that proof considers the set of characters $\chi$ appearing in the Fourier series of $f$, to prove Proposition \ref{prop:SzegedyQuotient} consider instead the set of characters appearing in the Fourier series of $f$ or the Fourier series of $g$ (or both).

\section{Reducing to a compact quotient of \texorpdfstring{$\mb G$}{G}}\label{sec:ToCompact}
Here we collect some results allowing us to study sumsets in an ultraproduct $\mb G$ of compact groups by studying related sumsets in a compact quotient of $\mb G$.

For this section, fix a sequence of compact abelian groups $G_{n}$ with Haar probability measure $m_{n}$ and a nonprincipal ultrafilter $\mathcal U$ on $\mathbb N$.  Let $\mb G$ denote the ultraproduct $\prod_{n\to \mathcal U} G_{n}$ and $\mu$ be the Loeb measure corresponding to $(m_{n})_{n\in \mathbb N}$. We continue to use the notation $\widehat{\mb G}$ of Definition \ref{def:Characters}: the elements of $\widehat{\mb G}$ are the functions $\chi:\mb G\to \mathcal S^{1}$ of the form $\chi=\lim_{n\to \mathcal U} \chi_{n}$, where $\chi_{n}\in \widehat{G}_{n}$ for each $n$; in particular the elements of $\widehat{\mb G}$ are $\mu$-measurable.

The main results of this section are Proposition \ref{prop:Kronecker} and its consequences, Lemma \ref{lem:ProjectAndPullBack} and Corollary \ref{cor:ProjectAndPullBack}.  Here we abbreviate the set $\{x: f(x)>0\}$ as $\{f>0\}$.

\begin{proposition}\label{prop:Kronecker}
	Let $f, g:\mb G\to [0,1]$ be $\mu$-measurable.  Then there is a compact metrizable quotient $G$ of $\mb G$ with Haar measure $m$, $\mu$-measure preserving quotient map $\pi:\mb G\to G$, and Borel functions $\tilde f, \tilde g:G\to[0,1]$ such that
	\begin{align}
	\label{eqn:TildeLevels} &\{f>0\} \subset_{\mu} \{\tilde{f}\circ \pi>0\}, \quad \{g>0\} \subset_{\mu} \{\tilde{g}\circ \pi>0\}, and\\
	\label{eqn:EquivalentConvolutions} &f*_{\mu}g\equiv_{\mu}(\tilde{f}*_{m }\tilde{g})\circ \pi.
	\end{align}
\end{proposition}

\begin{proof}
	Let $f, g:\mb G\to [0,1]$ be $\mu$-measurable functions. Lemma \ref{lem:ProjectOnF} implies $f^{(c)}$ and $g^{(c)}$ are bounded, so Proposition \ref{prop:SzegedyQuotient} provides a compact metrizable group $G$, a $\mu$-measure preserving quotient map $\pi:\mb G\to G$, and Borel functions $\tilde{f},\tilde{g}:G\to \mathbb C$ such that
	\begin{equation}\label{eqn:tildes}
	f^{(c)}\equiv_{\mu} \tilde{f}\circ \pi, \quad g^{(c)}\equiv_{\mu} \tilde{g}\circ \pi.
	\end{equation}We will show that $\tilde{f}$ and $\tilde{g}$ satisfy the conclusion of Proposition \ref{prop:Kronecker}.
	
	Lemma \ref{lem:ProjectOnF} implies $0\leq f^{(c)}(\mb x), g^{(c)}(\mb x)\leq 1 $ for $\mu$-a.e.~$\mb x$, so $0\leq \tilde{f}(x), \tilde{g}(x)\leq 1$ for $m$-a.e.~$x$. Lemma \ref{lem:ProjectOnF} also implies $\{f>0\}\subset_{\mu} \{f^{(c)}>0\}$ and $\{g>0\}\subset_{\mu} \{g^{(c)}>0\}$. Since $f^{(c)}\equiv_{\mu} \tilde{f}\circ \pi$, we get the essential containments stated in (\ref{eqn:TildeLevels}).
	
	To prove Equation (\ref{eqn:EquivalentConvolutions}),  first note that Lemma \ref{lem:fcgc} implies $f*_{\mu}g\equiv_{\mu} f^{(c)}*_{\mu}g^{(c)}$.  It therefore suffices to show that $f^{(c)}*_{\mu}g^{(c)}\equiv_{\mu} (\tilde{f}*_{m}\tilde{g})\circ \pi$, and this follows immediately from Lemma \ref{lem:MeasureIsPreserved}.
\end{proof}

\begin{remark}\label{rem:AEnotE}
	To be clear, Proposition \ref{prop:Kronecker} asserts that $(f*g)(\mb x) = (\tilde{f}*\tilde{g})(\pi(\mb x))$ for $\mu$-almost every $\mb x\in \mb G$.   This cannot be improved to replace ``$\mu$-almost every $\mb x$'' with ``every $\mb x$''. Examples exhibiting the obstruction are discussed in \S\ref{sec:Exceptional}.
\end{remark}

The next lemma lets us pass between convolutions and sumsets; the main application will take an estimate on the size of a sumset in $\mb G$ and obtain a similar estimate on the size of a related sumset in a compact quotient of $\mb G$.

\begin{lemma}\label{lem:ConvolutionToSumset}
	Let $G$ be a compact metrizable abelian group with Haar probability measure $m$.  If $f, g: G\to [0,1]$ are $m$-measurable functions let $C_{0}:=\{x\in G: f(x)>0\}$ and $D_{0}:=\{x\in G:g(x)>0\}$. Then there are Borel sets $C\subseteq C_{0}$, $D\subseteq D_{0}$ such that $m(C)=m(C_{0})$, $m(D)=m(D_{0})$, $C+D$ is Borel, and $C+D\subseteq \{x\in G:f*g(x)>0\}$.
\end{lemma}

For a proof of Lemma \ref{lem:ConvolutionToSumset}, see Lemma 2.13 of \cite{GriesmerUBD}.  When $G=\mathbb T^{d}$ for some $d\in \mathbb N$, we can take $C$ and $D$ to be countable unions of compact subsets of  the points of Lebesgue density of $\{x:f(x)>0\}$ and $\{x:g(x)>0\}$, respectively.

\begin{lemma}\label{lem:ProjectAndPullBack}
	Let $f, g:\mb G\to [0,1]$ be $\mu$-measurable functions and $E:=\{\mb x\in \mb G: f*_{\mu}g(\mb x)>0\}$.  If $A:=\{\mb x\in \mb G: f(\mb x)>0\}$ and $B:=\{\mb x\in \mb G: g(\mb x)>0\}$,  there is a compact quotient group $G$ with Haar probability measure $m$, a $\mu$-measure preserving quotient map $\pi:\mb G\to G$, and Borel sets  $C, D\subseteq G$ such that $A\subset_{\mu} A':=\pi^{-1}(C)$, $B\subset_{\mu} B':=\pi^{-1}(D)$ and $A'+B'\subset_{\mu} E$.  Furthermore $C+D$ is Borel and $m(C+D)\leq \mu(E)$.
	
\end{lemma}
\begin{proof}
	Let $G$, $\pi:\mb G\to G$, and $\tilde{f}, \tilde{g}:G\to [0,1]$ be as in Proposition \ref{prop:Kronecker}.   Consider the level sets $C_{0}:=\{t\in G:\tilde{f}(t)>0\}$, $D_{0}:=\{t\in G:\tilde{g}(t)>0\}$, and apply Lemma \ref{lem:ConvolutionToSumset} to choose Borel sets $C\subseteq C_{0}$, $D\subseteq D_{0}$ having $C\sim_{m} C_{0}$, $D\sim_{m} D_{0}$ such that $C+D$ is Borel and
	\begin{equation}\label{eqn:CDinV}
	C+D\subseteq V:=\{t\in G:\tilde{f}*_{m}\tilde{g}(t)>0\}.
	\end{equation}  Let $A':= \pi^{-1}(C)$ and $B':=\pi^{-1}(D)$.  Now Proposition \ref{prop:Kronecker} implies $(\tilde{f}*_{m}\tilde{g})\circ \pi \equiv_{\mu} f*_{\mu}g$, so $E\sim_{\mu}\pi^{-1}(V)$, and (\ref{eqn:CDinV}) implies $A'+B'\subset_{\mu} E$.
	
	The essential containment $A\subset_{\mu} A'$ follows from the similarity $A'\sim_{\mu}\{\mb x\in \mb G: \tilde{f}\circ \pi(\mb x)>0\}$ and Proposition \ref{prop:Kronecker}.  Likewise $B\subset_{\mu} B'$. The inequality $m(C+D)\leq \mu(E)$ follows from the containments $\pi^{-1}(C+D)=A'+B'\subset_{\mu} E$, since $\pi$ preserves $\mu$.
\end{proof}

Specializing Lemma \ref{lem:ProjectAndPullBack} to the case where $f=1_{A}$ and $g=1_{B}$ are characteristic functions of sets and observing that $A+_{0}B:=\{\mb x\in \mb G: 1_{A}*1_{B}(\mb x)>0\}\subseteq A+B$, we obtain the following corollary.

\begin{corollary}\label{cor:ProjectAndPullBack}
	If $A, B\subseteq \mb G$ are $\mu$-measurable sets then there is a compact quotient $G$ with Haar measure $m$, a $\mu$-measure preserving quotient map $\pi:\mb G\to G$, and Borel sets $C, D\subseteq G$ such that
	
	\begin{enumerate}
		\item[(i)]
		$A\subset_{\mu} A':=\pi^{-1}(C)$, $B\subset_{\mu} B':=\pi^{-1}(D)$ and $A'+B'\subset_{\mu} A+_{0}B \subseteq A+B$.   Furthermore $C+D$ is Borel (so that $A'+B'$ is $\mu$-measurable) and $m(C+D)\leq \mu_{*}(A+B)$.
		\item[(ii)] If $A''\sim_{\mu}A$ and $B''\sim_{\mu} B$, then $A'+B'\subset_{\mu} A''+B''$.
	\end{enumerate}
\end{corollary}

\begin{proof}
	Part (i) follows immediately from Lemma \ref{lem:ProjectAndPullBack}.  Part (ii) follows from Part (i) and  the observation that $A''+_{0}B''=A+_{0}B$.
\end{proof}

\section{Bohr intervals and sumsets}\label{sec:IntervalsInG}

For this section fix a sequence of compact abelian groups $G_{n}$ with Haar probability measure $m_{n}$ and a nonprincipal ultrafilter $\mathcal U$ on $\mathbb N$.  Let $\mb G$ be the ultraproduct $\prod_{n\to \mathcal U} G_{n}$ and $\mu$ the corresponding Loeb measure. Write $\widehat{\mb G}$ for the group of strong characters of $\mb G$, as defined in \S\ref{sec:Characters}. Lemma \ref{lem:InternalIntervals} and Corollary \ref{cor:PullPreintervalsDown}  characterize the internal Bohr intervals in $\mb G$ in terms of their constituent sets.  Lemma \ref{lem:MoveInterval} and Corollary \ref{cor:EssentialMoveInterval} show that certain sets related to Bohr intervals are themselves Bohr intervals.

Recall (Definition \ref{def:PreInterval}) that if $G$ is a group and $\mu$ is a translation invariant probability measure on $G$, a Bohr interval in $G$ is a set of the form $\tau^{-1}(I)$, where $\tau:G\to \mathbb T$ is a surjective $\mu$-measurable homomorphism and $I\subseteq \mathbb T$ is an interval.  Furthermore, we consider only closed intervals in $\mathbb T$.

Here we characterize those sequences of sets $A_{n}\subseteq G_{n}$ such that $\prod_{n\to \mathcal U} A_{n}$ is a Bohr interval in $\mb G$; we will see that this forces the $A_{n}$ to be contained in Bohr intervals or $N$-cyclic progressions (Definition \ref{def:PreInterval}) not much larger than $A_{n}$. We use the notation $C_{N}$ (Notation \ref{not:CN}) and continue to write $\lambda$ for Lebesgue measure ($=$ normalized Haar measure) on $\mathbb T$.

\begin{lemma}\label{lem:InternalIntervals}
	For each $n\in \mathbb N$, let $\tau_{n}:G_{n}\to \mathbb T$ be a continuous homomorphism, and define $\tau:\mb G\to \mathbb T$ as $\lim_{n\to \mathcal U} \tau_{n}$. Assume that $\tau(\mb G)=\mathbb T$.  Then
	
	\begin{enumerate}
		\item[(i)] Either
		\begin{enumerate}
			\item[$\cdot$]$ \tau_{n}(G_{n})=\mathbb T$ for $\mathcal U$-many $n$, or
			
			\item[$\cdot$]
			$\tau_{n}(G_{n})=C_{N_{n}}$ for $\mathcal U$-many $n$, where $\lim_{n\to \mathcal U} N_{n}=\infty$.
		\end{enumerate}
		
		\item[(ii)]    If $I\subseteq \mathbb T$ is an interval let $\tilde{I}_{n}:=\tau_{n}^{-1}(I)$.  Then either $\tilde{I}_{n}$ is a Bohr interval of $m_{n}$-measure $\lambda(I)$ for $\mathcal U$-many $n$, or $\tilde{I}_{n}$ is an $N_{n}$-cyclic progression of $m_{n}$-measure at most  $\lambda(I)+\frac{1}{N_{n}}$, where $\lim_{n\to \mathcal U} N_{n}=\infty$.

		\item[(iii)] If $\mb A=\prod_{n\to \mathcal U} A_{n}$ is an internal set and $\mb A\subseteq \tau^{-1}(I)$, then there are intervals $I_{n}'\subseteq \mathbb T$ such that $\lim_{n\to \mathcal U} \lambda(I_{n}')\leq \lambda(I)$ and $A_{n}\subseteq \tau_{n}^{-1}(I_{n}')$ for $\mathcal U$-many $n$.
	\end{enumerate}
\end{lemma}

\begin{proof}
	(i) The continuity of $\tau_{n}$ implies that $\tau_{n}(G_{n})$ is either $\mathbb T$ or a finite cyclic subgroup of $\mathbb T$, as these are the only compact subgroups of $\mathbb T$.  The hypothesis that $\tau$ is surjective implies that for all $x\in \mathbb T$, there is a sequence of elements $g_{n}\in G_{n}$ such that $\lim_{n\to \mathcal U} \tau_{n}(g_{n})=x$. Thus if $\tau_{n}(G_{n})\neq\mathbb T$ for $\mathcal U$-many $n$, then $\tau_{n}(G_{n})=C_{N_{n}}$, where $\lim_{n\to \mathcal U} N_{n}=\infty$.
	
	(ii) If $\tau_{n}(G_{n})=\mathbb T$, then $\tau_{n}^{-1}(I)$ is a Bohr interval of $m_{n}$-measure $\lambda(I)$, by definition.  If $\tau_{n}(G_{n})\neq \mathbb T$ for $\mathcal U$-many $n$, then Part (i) implies that $\tau_{n}(G_{n})=C_{N_{n}}$ for $\mathcal U$-many $n$, and $\lim_{n\to \mathcal U} N_{n}=\infty$. For such  $n$,  $\tilde{I}_{n}$ is an $N$-cyclic progression, and the estimate of $m_{n}(\tilde{I}_{n})$ then follows from Lemma \ref{lem:CyclicMeetInterval}.
	
	(iii) The hypothesis $\mb A\subseteq \tau^{-1}(I)$ is equivalent to the condition $\lim_{n\to \mathcal U} \tau_{n}(a_{n}) \in I$ whenever $a_{n}\in A_{n}$. Let $J\subseteq \mathbb T$ be an open interval containing $I$ and $\bar{J}$ its closure.   Observe that $\tau_{n}(A_{n})\subseteq \bar{J}$ for $\mathcal U$-many $n$: assuming otherwise, there is a sequence $a_{n}\in A_{n}$ such that $\tau_{n}(a_{n})\notin \bar{J}$ for $\mathcal U$-many $n$, contradicting the condition $\lim_{n\to \mathcal U} \tau_{n}(a_{n})\in I$.   If we let $I_{n}'$ be the smallest closed interval containing $\tau_n(A_{n})$, we then have $I_{n}'\subseteq \bar{J}$ for $\mathcal U$-many $n.$ Thus $L:=\lim_{n\to \mathcal U} \lambda(I_{n}')\leq \lambda(\bar{J})$.  Since $J$ is an arbitrary open interval containing the interval $I$, this implies $L\leq \lambda(I)$, so the $I_{n}'$ are the desired intervals.
\end{proof}

\begin{corollary}\label{cor:PullPreintervalsDown}
	If $\mb A=\prod_{n\to \mathcal U} A_{n}, \mb B=\prod_{n\to \mathcal U} B_{n}\subseteq \mb G$ are internal sets and $\tilde{I}, \tilde{J}\subseteq \mb G$ are parallel Bohr intervals such that $\mb A\subseteq \tilde{I}$, $\mb B\subseteq \tilde{J}$, and $\mu(\tilde{I})=\mu(\mb A)$, $\mu(\tilde{J})=\mu(\mb B)$ then there are $\tilde{I}_{n}, \tilde{J}_{n}\subseteq  G_{n}$ such that
	\begin{equation}\label{eqn:Containers}
	A_{n}\subseteq \tilde{I}_{n}, \quad B_{n}\subseteq \tilde{J}_{n}, \quad \lim_{n\to \mathcal U} m_{n}(\tilde{I}_{n}\setminus A_{n})+m_{n}(\tilde{J}_{n}\setminus B_{n})=0,
	\end{equation}and $\tilde{I}_{n}$ and $\tilde{J}_{n}$ are either parallel Bohr intervals or parallel $N_{n}$-cyclic progressions, where $\lim_{n\to \mathcal U} N_{n}=\infty$.
\end{corollary}

\begin{proof}
	By the definition of ``parallel Bohr intervals'', we can write $\tilde{I}$ and $\tilde{J}$ as $\tau^{-1}(I)$ and $\tau^{-1}(J)$ for some surjective $\mu$-measurable homomorphism $\tau: \mb G\to \mathbb T$ and intervals $I, J$ contained in $\mathbb T$.  Lemma \ref{lem:InternalCharacter}  implies $\tau=\lim_{n\to \mathcal U} \tau_{n}$ for some continuous homomorphisms $\tau_{n}:G_{n}\to \mathbb T$. Parts (ii) and (iii) of Lemma \ref{lem:InternalIntervals} then guarantee the existence of Bohr intervals or $N$-cyclic progressions $\tilde{I}_{n}, \tilde{J}_{n}$ satisfying (\ref{eqn:Containers}).
\end{proof}

We often need to infer information about a group element $b$ based on the fact that a translate $A+b$ is contained in a certain Bohr interval. This possible under strong hypotheses, according to Lemma \ref{lem:MoveInterval}.  The following definition helps abbreviate the statement.

\begin{definition}\label{def:Movers}
	Let $G$ be a group and  $A, B\subseteq G$.  Define
	\[
	\bar{A}_{B}:=\{g\in G: g+B\subseteq A+B\}.
	\]
	If $\nu$ is a measure on $G$, let $\bar{A}_{B,\nu}:=\{g\in G: g+B\subset_{\nu} A+B\}$.
\end{definition}

The following lemma is elementary and well known in the special case where $G$ is the torus $\mathbb T$.

\begin{lemma}\label{lem:MoveInterval}
	If $G$ is an abelian group with translation invariant probability measure $\nu$ and $A, B\subseteq G$ are parallel Bohr intervals with $\nu(A)+\nu(B)<1$, then $\bar{A}_{B}=\bar{A}_{B,\nu}=A$ and $\bar{B}_{A}=\bar{B}_{A,\nu}=B$.
\end{lemma}

We omit the proof.  The general case follows from the case where $G$ is $\mathbb T$, since the homomorphism $\tau:G\to \mathbb T$ in the definition of ``Bohr interval'' is always surjective and measure preserving.

\begin{corollary}\label{cor:EssentialMoveInterval}
	With $G$ as in Lemma \ref{lem:MoveInterval}, let $A, B\subseteq G$ satisfy $\nu_{*}(A+B)\leq \nu(A)+\nu(B)$, and suppose $\tilde{I}, \tilde{J}\subseteq G$ are parallel Bohr intervals such that $\nu(\tilde{I})+\nu(\tilde{J})<1$,  $A\sim_{\nu} \tilde{I}$, and $B\sim_{\nu}\tilde{J}$.  Then $A\subseteq \tilde{I}$ and $B\subseteq \tilde{J}$.
\end{corollary}

\begin{proof}
	By Lemma \ref{lem:MoveInterval}, to prove $A\subseteq \tilde{I}$ it suffices to show that for all $a\in A$, $a+\tilde{J}\subset_{\nu}\tilde{I}+\tilde{J}$.  Note that the similarities $A\sim_{\nu} \tilde{I}$ and $B\sim_{\nu}\tilde{J}$ imply $1_{A}*1_{B}=1_{\tilde{I}}*1_{\tilde{J}}$, and the level set $E:=\{x:1_{\tilde{I}}*1_{\tilde{J}}(x)>0\}$ satisfies $E\sim_{\nu} \tilde{I}+\tilde{J}$.  Hence $\tilde{I}+\tilde{J}\subset_{\nu} A+B$, so the hypothesis $\nu_{*}(A+B)\leq \nu(A)+\nu(B)$ implies $\nu_{*}(A+B)=\nu(\tilde{I}+\tilde{J})$.  In particular if $C$ is a $\nu$-measurable subset of $A+B$, then $C\subset_{\nu} \tilde{I}+\tilde{J}$. Thus $a+B\subset_{\nu} \tilde{I}+\tilde{J}$, so the similarity $B\sim_{\nu} \tilde{J}$ implies $a+\tilde{J}\subset_{\nu} \tilde{I}+\tilde{J}$ for all $a\in A$, as desired.  This establishes $A\subseteq \tilde{I}$, and the containment $B\subseteq \tilde{J}$ follows by symmetry. \end{proof}

\section{Proof of Theorem \ref{thm:TaoInverse1}}\label{sec:ConnectedProof}

We prove Theorem \ref{thm:TaoInverse1} at the end of this section.  The main ingredient of the proof is Proposition \ref{prop:ConnectedLift}, the natural analogue of Theorem \ref{thm:Satz2} in the setting of ultraproducts of connected compact groups. However, we will not assume in Proposition \ref{prop:ConnectedLift} that the constituent groups $G_{n}$ are connected; we merely assume that $\widehat{\mb G}$ is torsion free.  In fact $\widehat{\mb G}$ is torsion free whenever the $G_{n}$ are connected, but this is so for other ultraproducts, such as $\prod_{n\to \mathcal U} \mathbb Z/p_{n}\mathbb Z$, where $p_{n}$ is the $n^{\text{th}}$ prime.

Fix a sequence of compact abelian groups $G_{n}$ with Haar probability measure $m_{n}$ and a nonprincipal ultrafilter $\mathcal U$ on $\mathbb N$.  Let $\mb G=\prod_{n\to \mathcal U} G_{n}$ with corresponding Loeb measure $\mu$.   Let $\widehat{\mb G}$ denote the group of strong characters of $\mb G$, as in \S\ref{sec:Characters}.  Recall the term ``Bohr interval'' from Definition \ref{def:PreInterval}.

\begin{proposition}\label{prop:ConnectedLift}
	Let $\mb G$ be as above and assume $\widehat{\mb G}$ is torsion free.  If $A$, $B\subseteq \mb G$ are $\mu$-measurable sets satisfying $\mu(A), \mu(B)>0$ and $\mu_{*}(A+B)\leq \mu(A)+\mu(B)<1$, then there are parallel Bohr intervals $\tilde{I}, \tilde{J}\subseteq \mb G$ such that $A\subseteq \tilde{I}$, $B\subseteq \tilde{J}$, and $\mu(A)=\mu(\tilde{I})$, $\mu(B)=\mu(\tilde{J})$.
\end{proposition}

\begin{proof}
	By Corollary \ref{cor:EssentialMoveInterval}, it suffices to find Bohr intervals $\tilde{I}, \tilde{J}\subseteq \mb G$ such that $A\sim_{\mu} \tilde{I}$ and $B\sim_{\mu}\tilde{J}$.

	First we show that every compact quotient $G$ of $\mb G$ with $\mu$-measurable quotient map $\pi:\mb G\to G$ is connected.  To see this, note that the map $\pi^{*}: \widehat{G}\to \widehat{\mb G}$, $\pi^{*}(\chi):=\chi\circ \pi$ is an injective homomorphism, so that $\widehat{G}$ is isomorphic to a subgroup of $\widehat{\mb G}$.  Hence, $\widehat{G}$ is torsion free, and we conclude that $G$ is connected by Theorem 2.5.6(c) of \cite{RudinFourier}.

	Now apply Corollary \ref{cor:ProjectAndPullBack} to find a compact metrizable quotient $G$ of $\mb G$ and a $\mu$-measurable quotient map $\pi$ and Borel sets $C, D\subseteq G$, $A':=\pi^{-1}(C)$, $B':=\pi^{-1}(D)$, such that $C+D$ is Borel,  $A'+B'\subset_{\mu} A+B$, and $A\subset_{\mu} A'$, $B\subset_{\mu} B'$.  Then $\mu(A')\geq \mu(A)$ and $\mu(B')\geq \mu(B)$, so
	\begin{align*}
	m(C+D)=\mu(A'+B')&\leq \mu_{*}(A+B)\\ &\leq \mu(A)+\mu(B)\leq \mu(A')+\mu(B')=m(C)+m(D).
	\end{align*}Then $m(C+D)\leq m(C)+m(D)$ and $m(C+D)<1$, so Corollary \ref{cor:ConnectedExpansion} and the connectedness of $G$ imply $m(C+D)=m(C)+m(D)$.  The inequalities displayed above are therefore all equalities, and in particular $\mu(A)+\mu(B)=\mu(A')+\mu(B')$. The containments $A\subset_{\mu} A'$ and $B\subset_{\mu} B'$ then imply $\mu(A)=\mu(A')$ and $\mu(B)=\mu(B')$, so \begin{equation}\label{eqn:AsimA'}
	A\sim_{\mu}A' \quad \text{and} \quad B\sim_{\mu} B'.
	\end{equation} Theorem \ref{thm:Satz2} implies that there are parallel Bohr intervals $I', J'\subseteq G$ such that $C\subseteq I'$, $D\subseteq J'$ and $m(C)=m(I')$, $m(D)=m(J')$.  Then $\tilde{I}:=\pi^{-1}(I')$ and $\tilde{J}:=\pi^{-1}(J')$ are parallel Bohr intervals such that $A'\subseteq \tilde{I}$, $B'\subseteq \tilde{J}$, and $\mu(A')=\mu(\tilde{I})$, $\mu(B')=\mu(\tilde{J})$.   The similarities in (\ref{eqn:AsimA'}) then imply $A\sim_{\mu} \tilde{I}$ and $B\sim_{\mu}\tilde{J}$, as desired.
\end{proof}

\begin{proof}[Proof of Theorem \ref{thm:TaoInverse1}]
	Suppose, to get a contradiction, that Theorem \ref{thm:TaoInverse1} fails for a given $\varepsilon>0$.  Then for all sufficiently large $n$, there is a connected compact abelian group $G_{n}$ with Haar probability measure $m_{n}$, inner Haar measure $m_{n*}$, and sets $A_{n}, B_{n} \subseteq G_{n}$ having $m_{n}(A_{n}), m_{n}(B_{n})>\varepsilon$ such that
	\begin{enumerate}
		\item[(a.1)] $m_{n*}(A_{n}+B_{n})\leq m_{n}(A_{n})+m_{n}(B_{n})+\frac{1}{n}<1-\frac{
			\varepsilon}{2}$,   and
		
		\item[(a.2)] for every pair of parallel Bohr intervals $\tilde{I}_{n}$, $ \tilde{J}_{n}\subseteq G_{n}$ having $m_{n}(\tilde{I}_{n})< m(A_{n})+\frac{\varepsilon}{2}$, $m_{n}(\tilde{J}_{n})< m(B_{n})+\frac{\varepsilon}{2}$, we have $A_{n}\nsubseteq\tilde{I}_{n}$ or $B_{n}\nsubseteq \tilde{J}_{n}$.
	\end{enumerate}
	Let $\mathcal U$ be a nonprincipal ultrafilter on $\mathbb N$, and form the ultraproduct $\mb G=\prod_{n\to \mathcal U} G_{n}$ with Loeb measure $\mu$ corresponding to $m_{n}$. Consider the internal sets $\mb A=\prod_{n\to \mathcal U} A_{n}$, $\mb B=\prod_{n\to \mathcal U} B_{n}\subseteq \mb G$.  Then
	\[\mu(\mb A)=\lim_{n\to \mathcal U} m_{n}(A_{n})\geq \varepsilon,\  \mu(\mb B)=\lim_{n\to\mathcal U} m_{n}(B_{n})\geq \varepsilon,\] and $\mu_{*}(\mb A+\mb B)=\lim_{n\to \mathcal U} m_{n*}(A_{n}+B_{n})$ by Lemma \ref{lem:InnerLoeb},  so  $\mu_{*}(\mb A+\mb B)\leq \mu(\mb A)+\mu(\mb B)~<~1$.
	
	Observe that $\widehat{\mb G}$ is torsion free by Lemma \ref{lem:GHatTorsion}, as  each $\widehat{G}_{n}$ is torsion free due to the connectedness of $G_{n}$. Proposition \ref{prop:ConnectedLift} implies that there are parallel Bohr intervals $\tilde{I}, \tilde{J}\subseteq \mb G$ with $\mb A\subseteq \tilde{I}$, $\mb B\subseteq \tilde{J}$, and $\mu(\mb A)=\mu(\tilde{I})$, $\mu(\mb B)=\mu(\tilde{J})$.  Corollary \ref{cor:PullPreintervalsDown} and the connectedness of $G_{n}$ then imply that for $\mathcal U$-many $n$, there are parallel Bohr intervals $\tilde{I}_{n}, \tilde{J}_{n}\subseteq \mathbb T$ containing $A_{n}$ and $B_{n}$, respectively, such that $m_{n}(\tilde{I}_{n}\setminus A_{n})+m_{n}(\tilde{J}_{n}\setminus B_{n})<\frac{\varepsilon}{2}$.  These inequalities and containments contradict assumption (a.2). \end{proof}

\section{Periodic and quasi-periodic sets in ultraproducts}\label{sec:FiniteIndex}
Here we study periodic and quasi-periodic subsets of ultraproducts (Definitions \ref{def:Periodic} and \ref{def:QP}).
Let $(G_{n})_{n\in \mathbb N}$ be a sequence of compact abelian groups with Haar probability measure $m_{n}$ and inner Haar measure $m_{n*}$.  Let $\mathcal U$ be a nonprincipal ultrafilter on $\mathbb N$, and fix the ultraproduct $\mb G=\prod_{n\to \mathcal U} G_{n}$  with Loeb measure $\mu$ corresponding to $m_{n}$.  Let $\mu_{*}$ denote the inner measure associated to $\mu$. If $K\leq \mb G$ is  a $\mu$-measurable finite index subgroup, Corollary \ref{cor:FiniteIndexInternal} implies $K$ is internal: there is a sequence of compact open subgroups $K_{n}\leq G_{n}$ such that $K=\mb K:=\prod_{n\to \mathcal U} K_{n}$.  With a view toward the proofs of Theorems \ref{thm:Popular} and \ref{thm:Precise}, we investigate how decompositions of an internal set $\mb A$ by cosets of $\mb K$ are related to coset decompositions of the constituent sets $A_{n}$ by cosets of $K_{n}$.

For the remainder of the section, fix coset representatives $\mb x^{(1)}, \dots, \mb x^{(k)}$ for $\mb K$, so that $\mb G$ is the disjoint union  $\bigcup_{j=1}^{k} (\mb x^{(j)}+\mb K)$.  For each $j\leq k$ and $n\in \mathbb N$, choose  $x_{n}^{(j)}\in G_{n}$ so that $(x_{n}^{(j)})_{n\in \mathbb N}$ represents $\mb x^{(j)}$ (in the sense of \S\ref{sec:UPconstruction}).  We write $\mb K^{(j)}$ for $\mb x^{(j)}+\mb K$ and write $K^{(j)}_{n}$ for $x^{(j)}_{n}+K_{n}$.

\begin{lemma}\label{lem:FiniteIsomorphism}
	Suppose $\mb G$, $\mb K$, $\mb K^{(j)}$, and $K_{n}^{(j)}$ are as defined in the preceding paragraph.  Let $\mb A=\prod_{n\to \mathcal U} A_{n}\subseteq \mb G$ be an internal set.  Then the following hold for $\mathcal U$-many $n$.  To be clear, this means that there is a $U\in \mathcal U$ such that for all $n\in U$, the following are true.
	
	\begin{enumerate}
		\item[(i)]    The map $\phi_{n}: G_{n}/K_{n}\to \mb G/\mb K$ defined by $\phi_{n}(K_{n}^{(j)})=\mb K^{(j)}$ is an isomorphism.
		
		\item[(ii)]  $A_{n}\cap K_{n}^{(j)}=\varnothing$ iff $\mb A \cap \mb K^{(j)}=\varnothing$, while $A_{n}\cap K_{n}^{(j)}=K_{n}^{(j)}$ iff $\mb A \cap \mb K^{(j)}=\mb K^{(j)}$.
		
		\item[(iii)]   $\phi_{n}(A_{n}+K_{n})=\mb A+\mb K$; consequently $\mu(\mb A+\mb K)=m_{n}(A_{n}+K_{n})$.
		
	\end{enumerate}
	
	\noindent Furthermore,
	
	\begin{enumerate}
		\item[(iv)] For all $j$, $\lim_{n\to \mathcal U} m_{n*}(A_{n}\cap K_{n}^{(j)})=\mu_{*}(\mb A\cap \mb K^{(j)})$; if the sets $A_{n}$ are $m_{n}$-measurable we have $\lim_{n\to \mathcal U} m_{n}(A_{n}\cap K_{n}^{(j)})=\mu(\mb A\cap \mb K^{(j)})$.
	\end{enumerate}
	
\end{lemma}

\begin{proof}
	To prove (i) first observe that by Corollary \ref{cor:FiniteIndexInternal} we have $|G_{n}/K_{n}|=k$ for $\mathcal U$-many $n$.  Thus it suffices to prove that $\Phi:=\prod_{n\to \mathcal U} \phi_{n}$ is an isomorphism from $\prod_{n\to \mathcal U} (G_{n}/K_{n})$ to $\mb G/\mb K$.  That $\Phi$ is a homomorphism follows from the identity $(x_{n}^{(i)}+x_{n}^{(j)})\sim_{\mathcal U} (x_{n}^{(i)})+(x_{n}^{(j)})$ for each $i, j\leq k$.  The injectivity of $\Phi$ follows from the fact that if $x_{n}^{(i)}-x_{n}^{(j)} \in K_{n}$ for $\mathcal U$-many $n$, then $\mb x^{(i)}-\mb x^{(j)}\in \mb K$, and surjectivity then follows from the equality of cardinalities of the domain and codomain of $\Phi$.  So $\Phi$ is an isomorphism, and hence $\phi_{n}$ is an isomorphism for $\mathcal U$-many $n$.

	Part (ii) follows from the definition of the ultraproduct and our choice of $(x_{n}^{(j)})_{n\in \mathbb N}$.  Part (iii) is an immediate consequence of Part (ii).  Part (iv) follows from the definition of Loeb measure and Lemma \ref{lem:InnerLoeb}. \end{proof}

Before stating the next lemma, we remark that if $\mb A\subseteq \mb G$ is internal, $\mb K$ is an internal subgroup, and $\mb A=A_{1}\cup A_{0}$ is a quasi-periodic decomposition of $\mb A$ with respect to $\mb K$, then the sets $A_{1}$ and $A_{0}$ are internal: $A_{0}$ is the intersection of $\mb A$ with a coset of $\mb K$, while $A_{1}$ is the difference $\mb A\setminus A_{0}$.

\begin{lemma}\label{lem:PullQPdown}
	Suppose $\mb G, \mb K$, $\mb A$, $K_{n}$, and $\phi_{n}$ are as in Lemma \ref{lem:FiniteIsomorphism}, and $\mb A$ has a quasi-periodic decomposition $\mb A=\mb A_{1}\cup \mb A_{0}$ with respect to $\mb K$. If $\varepsilon>0$ then for $\mathcal U$-many $n$, $A_{n}$ has an $\varepsilon$-quasi-periodic decomposition $A_{n}=A_{n,1}\cup A_{n,0}$ with respect to $K_{n}$, such that $\phi_{n}(A_{n,l}+K_{n})=\mb A_{l}+\mb K$ for $l=0, 1$.
\end{lemma}

\begin{proof}
	Write $\mb A_{l}$ as $\prod_{n\to \mathcal U} A_{n,l}$ for $l=0,1$. Parts (i) and (ii) of Lemma \ref{lem:FiniteIsomorphism} imply that for $\mathcal U$-many $n$,  $A_{n,0}$ is contained in a coset of $K_{n}$, $A_{n,0}+K_{n}$ is disjoint from $A_{n,1}$, and  $m_{n}(A_{n}+K_{n})=\mu(\mb A+\mb K)$. Part (iv) then implies $\lim_{n\to \mathcal U}m_{n}(A_{n,1}+K_{n})- m_{n*}(A_{n,1})=0$, and in particular $m_{n}(A_{n,1}+K_{n})-m_{n*}(A_{n,1})<\varepsilon \mu(K_{n})$ for $\mathcal U$-many $n$.
\end{proof}

\begin{lemma}\label{lem:StabilizeInLimit}
	With $\mb K$ and $K_{n}$ as in the preceding lemmas,  let $\mb A=\prod_{n\to \mathcal U} A_{n}$ and $\mb B=\prod_{n\to \mathcal U} B_{n}\subseteq \mb G$ be internal sets such that $\mb A+\mb B=\mb A+\mb B+\mb K$. Then
	\begin{enumerate}
		\item[(i)]
		$A_{n}+B_{n}=A_{n}+B_{n}+K_{n}$ for $\mathcal U$-many $n$.
		
		\item[(ii)] If, in addition, $\mu(\mb A+\mb B)=\mu(\mb A+\mb K)+\mu(\mb B+\mb K)-\mu(\mb K)$, then \[m_{n}(A_{n}+B_{n})=m_{n}(A_{n}+K_{n})+m_{n}(B_{n}+K_{n})-m_{n}(K_{n}) \quad \text{for $\mathcal U$-many $n$.}\]
		
	\end{enumerate}
\end{lemma}

\begin{proof}
	Parts (i) and (ii) follow from Lemma \ref{lem:FiniteIsomorphism} and the fact that Haar measure on $G_{n}/K_{n}$ and Loeb measure on $\mb G/\mb K$ are both normalized counting measure.
\end{proof}

\section{Reducing to tame critical pairs}\label{sec:ReducibleGeneral}

Here we collect some technical lemmas for \S\ref{sec:LiftLCA}, where the majority of our proofs are carried out.  These are mainly consequences of the \emph{conclusion} of Theorem \ref{thm:Satz1}.  The results of this section do not require countable additivity  of the relevant measures, so we work in a slightly more general setting than in the other sections.

Call a measure $\nu$ on a group $G$ \emph{symmetric} if $\nu(-A)=\nu(A)$ for all $\nu$-measurable sets $A$.  Note that Haar measure $m$ on a compact abelian group is symmetric, so the associated Loeb measure on an ultraproduct of compact abelian groups is also symmetric.

For the remainder of the section, we fix an abelian group $G$ and a \emph{finitely} additive (or countably additive), symmetric, translation invariant measure $\nu$ on $G$.  We do not assume that $\nu(G)$ is finite.  We continue denote by $\nu_{*}$ the inner measure associated to $\nu$: $\nu_{*}(A)=\sup\{\nu(C): C\subseteq A \text{ is } \nu\text{-measurable}\}$.

\begin{lemma}\label{lem:LargeAB}
	If $A, B\subseteq G$ have finite measure and satisfy $\nu(A)+\nu(B)> \nu(G)$, then $A+B=G$.  If $K\leq G$ is a $\nu$-measurable subgroup such that $0<\nu(K)<\infty$ and $\nu((a+K)\cap A)+\nu((b+K)\cap B)>\nu(K)$ for some $a\in A, b\in B$, then $a+b+K\subseteq A+B$.
\end{lemma}

\begin{proof}
	If $\nu(A)+\nu(B)>\nu(G)$  then for all $t\in G$, we have $\nu(A\cap ( t-B))>0$, and in particular $t\in A+B$.  Thus $A+B=G$.  To prove the second assertion, define the measure $\nu_{K}$ as $\nu$ restricted to $K$. Then $\nu_{K}((A-a)\cap K)+\nu_{K}((B-b)\cap K)>\nu(K)$, so by the first part of the lemma we have $K\subseteq A+B-a-b$, meaning $a+b+K\subseteq A+B$.
\end{proof}

\begin{definition}\label{def:Solid}
	If $K$ is a $\nu$-measurable finite index subgroup of $G$ and $A\subseteq G$, we say that $A$ is \emph{$K$-solid} if for all $g\in G$, either $\nu((g+K)\cap A)>0$ or $(g+K)\cap A=\varnothing$.
\end{definition}

\begin{lemma}\label{lem:SolidMove}
	Let $ K\leq  G$ be a $\nu$-measurable finite index subgroup of $G$, $A\subseteq  G$, and let  $B\subseteq G$ be $ K$-solid. If $A+B\sim_{\nu} A+B+ K$ and $g\in G$, then $g+B\subset_{\nu} A+B+K$ if and only if $g+B+K\subseteq A+B+K$.
\end{lemma}
In the notation of Definition \ref{def:Movers}, the conclusion of the lemma says   $\bar{A}_{B,\nu}=\bar{A}_{B+K}$.
\begin{proof}
	The inclusion $\bar{A}_{B+K}\subseteq \bar{A}_{B,\nu}$ follows from the similarity $A+B\sim_{\nu} A+B+K$. To prove the reverse inclusion, let $g\in \bar{A}_{B,\nu}$, so that  $g+B\subset_{\nu} A+B$.  If $g+B\nsubseteq A+B+K$, then for some coset $h+K$ disjoint from $A+B+K$, $(g+B)\cap(h+K)\neq \varnothing$, so the $K$-solidity of $B$ implies $\nu((g+B)\cap (h+K))>0$, and we get that $g+B\not\subset_{\nu} A+B+K$, and therefore $g\notin \bar{A}_{B,\nu}$.  So we conclude that $g+B\subseteq A+B+K$, meaning $g\in \bar{A}_{B+K}$.
\end{proof}

In the next definition we use the term ``stabilizer'' from Definition \ref{def:Stabilizer}.

\begin{definition}\label{def:TameCritical}  We say that two $\nu$-measurable subsets $A, B$ of $G$ form a \emph{critical pair} if $\nu_{*}(A+B)<\nu(A)+\nu(B)$.  We say they form a \emph{tame critical pair} if, in addition, $A+B$ is $\nu$-measurable, the stabilizer $H:=H(A+B)$ is $\nu$-measurable, and
	\begin{equation}\label{eqn:PushInequality}
	\nu(A+B)=\nu(A+H)+\nu(B+H)-\nu(H).
	\end{equation} In other words, $A, B$ is a tame critical pair if it satisfies the conclusion of Theorem \ref{thm:Satz1}.  \end{definition}
Observe that the criticality of $A,B$ and Equation (\ref{eqn:PushInequality}) imply $\nu(H)>0$.

The following lemmas provide useful information for tame pairs and related pairs of subsets of $G$.  The first of these uses notation from Definition \ref{def:Movers} and the following.

\begin{notation}
	If $H\leq G$ is a subgroup and $g\in G$, we write $H_{g}$ for the coset $g+H$.
\end{notation}

\begin{lemma}\label{lem:PushSubCritical}
	Let $A, B\subseteq G$ be a tame critical pair and $H:=H(A+B)$  the stabilizer of $A+B$. Then
	\begin{enumerate}
		\item[(i)]
		for all $a\in A$, $b\in B$
		\begin{equation}\label{eqn:FillCosets} \begin{split}
		\nu(A\cap H_{a})+\nu(A+B)\geq \nu(A)+\nu(B), \\
		\nu(B\cap H_{b})+\nu(A+B) \geq \nu(A)+\nu(B), \end{split}
		\end{equation}
		In particular, $A$ and $B$ are $H$-solid.
		\item[(ii)]  We have
		\begin{equation}\label{eqn:CompleteSummand}
		\bar{A}_{B,\nu} =\bar{A}_{B}= A+H,\  \bar{B}_{A,\nu} = \bar{B}_{A}=B+H.
		\end{equation} 
 \item[(iii)]  Furthermore, if $g\in G$ and $A+g\nsubseteq A+B$, then $\nu((A+g)\cup (A+B))\geq \nu(A)+\nu(B).$	
\end{enumerate}
\end{lemma}

\begin{proof}  To prove the first inequality in Part (i), fix $a\in A$. Let $A':=A\cap H_{a}$, $A'':=A\setminus A'$, so that $A''\cap H_{a}=\varnothing$ and $\nu(A)=\nu(A')+\nu(A'')$.  Then $A+H$ is the disjoint union $(A''+H)\cup H_{a}$, so $\nu(A+H)-\nu(H)=\nu(A''+H)$.  Replacing $\nu(A+H)-\nu(H)$ with $\nu(A''+H)$ in Equation (\ref{eqn:PushInequality}) and adding $\nu(A')$ to both sides results in the following:
	\[\nu(A')+\nu(A+B)=\nu(A')+\nu(A''+H)+\nu(B+H)\geq \nu(A')+\nu(A'')+\nu(B)= \nu(A)+\nu(B),\] meaning $\nu(A')+\nu(A+B)\geq \nu(A)+\nu(B)$.  This is the desired inequality. The second inequality in Part (i) follows by symmetry. Note that the criticality of $A,B$ and (\ref{eqn:FillCosets}) imply that $A$ and $B$ are $H$-solid.

	To prove Part (ii) we first establish $\bar{A}_{B,\nu}=\bar{A}_{B}$. Note that the tameness of $A+B$ implies $A+B$ is a union of cosets of $H$, so the $H$-solidity of $B$ and Lemma \ref{lem:SolidMove} imply $\bar{A}_{B,\nu}=\bar{A}_{B}$.  To prove that $\bar{A}_{B,\nu}=A+H$, we reduce the problem to the special case where $\nu$ is counting measure. Lemma \ref{lem:SolidMove} implies $\bar{A}_{B,\nu} = \bar{A}_{B+H}$, and the last set is clearly a union of cosets of $H$. So all sets we are currently considering in the proof are unions of cosets of $H$, and we may work in the quotient $G/H$.  We consider $G/H$ as a discrete group with counting measure, so we may apply Theorem \ref{thm:Satz1} in this setting.  Replacing $A$ with $A+H$ and $B$ with $B+H$ in $G/H$, the problem is to prove that $\bar{A}_{B}=A$, assuming $|A+B|=|A|+|B|-1$ and the stabilizer $H$ of $A+B$ is trivial.  We clearly have $A\subseteq \bar{A}_{B}$, so we will prove the reverse inclusion. Assume, to get a contradiction, that there is a $c\notin A$ such that $c+B\subseteq A+B$.  Setting $\tilde{A}=A\cup \{c\}$, we get that $\tilde{A}+B=A+B$, so $H(\tilde{A}+B)=H(A+B)$.  Our choice of $\tilde{A}$ then implies $|\tilde{A}+B|= |A+B| = |\tilde{A}|+|B|-2$, which is strictly less than $|\tilde{A}+H|+|B+H|-|H|$, as $|H|=1$. This contradicts Inequality (\ref{eqn:Kneser}) in Theorem \ref{thm:Satz1} applied to $\tilde{A},B$.  Thus $\bar{A}_{B}=A$.  

Part (iii) is a consequence of Part (i) and the fact that $A+B$ is a union of cosets of $H$: under the hypothesis of (iii) we have that $(A+g)\setminus (A+B)$ contains a set of the form $(A\cap H_a)+g$ which is disjoint from $A+B$.  \end{proof}

\begin{lemma}\label{lem:Overflow}
	Suppose $H$ is a subgroup of $G$ having $0<\nu(H)<\infty$ and $A,B\subseteq G$ satisfy
	\begin{equation}\label{eqn:PseudoTame}
	\nu(A+H)+\nu(B+H) - \nu(H) <\nu(A)+\nu(B).
	\end{equation}
	Then for all $a\in A$, $b\in B$, we have $\nu(A\cap H_{a})+\nu(B\cap H_{b})>\nu(H)$.
\end{lemma}

\begin{proof}
	Inequality (\ref{eqn:PseudoTame}) can be written as $\nu(A+H)-\nu(A)+\nu(B+H)-\nu(B) < \nu(H)$, meaning 
	\begin{equation}\label{eqn:TooSmall}
		\nu((A+H)\setminus A)+\nu((B+H)\setminus B)<\nu(H).
	\end{equation}
	Now $A+H=\bigcup_{a\in A} H_{a}$ and $B+H=\bigcup_{b\in B} H_{b}$, so (\ref{eqn:TooSmall}) implies $\nu(H_{a}\setminus A)+\nu(H_{b}\setminus B)<\nu(H)$ for all $a, b\in H$.  The last inequality can be written as $\nu(H)-\nu(A\cap H_{a})+\nu(H)-\nu(B\cap H_{b})<\nu(H)$, which can be rearranged to yield the desired conclusion.
\end{proof}

The next two lemmas form a major portion of the proof of Proposition \ref{prop:LiftSubcritical},  generalizing  Theorem \ref{thm:Satz1} to ultraproducts of compact abelian groups.

\begin{lemma}\label{lem:SimCriticalSub}
	Suppose $A,B\subseteq G$ is a critical pair and $A', B'$ is a tame critical pair such that   $A\subseteq A'$, $B\subseteq B'$, and $\nu_{*}(A+B)=\nu(A'+B')$.  Let $H=H(A'+B')$. Then for all $a\in A', b\in B'$, we have
	\begin{equation}\label{eqn:SolidPrime}
	\nu(A\cap H_{a})+\nu(B\cap H_{b})>\nu(H).
	\end{equation}
	Consequently,
	\begin{equation}\label{eqn:SimCritical}
	A+H=A'+H, \quad B+H=B'+H,  \quad A+B=A'+B'+H,
	\end{equation}
	and $A, B$ is a tame critical pair with $H(A+B)=H$.
\end{lemma}

\begin{proof}
	The proof of Inequality (\ref{eqn:SolidPrime}) is  similar to the proof of Lemma \ref{lem:Overflow}. We have
	\[\nu(A)+\nu(B)>\nu_{*}(A+B)=\nu(A'+B')=\nu(A'+H)+\nu(B'+H)-\nu(H),\] so $\nu(A'+H)-\nu(A)+\nu(B'+H)-\nu(B)<\nu(H)$. Since $A\subseteq A'+H$ and $B\subseteq B'+H$, this inequality implies $\nu(H_{a}\setminus A)+\nu(H_{b} \setminus B)<\nu(H)$ for all $a\in A', b\in B'$.  As in the proof of Lemma \ref{lem:Overflow}, the latter inequality  can be rearranged to yield (\ref{eqn:SolidPrime}).
	
	Now Inequality (\ref{eqn:SolidPrime}) implies that $A\cap H_{a}\neq \varnothing$ for all $a\in A'$. Together with the assumption $A\subseteq A'$ this implies $A+H=A'+H$, and by symmetry we get $B+H=B'+H$.  The equation $A+B+H=A'+B'+H$ follows from these equations, and the equation $A+B=A'+B'+H$ then follows from (\ref{eqn:SolidPrime}) and Lemma \ref{lem:LargeAB}.  We therefore have $H(A+B)=H(A'+B'+H)$, while the definition of $H$ implies $A'+B'=A'+B'+H$, and we may conclude $H(A+B)=H(A'+B')$.  The tameness of $A,B$ then follows from the tameness of $A',B'$ and the equations in (\ref{eqn:SimCritical}).
\end{proof}

\begin{lemma}\label{lem:TameSuper}
	Suppose $A, B\subseteq G$ is a critical pair and $A',B'$ is a tame critical pair with $A'\subseteq A, B'\subseteq B$, and $\nu(A)=\nu(A')$, $\nu(B)=\nu(B')$.  Then $A,B$ is a tame critical pair with $H(A+B)=H(A'+B')$.
\end{lemma}

\begin{proof}
	We must prove that $H:=H(A'+B')$ is the stabilizer of $A+B$ and and that Equation (\ref{eqn:PushInequality}) is satisfied.  To prove this, we will show that $A+B=A'+B'$, $A+H=A'+H$, and $B+H=B'+H$.
	
	We first show that 
	\begin{equation}\label{eqn:essentialA'b}
	A'+b\subset_{\nu} A'+B' \quad \text{ for all } b\in B.
	\end{equation} To prove (\ref{eqn:essentialA'b}) we assume otherwise and derive a contradiction by proving that $\nu_*(A+B)\geq \nu(A)+\nu(B)$, contrary to the hypothesis.  So we assume there is a $b\in B$  such that $C:=(A'+b)\setminus (A'+B')$ satisfies $\nu(C)>0$.  Then Part (iii) of Lemma \ref{lem:PushSubCritical} implies $\nu((A'+b)\cup (A'+B'))\geq \nu(A')+\nu(B')$.  The containments in the hypothesis imply $(A'+b)\cup(A'+B')\subseteq A+B$, so  $\nu_*(A+B)\geq \nu(A')+\nu(B')=\nu(A)+\nu(B)$, which is the desired contradiction.
	
Having proved (\ref{eqn:essentialA'b}),  Part (ii) of Lemma \ref{lem:PushSubCritical} and (\ref{eqn:essentialA'b}) together imply $B\subseteq B'+H$.  Likewise we have $A\subseteq A'+H$.  The containments $A'\subseteq A$ and $B'\subseteq B$ then imply $A+H=A'+H$ and $B+H=B'+H$, and the equation $A'+B'=A'+B'+H$ implies $A+B=A'+B'+H$. Finally, the tameness of $A',B'$  implies $\nu(A+B)=\nu(A'+B')=\nu(A+H)+\nu(B+H)-\nu(H)$, as desired. \end{proof}

The next definition  is useful in dealing with quasi-periodic pairs (Definition \ref{def:QP}).

\begin{definition}\label{def:UniqueExpression}
	If $H$ is a subgroup of $G$, $A, B\subseteq G$, and $a\in A, b\in B$, we say that $a+b+H$ is a \emph{unique expression element of $A+B+H$ in $G/H$} if $a'\in A$, $b'\in B$ and $a'+b'+H=a+b+H$ imply $a' \in a+H$ and $b'\in b+H$.
\end{definition}

The next lemma deals with a case arising frequently  in the proofs of our main results.

\begin{lemma}\label{lem:Pop1}
	Assume $A, B\subseteq G$ satisfy $\nu_{*}(A+B)=\nu(A)+\nu(B)$, $B'\subseteq B$ satisfies $\nu(B')=\nu(B)$, and $A,B'$ is a tame critical pair in $G$.  Let $H=H(A+B')$.  Then
	\begin{enumerate}
		\item[(i)] $(A+B+H)\setminus (A+B')$ is a coset of $H$,
		\item[(ii)]  $B'\sim_{\nu} B'+H$,
		\item[(iii)] $A$ has a quasi-periodic decomposition $A_{1}\cup A_{0}$ with respect to $H$, where \[\nu(A_{0})=\nu_{*}((A+B)\setminus (A+B')).\]
		\item[(iv)] If $A+B\nsim_{\nu} A+B+H$, then setting $B_{1}:=B\cap (B'+H)$, $B_{0}:=B\setminus B_{1}$ yields a quasi-periodic decomposition of $B$ with respect to $H$ where $\nu(B_{0})=0$, and $A_{0}+B_{0}+H$ is a unique expression element of $A+B+H$.  Furthermore, $\nu_{*}(A_{0}+B_{0})=\nu(A_{0})+\nu(B_{0})$.
	\end{enumerate}
\end{lemma}

\begin{proof}
	Throughout this proof we let $C$ denote $(A+B+H)\setminus (A+B')$.  We first observe that $C$ is nonempty, as $A+B'$ is a union of cosets of $H$ and is a proper subset of $A+B$.
	
	We prove Part (i) by contradiction: assume that $C$ has nonempty intersection with at least two cosets of $H$.  We may then choose $a_{1}, a_{2}\in A$ and $b_{1}, b_{2}\in B$ such that $a_{1}+b_{1}$ and $a_{2}+b_{2}$ lie in $C$ and do not occupy the same coset of $H$.  Let $A_{1}=A\cap (a_{1}+H)$, $A_{2}=A\cap (a_{2}+H)$.  Then $A_{1}+b_{1}$, $A_{2}+b_{2}$, and $A+B'$ are mutually disjoint subsets of $A+B$, while Lemma \ref{lem:PushSubCritical} (i) implies $\nu(A_{i}+b_{i})+\nu(A+B')\geq \nu(A)+\nu(B')$ for each $i$.  The last inequality, the criticality of $A,B'$, and the disjointness of $C$ from $A+B'$ imply $\nu_{*}(A+B)>\nu(A)+\nu(B')=\nu(A)+\nu(B)$, contradicting our assumption to the contrary.
	
	To prove Parts (ii) and (iii),  choose $a\in A$ and $b\in B$ such that $a+b\in (A+B)\setminus (A+B')$.  Let $A_{0}:=A\cap (a+H)$, so that Lemma \ref{lem:PushSubCritical} (i) implies  $\nu(A_{0})+\nu(A+B')\geq \nu(A)+\nu(B')$, while the containment $(A_{0}+b)\cup (A+B')\subseteq A+B$ and the hypothesis on $\nu(A+B)$ imply the reverse inequality.  Thus
	\begin{equation}\label{eqn:A0andC}
	\nu(A_{0})+\nu(A+B')=\nu(A)+\nu(B'),
	\end{equation} and the assumption that $A, B'$ is a tame critical pair allows us to use Equation (\ref{eqn:PushInequality}) with $B'$ in place of $B$. Combining that equation with (\ref{eqn:A0andC}) and rearranging, we obtain
	\begin{equation}\label{eqn:Penultimate}
	\nu(A+H)-\nu(A)+\nu(B'+H)-\nu(B') = \nu(H)-\nu(A_{0}).
	\end{equation}
	Since $A_{0}$ is contained in a coset of $H$, the right hand side above is equal to $\nu((A_{0}+H)\setminus A_{0})$. So we can rewrite (\ref{eqn:Penultimate}) as
	\begin{equation}\label{eqn:ReallyPenultimate}
	\nu((A+H)\setminus A)+\nu((B'+H)\setminus B') = \nu((A_{0}+H)\setminus A_{0}).
	\end{equation}
	Since $(A_{0}+H)\setminus A_{0}\subseteq (A+H)\setminus A$, we have $\nu((A+H)\setminus A)\geq  \nu((A_{0}+H)\setminus A_{0})$, and (\ref{eqn:ReallyPenultimate}) then implies
	\begin{equation}\label{eqn:AisQP}
	\nu(A+H)-\nu(A) = \nu(H)-\nu(A_{0})
	\end{equation} and $\nu(B'+H)-\nu(B')=0$. The last equation implies (ii). Setting $A_{1}=A\setminus A_{0}$, Equation (\ref{eqn:AisQP}) implies $A_{1}\sim_{\nu} A_{1}+H$.  Thus $A_{1}\cup A_{0}$ is the desired quasi-periodic decomposition of $A$.  To complete the proof of (iii), note that the right hand side of Equation (\ref{eqn:A0andC}) is $\nu_{*}(A+B)$, so that equation can be rewritten as $\nu(A_{0})=\nu_{*}(A+B)-\nu(A+B')$.
	
	We prove Part (iv) using the following claim.
	\begin{claim} The additional hypothesis (iv) implies $C+H$ is a unique expression element of $(A+H)+(B+H)$ in $G/H$.
	\end{claim}  To prove the Claim, we will show that $C+H$ is disjoint from $A_1+B$, meaning that if $a\in A, b\in B$ satisfy $a+b+H=C+H$, then $a\in A_0+H$ and $b\in C-A_0+H$. Assume, to get a contradiction, that $C+H$ is not disjoint from $A_1+B$.   Choose $a_{1}\in A_{1}$ and $b\in B$ such that $a_{1}+b\in C+H$.  Since $A_{1}\sim_{\nu} A_{1}+H$, we conclude that $C+H\subset_{\nu} A+B$.  Together with Part (i) and the equation $A+B'=A+B'+H$, this implies $A+B\sim_{\nu} A+B+H$, contradicting the hypothesis in (iv).  This proves the Claim.
	
	Now let $B_{1}$ be as in the statement of Part (iv), and let $A_{0}$ be as defined in the proof of Part (iii). In particular $A_{0}+b_{0}\subseteq C+H$ for some $b_{0}\in B$.  Let $B_{0}=B\cap (b_{0}+H)$. We will show that $B=B_{1}\cup B_{0}$, which by Part (ii) yields the desired quasi-periodic decomposition. If $A+b\subseteq A+B'$, we will show that $b\in B'+H$, and otherwise we will show that $b\in B_{0}$.  Under the first assumption,  we apply Part (ii) of Lemma \ref{lem:PushSubCritical} to $A, B'$ and conclude $b\in B'+H$.  If $A+b\nsubseteq A+B'$, the containment $A+B\subseteq C\cup (A+B')$ implies $(A+b)\cap C\neq \varnothing$.  Then there is an $a\in A$ such that $a+b\in C$.  The Claim then implies $a\in A_{0}$, so $b\in B_{0}$.  The last equation in the conclusion of (iv) follows from (iii) and the fact that $\nu(B_{0})=0$.     \end{proof}

The following lemma is used only in the proof of Lemma \ref{lem:FirstHalf}.

\begin{lemma}\label{lem:BuildQP}
	Let $K$ be a $\nu$-measurable subgroup of $G$ with $0<\nu(K)<\infty$.  Suppose $C, D\subseteq G$ are periodic with respect to $K$, and $\nu(C+D)=\nu(C)+\nu(D)-\nu(K)$.  Assume further that $C_{0}, D_{0}$ are cosets of $K$ contained in $C$ and $D$, respectively, and $C_{0}+D_{0}$ is a unique expression element of $C+D+K$ in $G/K$.  Let $A_{1}=C\setminus C_{0}$, $B_{1}=D\setminus D_{0}$, and suppose $A_{0}\subseteq C_{0}$ and $B_{0}\subseteq D_{0}$ are nonempty and satisfy $\nu_{*}(A_{0}+B_{0})\leq \nu(A_{0})+\nu(B_{0})$.  If $A=A_{1}\cup A_{0}$ and $B=B_{1}\cup B_{0}$, then $\nu_{*}(A+B)\leq \nu(A)+\nu(B)$.
\end{lemma}

\begin{proof}
	Let $E=(A+B_{1})\cup(A_{1}+B)$.  The definitions of $A$ and $B$ and the fact that $C_{0}+D_{0}$ is a unique expression element of $C+D+K$ then imply $E=(C+D)\setminus (C_{0}+D_{0})$, so $\nu(E)=\nu(C+D)-\nu(K)$.  Note that $\nu(C+D)-\nu(K)$ simplifies to \[\nu(C)-\nu(K)+\nu(D)-\nu(K)=\nu(A_{1})+\nu(B_{1}).\] Then
	\[\nu_{*}(A+B)=\nu(E)+\nu_{*}(A_{0}+B_{0})\leq \nu(A_{1})+\nu(A_{0})+\nu(B_{1})+\nu(B_{0})=\nu(A)+\nu(B),\]
	as desired.
\end{proof}

\section{Lifting Theorems \ref{thm:Satz1} and \ref{thm:LCAInverse} to ultraproducts}\label{sec:LiftLCA}
The main results of this section are Propositions \ref{prop:LiftSubcritical} and \ref{prop:Reduction},  the natural generalizations of Theorems \ref{thm:Satz1} and \ref{thm:LCAInverse} to ultraproducts of compact abelian groups.

For this section fix a sequence of compact abelian groups $G_{n}$ with Haar probability measure $m_{n}$ and a nonprincipal ultrafilter $\mathcal U$ on $\mathbb N$. Let $\mb G$ be the ultraproduct $\prod_{n\to \mathcal U}G_{n}$ and $\mu$ the Loeb measure on $\mb G$ corresponding to $(m_{n})_{n\in \mathbb N}$.  Let $\mu_{*}$ be the inner measure on $\mb G$ corresponding to $\mu$. Recall Definition \ref{def:Stabilizer}: the stabilizer of a set $C\subseteq \mb G$ is the subgroup \[H(C):=\{g\in \mb G:C+g=C\}.\]

\begin{proposition}\label{prop:LiftSubcritical}
	If $A, B\subseteq \mb G$ are $\mu$-measurable sets with $\mu_{*}(A+B)< \mu(A)+\mu(B)$, then the stabilizer $\mb H:=H(A+B)$ is an internal $\mu$-measurable  finite index subgroup of $\mb G$ and
	\begin{equation}\label{eqn:HsubUltra}
	\mu(A+B)=\mu(A+\mb H)+\mu(B+\mb H)-\mu(\mb H).
	\end{equation}  \end{proposition}

The proof of Proposition \ref{prop:LiftSubcritical} is presented after the proof of Lemma \ref{lem:LiftTame}.

In the next proposition we use the terms ``Bohr interval'' and ``quasi-periodic'', from Definitions \ref{def:PreInterval} and \ref{def:QP}.

\begin{proposition}\label{prop:Reduction}Let $A, B\subseteq \mb G$ be $\mu$-measurable sets satisfying $\mu(A)>0$, $ \mu(B)>0$, and $\mu_{*}(A+B)=\mu(A)+\mu(B)$.  Then there is an internal $\mu$-measurable finite index subgroup $\mb K\leq \mb G$ such that exactly one of the following holds.
	\begin{enumerate}
		\item[(I)] $A+B\sim_{\mu} A+B+\mb K$.

		\item[(II)] $A+B\nsim_{\mu} A+B+\mb K$  and $A$ and $B$ have quasi-periodic decompositions $A=A_{1}\cup A_{0}$, $B=B_{1}\cup B_{0}$ with respect to $\mb K$ where at least one of $A_{1}$, $B_{1}$ is nonempty and $\mu_{*}(A_{0}+B_{0})=\mu(A_{0})+\mu(B_{0})$.

		\item[(III)]  $A+B\nsim_{\mu} A+B+\mb K$ and there are parallel Bohr intervals $A', B'\subseteq \mb K$ and $\mb a, \mb b\in \mb G$ such that $\mu(A)=\mu(A')$, $\mu(B)=\mu(B')$, and $A\subseteq \mb a+A'$, $B\subseteq \mb b+ B'$.
	\end{enumerate}
\end{proposition}
The proof of Proposition \ref{prop:Reduction} is presented following the proof of Lemma \ref{lem:TameSimilar}.

\begin{remark}\label{rem:UEE}
	The condition $A+B\nsim_{\mu} A+B+\mb K$ in conclusion (II) above implies that $A_{0}+B_{0}+\mb K$ is a unique expression element of $A+B+\mb K$ in $\mb G/\mb K$ (Definition \ref{def:UniqueExpression}):  assuming otherwise we would have $A+B\sim_{\mu} A+B+\mb K$.   Consequently,
	\begin{align*}
	\mu_{*}((A+B)\setminus (A_{0}+B_{0}))&=\mu_{*}(A+B)-\mu_{*}(A_{0}+B_{0})\\
	&=\mu(A)+\mu(B)-\mu(A_{0})-\mu(B_{0})\\
	&=\mu(A_{1})+\mu(B_{1}).
	\end{align*}
	
	Thus if $A, B$ satisfies conclusion (II) of Proposition \ref{prop:Reduction}, then
	\begin{equation}\label{eqn:MeasureC1}
	\mu_{*}((A+B)\setminus (A_{0}+B_{0}))=\mu(A_{1})+\mu(B_{1}).
	\end{equation}
	We thereby deduce the following corollary (cf. Corollary 3.1 of \cite{GriesmerLCA}).
\end{remark}

\begin{corollary}\label{cor:StabilizerUltra}
	If $G$, $A$, and $B$ are as in Proposition \ref{prop:Reduction} and $A, B$, and $\mb K$ satisfy conclusion (II) therein, then $\mu(A+B+\mb K)=\mu(A+\mb K)+\mu(B+\mb K)-\mu(\mb K).$
\end{corollary}

\begin{proof}  Note that the quasi-periodic decompositions of $A$ and $B$ with respect to $\mb K$ yield $\mu(A+\mb K)=\mu(A_{1})+\mu(\mb K)$, and $\mu(B+\mb K)=\mu(B_{1})+\mu(\mb K)$.  Adding these and rearranging yields
	\begin{equation}\label{eqn:RearrangedOnes}
	\mu(A_{1})+\mu(B_{1})+\mu(\mb K)=\mu(A+\mb K)+\mu(B+\mb K)-\mu(\mb K).
	\end{equation} Equation (\ref{eqn:MeasureC1}) and the fact that $A_{0}+B_{0}$ is a unique expression element of $A+B+\mb K$ in $\mb G/\mb K$ implies $\mu(A+B+\mb K)=\mu(A_{1})+\mu(B_{1})+\mu(\mb K)$, whereby (\ref{eqn:RearrangedOnes}) provides the desired simplification.
\end{proof}

The remainder of this section is dedicated to the proofs of Propositions \ref{prop:LiftSubcritical} and \ref{prop:Reduction}.  The next definitions introduce some useful terminology for these proofs.

For the following definitions let $G$ be an abelian group and $\nu$ a finitely additive translation invariant  measure on $G$.  We call a pair of subsets $A, B\subseteq G$ \emph{critical}  (as in \S\ref{sec:ReducibleGeneral}) if $\nu_{*}(A+B)<\nu(A)+\nu(B)$, or \emph{subcritical} if $\nu_{*}(A+B)=\nu(A)+\nu(B)$.

We continue to use the term ``tame'' (Definition \ref{def:TameCritical}) when considering critical pairs.  For subcritical pairs we extend the definition as follows.

\begin{definition}[Tameness and type]\label{def:Tame}
	A pair of sets $A, B\subseteq \mb G$ is a \emph{tame subcritical pair} if $\nu_{*}(A+B)=\nu(A)+\nu(B)$ and there is a $\nu$-measurable finite index subgroup $ K\leq G$ such that $A,B$, and $ K$ satisfy one of the conclusions (I)-(III) of Proposition \ref{prop:Reduction}.  If $A, B$, and $K$ satisfy conclusion (I), we say that the triple $A, B,  K$ \emph{has type (I)}, and similarly for the other conclusions. We say that $A, B$ has type (I) if there exists a finite index subgroup $ K\leq  G$ such that $A, B,  K$ has type (I), and similarly for the other types.
	
	When the ambient group is compact, the terms ``tame'' and ``type'' can be applied, simply replacing Proposition \ref{prop:Reduction} with Theorem \ref{thm:LCAInverse} in the definition.
\end{definition}

\begin{lemma}\label{lem:LiftTame}
	Let $G$ be a compact  quotient of $\mb G$ with $\mu$-measurable quotient map $\pi: \mb G\to G$.  If $C, D\subseteq G$ is a subcritical pair (critical pair) in $G$, then $\pi^{-1}(C)$, $ \pi^{-1}(D)$ is a tame subcritical pair (tame critical pair) in $\mb G$.
\end{lemma}
\begin{proof}
	We only prove the assertion for critical pairs.  The assertion for subcritical pairs follows from a similar routine analysis, using Theorem \ref{thm:LCAInverse} in place of Theorem \ref{thm:Satz1}.  Let $m$ denote Haar measure on $G$.
	
	Assume $C, D$ is a critical pair in $G$.  Let $A=\pi^{-1}(C)$, $B=\pi^{-1}(D)$, so the $\mu$-measurability of $\pi$ and Lemma \ref{lem:MeasureIsPreserved} imply $\mu(A)=m(C)$, $\mu(B)=m(D)$, and
	\[\mu(A+B)=m(C+D)<m(C)+m(D)=\mu(A)+\mu(B),\] so $A, B$ is a critical pair.    Theorem \ref{thm:Satz1} guarantees the existence of a compact open (and therefore finite index) subgroup $H\leq G$ such that $C+D=C+D+H$ and $m(C+D)=m(C+H)+m(D+H)-m(H)$.  The quotient $G/H$ is finite, so if $\rho:G\to G/H$ denotes the quotient map, then $\tau:=\rho\circ \pi$ is a $\mu$-measurable homomorphism from $\mb G$ onto the finite group $G/H$. Corollary \ref{cor:FiniteIndexInternal} implies the kernel $\mb H$ of $\tau$ is an internal $\mu$-measurable finite index subgroup of $\mb G$.  Now $\mb H$ is the stabilizer of $A+B$  and the equation $\mu(A+B)=\mu(A+\mb H)+\mu(B+\mb H)-\mu(\mb H)$ follows as $\pi$ preserves $\mu$.   Thus $A, B$ is a tame critical pair in $\mb G$. \end{proof}

Perhaps there is a direct proof of Proposition \ref{prop:LiftSubcritical} imitating the proof of Theorem \ref{thm:Satz1} in \cite{Kneser56}, but we do not pursue this approach.  Instead we lift the result from the compact setting using an argument from \cite{JinUBDInverse}, where an analogue of Theorem \ref{thm:Satz1} is proved for upper Banach density in $\mathbb Z$.  For this argument we require Lemmas \ref{lem:SimCriticalSub} and \ref{lem:TameSuper}.

\begin{proof}[Proof of Proposition \ref{prop:LiftSubcritical}] Consider $\mu$-measurable sets $A, B\subseteq \mb G$ satisfying $\mu_{*}(A+B)<\mu(A)+\mu(B)$. Our goal, in the terminology of Definition \ref{def:TameCritical}, is to prove that $A,B$ is a tame critical pair.  Apply Corollary \ref{cor:ProjectAndPullBack} to find a compact abelian quotient $G$ of $\mb G$ with Haar probability measure $m$, a $\mu$-measure preserving  quotient map $\pi: \mb G\to G$, and sets $C, D\subseteq G$, such that $A\subset_{\mu} A':=\pi^{-1}(C), B\subset_{\mu} B':=\pi^{-1}(D)$,  $A'+B'$ is $\mu$-measurable  and $A'+B'\subset_{\mu}A+B$.  Then
	\begin{align*}
	m(C+D)=\mu(A'+B')&\leq \mu_{*}(A+B) <\mu(A)+\mu(B)\leq \mu(A')+\mu(B')=m(C)+m(D).
	\end{align*}Thus $m(C+D)<m(C)+m(D)$, so $C, D$ is a critical pair in $G$.  Then Theorem \ref{thm:Satz1} implies $C,D$ is a tame critical pair, so Lemma \ref{lem:LiftTame} implies $A', B'$ is a tame critical pair.  Let $A''=A\cap A'$ and $B''=B\cap B'$.  We will prove that the pair $A'', B''$ is tame.  First, observe that the essential containment $A\subset_{\mu} A'$ implies $A''\sim_{\mu} A$, and likewise $B''\sim_{\mu} B$.  We now show that $\mu_{*}(A''+B'')=\mu(A'+B')$. Note that $A''+B''\subseteq A'+B'$, so $\mu_{*}(A''+B'')\leq \mu(A'+B')$.  The reverse inequality follows from Part (ii) of Corollary \ref{cor:ProjectAndPullBack} and the similarities noted above.
	
	We have shown that $\mu_{*}(A''+B'')=\mu(A'+B')$, while the inequality $\mu_{*}(A''+B'')<\mu(A'')+\mu(B'')$ follows from the similarities $A''\sim_{\mu}A$ and $B''\sim_{\mu} B$, the containments $A''\subseteq A$, $B''\subseteq B$, and the fact that $A,B$ is a critical pair.  We may therefore apply Lemma \ref{lem:SimCriticalSub} to conclude that $A'', B''$ is a tame critical pair. Now Lemma \ref{lem:TameSuper} implies $A,B$ is a tame critical pair, completing the proof of Proposition \ref{prop:LiftSubcritical}. \end{proof}

The next definition introduces some  pairs of sets for the proof of Proposition \ref{prop:Reduction}.

\begin{definition}\label{def:ExtendibleReducible}
	A subcritical pair of sets $A, B\subseteq \mb G$ is
	\begin{enumerate}
		\item[$\bullet$] \emph{extendible} if there are $A'\supseteq A$, $B'\supseteq B$ such that $\mu(A')+\mu(B')>\mu(A)+\mu(B)$, while $\mu_{*}(A'+B')=\mu_{*}(A+B)$,
		
		\smallskip
		
		\item[$\bullet$] \emph{reducible} if there are sets $A''\subseteq A$, $B''\subseteq B$ such that $\mu(A'')=\mu(A)$, $\mu(B'')=\mu(B)$, and $\mu_{*}(A''+B'')<\mu_{*}(A+B)$.  In this case we say $A, B$ \emph{reduces to} $A'', B''$.
	\end{enumerate}
	If the pair $A, B$ is not extendible, we say it is \emph{nonextendible}, and if it is not reducible, we say that it is \emph{irreducible}.  \end{definition}

In subsequent proofs it will be convenient to assume that a pair $A,B$ is nonextendible and irreducible. The following three lemmas allow us to do so, by proving that if a subcritical pair is extendible or reducible, then it is a tame subcritical pair of type (I) or type (II) (in the terminology of Definition \ref{def:Tame}).

\begin{lemma}\label{lem:Extendible}
	If $A, B\subseteq \mb G$ satisfy $\mu(A), \mu(B)>0$, $\mu_{*}(A+B)=\mu(A)+\mu(B)$, and $A,B$ is extendible, then $A, B$ has type (I).
\end{lemma}

\begin{proof}
	If $A, B$ is extendible, then there exist sets $A', B'\subseteq \mb G$ such that $A\subseteq A'$, $B\subseteq B'$, $\mu(A')+\mu(B')> \mu(A)+\mu(B)$, and $\mu_{*}(A'+B')=\mu_{*}(A+B)$.  Thus $\mu_{*}(A'+B')<\mu(A')+\mu(B')$.  Then Proposition \ref{prop:LiftSubcritical} implies that there is an internal $\mu$-measurable finite index subgroup $\mb H\leq \mb G$ such that $A'+B'=A'+B'+\mb H$.  The assumption $\mu_{*}(A'+B')=\mu_{*}(A+B)$ and the containment $A+B\subseteq A'+B'$ now imply $A+B\sim_{\mu} A'+B'+\mb H$, so $A+B\sim_{\mu} A+B+\mb H$, meaning $A, B$ has type (I).
\end{proof}

\begin{lemma}\label{lem:RedicubleQP}
	Suppose that $A,B\subseteq G$ satisfies $\mu_{*}(A+B)=\mu(A)+\mu(B)$, and $A,B$ reduces to a tame critical pair $A', B'$.  Let $\mb H=H(A'+B')$, and assume that $A+B\nsim_{\mu} A+B+\mb H$.  Then $A$ and $B$ have quasi-periodic decompositions with respect to $\mb H$ where at least one of $\mu(A_{0})=0$ or $\mu(B_{0})=0$, and
	\begin{equation}\label{eqn:EndSum}
	\mu_{*}(A_{0}+B_{0})=\mu(A_{0})+\mu(B_{0}).
	\end{equation}
\end{lemma}

\begin{proof}  We consider cases based on $\mu_{*}(A+B')$ and $\mu_{*}(A'+B)$.
	
	\noindent \textbf{Case 1:}  $\mu_{*}(A+B')<\mu_{*}(A+B)$ or $\mu_{*}(A'+B)<\mu_{*}(A+B)$.  By symmetry we need only consider the first inequality. In this case Proposition \ref{prop:LiftSubcritical} implies  $A,B'$ is a tame critical pair.   Lemma \ref{lem:LargeAB} and Lemma \ref{lem:Overflow} now imply $A+B'=A'+B'$, so $H(A+B')=H(A'+B')$.  Part (iv) of Lemma \ref{lem:Pop1} then implies the desired conclusion.
	
	\smallskip
	
	\noindent \textbf{Case 2:} Neither $\mu_{*}(A+B')<\mu_{*}(A+B)$ nor $\mu_{*}(A'+B)<\mu_{*}(A+B)$.  In this case we have $\mu_{*}(A'+B)=\mu_{*}(A+B')=\mu_{*}(A+B)$.  We apply Lemma \ref{lem:Pop1} with $A'$ in place of $A$, and thus obtain a quasi-periodic decomposition $B=B_{1}\cup B_{0}$ with respect to $\mb H$ where $\mu(B_{0})=0$ and $B_{1}=B\cap (B'+\mb H)$.    Similarly we apply Lemma \ref{lem:Pop1} with $B'$ in place of $A$ and $A$ in place of $B$.  We thereby obtain a quasi-periodic decomposition $A=A_{1}\cup A_{0}$ with respect to $\mb H$ such that $\mu(A_{0})=0$.  The assumption that $A+B\nsim_{\mu} A+B+\mb H$ then implies $A_{0}+B_{0}+\mb H$ is a unique expression element of $A+B+\mb H$.  In particular, $A_{0}+B_{0}\subseteq D:=(A+B)\setminus (A+B_{1})$.  Now our assumption that $\mu_{*}(A+B)=\mu_{*}(A+B')$ implies $\mu_{*}(A+B_{1})=\mu_{*}(A+B)$.  Then $\mu_{*}(D)=0$, so that $\mu_{*}(A_{0}+B_{0})=0$, which implies (\ref{eqn:EndSum}). \end{proof}

\begin{lemma}\label{lem:Reducible}
	If $A, B\subseteq \mb G$ has $\mu_{*}(A+B)=\mu(A)+\mu(B)$ and $A,B$ is reducible and does not satisfy conclusion (I) in Proposition \ref{prop:Reduction}, then $A, B$ satisfies conclusion (II) of Proposition \ref{prop:Reduction} with at least one of $\mu(A_{0})=0$ or $\mu(B_{0})=0$.
\end{lemma}

\begin{proof}
	Suppose the pair $A, B$ is as above and $A, B$ reduces to $A',B'$, where $\mu_{*}(A'+B')<\mu(A)+\mu(B)$.  Let $\mb H=H(A'+B')$, so that $\mb H$ is an internal $\mu$-measurable  finite index subgroup of $\mb G$, by Proposition \ref{prop:LiftSubcritical}.  Our hypothesis implies $A+B\nsim_{\mu} A+B+\mb H$, so  Lemma \ref{lem:RedicubleQP}  yields the desired conclusion.
\end{proof}
The next lemma uses notation from Definition \ref{def:Movers}.

\begin{lemma}\label{lem:ExtendMove}
	Let $A, B\subseteq \mb G$ be $\mu$-measurable sets having $\mu(A),\mu(B)>0$ such that $\mu_{*}(A+B)\leq \mu(A)+\mu(B)$, and assume the stabilizer $\mb H$ of $A+B$ is a $\mu$-measurable finite index subgroup of $\mb G$.  Then either $\bar{A}_{B}=A+\mb H$, or $\bar{A}_{B}=(A+\mb H)\cup (\mb g+\mb H)$ for some $\mb g\in \mb G\setminus (A+\mb H)$.
\end{lemma}

\begin{proof}
	Since $\mb H$ is the stabilizer of $A+B$, we have $A+\mb H\subseteq \bar{A}_{B}$, and we aim to establish the reverse containment. If $\bar{A}_{B}\neq A+\mb H$, then there is a  $\mb g\in \mb G\setminus (A+\mb H)$ with $\mb g+B\subseteq A+B$. Setting $C:=A\cup (\mb g+\mb H)$, we have $\mu(C)=\mu(A)+\mu(\mb H)>\mu(A)$, while  $C+B= A+B$, so $\mu(C+B)=\mu(A+B)< \mu(C)+\mu(B)$.  Furthermore $H(C+B)=H(A+B)$, so Proposition \ref{prop:LiftSubcritical} implies the hypotheses of Lemma \ref{lem:PushSubCritical} are satisfied with $C$ in place of $A$. Then $\bar{C}_{B}=C+\mb H$.  Now $\bar{C}_{B}=\bar{A}_{B}$, so $\bar{A}_{B}=C+\mb H=(A+\mb H)\cup (\mb g+\mb H)$.\end{proof}

The next lemma deals with a case arising in the proof of Proposition \ref{prop:Reduction}.

\begin{lemma}\label{lem:LongCase}
	Suppose $A, B\subseteq \mb G$ have $\mu(A), \mu(B)>0$, and $\mu_{*}(A+B)=\mu(A)+\mu(B)$.   Furthermore assume there are sets $A'\subseteq A$ and $B'\subseteq B$ and an internal $\mu$-measurable finite index subgroup $\mb H\leq \mb G$ such that
	\begin{enumerate}
		\item[$\bullet$]   $A'\sim_{\mu}A, B'\sim_{\mu} B$, and $\mu_{*}(A'+B')=\mu(A')+\mu(B')$ (so $\mu_{*}(A'+B')=\mu_{*}(A+B)$),
		
		\item[$\bullet$] $A'$ and $B'$ are both $\mb H$-solid (Definition \ref{def:Solid}),
		
		\item[$\bullet$]   $A'+B'\sim_{\mu}A'+B'+\mb H$, and $\mb H$ is the stabilizer of $A'+B'+\mb H$.
		
	\end{enumerate}
	Then $A, B$ satisfies  conclusion (I) or conclusion (II) of Proposition \ref{prop:Reduction}.
\end{lemma}

\begin{proof}
	First note that for all $a\in A$, $a+B'\subset_{\mu} A'+B'$.  This follows from the assumption $\mu_{*}(A'+B')=\mu_{*}(A+B)$ and the containment $A+B'\subseteq A+B$.  Thus $A\subseteq \bar{A'}_{B',\mu}$ in the notaion of Definition \ref{def:Movers}. Furthermore, the $\mb H$-solidity of $B'$ implies, via Lemma \ref{lem:SolidMove}, that $\bar{A'}_{B',\mu}=\bar{A'}_{B'+\mb H}$.  We consider two cases.
	
	\noindent \textbf{Case 1:} $\mu(A'+\mb H)+\mu(B'+\mb H)> \mu(A')+\mu(B')$.   In this case  we will prove that $A\subseteq A'+\mb H$ and $B\subseteq B'+\mb H$.   Let $C=A'+\mb H$, $D=B'+\mb H$.  Then $\mu(C+D)<\mu(C)+\mu(D)$, so we may apply Proposition \ref{prop:LiftSubcritical} to the pair $C,D$.  Our hypothesis on $A', B'$ implies that $\mb H$ is the stabilizer of $C+D$.  Then Lemma \ref{lem:PushSubCritical} implies $\bar{C}_{D}=C+\mb H$.  In other words: if $a+B'+\mb H\subseteq A'+B'+\mb H$, then $a\in A'+\mb H$.  So we have shown that $\bar{A'}_{B'+\mb H}=A'+\mb H$, and the previously established identities and containments imply $A\subseteq A'+\mb H$.  The containment $B\subseteq B'+\mb H$ follows by symmetry.
	
	Now $A\subseteq C$ and $B\subseteq D$, while $\mu_{*}(C+D)=\mu_{*}(A+B)$.  The identity $C+D=C+D+\mb H$ and containment $A+B\subseteq C+D$ then imply $A+B\sim_{\mu} A+B+\mb H$.  Thus $A, B$ satisfy conclusion (I) of Proposition \ref{prop:Reduction}.

	\noindent \textbf{Case 2:} $A'+\mb H\sim_{\mu}A'$ and $B'+\mb H\sim_{\mu}B'$.  In this case Lemma \ref{lem:LargeAB} implies $A'+B'=A'+B'+\mb H$, so $\mb H$ is the stabilizer of $A'+B'$. Then  $\bar{A'}_{B',\mu}=\bar{A'}_{B'}=\bar{A'}_{B'+\mb H}$, so that $A\subseteq \bar{A'}_{B'+\mb H}$. Lemma \ref{lem:ExtendMove} implies $\bar{A'}_{B'+\mb H}$ is either $A'+\mb H$, or $(A'+\mb H)\cup(\mb g+\mb H)$ for some $\mb g\notin A'+\mb H$.  We therefore have the following disjunction: \begin{equation}\label{eqn:Alt1}A \subseteq A'+\mb H,\end{equation}
	or
	\begin{equation}\label{eqn:Alt2} A\subseteq (A'+\mb H)\cup(\mb g+\mb H) \text{ for some } \mb g\notin A'+\mb H.\end{equation}
	If one of the containments $A\subseteq A'+\mb H$, $B\subseteq B'+\mb H$, or $A+B\subseteq A'+B'+\mb H$ holds, then we are done:  the similarities $A\sim_{\mu} A'$ and $B\sim_{\mu} B'$  would imply $A+B\sim_{\mu} A+B+\mb H$, so that $A$ and $B$ satisfy Conclusion (I) of Proposition \ref{prop:Reduction}.  We therefore assume these containments fail, so that (\ref{eqn:Alt2}) holds, and likewise $B\subseteq (B'+\mb H)\cup (\mb g+\mb H)$ for some $\mb g \notin B'+\mb H$.   We aim to produce quasi-periodic decompositions $A=A_{1}\cup A_{0}$, $B=B_{1}\cup B_{0}$ satisfying conclusion (II) of Proposition \ref{prop:Reduction}. Based on the containment (\ref{eqn:Alt2}), we set $A_{1}:=A\cap (A'+\mb H)$ and $A_{0}:=A\setminus A_{1}$.  Then $A_{1}\cup A_{0}$ is a quasi-periodic decomposition of $A$.  By symmetry we produce a quasi-periodic decomposition of $B$ as $B_{1}\cup B_{0}$, where $B_{0}=B\setminus (B'+\mb H)$.  Since we are assuming $A+B\nsubseteq A'+B+\mb H$, we have that $A_{0}+B_{0}\subseteq (A+B)\setminus (A'+B')$.  The hypothesis $\mu_{*}(A+B)=\mu_{*}(A'+B')$ then implies $\mu_{*}(A_{0}+B_{0})=0$, whence $\mu(A_{0})=\mu(B_{0})=0$.  Thus $A=A_{1}\cup A_{0}$ and $B=B_{1}\cup B_{0}$ are decompositions of $A$ and $B$ satisfying conclusion (II) of Proposition \ref{prop:Reduction}. \end{proof}

The next lemma forms the majority of our proof of Proposition \ref{prop:Reduction}.  It says that if $\mu_{*}(A+B)=\mu(A)+\mu(B)$, $\mu_{*}(A'+B')=\mu(A')+\mu(B')$, $A'\sim_{\mu}A$, $B'\sim_{\mu}B$, and the pair $A',B'$ satisfies the conclusion of Proposition \ref{prop:Reduction}, then so does  the pair $A,B$.

\begin{lemma}\label{lem:TameSimilar}
	Let $A, B\subseteq \mb G$ have $\mu(A)$, $\mu(B)>0$ and  \[\mu_{*}(A+B)=\mu(A)+\mu(B).\] If $A', B'\subseteq \mb G$ is a tame subcritical pair such that $A\sim_{\mu}A'$ and $B\sim_{\mu}B'$,  then $A, B$ is also a tame subcritical pair.
\end{lemma}
Note: in Lemma \ref{lem:TameSimilar} we do not assume $A'+B'\sim_{\mu} A+B$.

\begin{proof}
	By Lemmas \ref{lem:Extendible} and \ref{lem:Reducible}, we may assume $A,B$ is irreducible and nonextendible.
	
	Since $A',B'$ is a tame subcritical pair, there is a $\mu$-measurable finite index subgroup $\mb K'\leq \mb G$ such that $A', B', \mb K'$ satisfies one of the conclusions of Proposition \ref{prop:Reduction}.
	Our goal is to prove that there is a $\mu$-measurable finite index subgroup $\mb K$ (not necessarily equal to $\mb K'$) such that the triple $A, B, \mb K$ satisfies one of the conclusions of Proposition \ref{prop:Reduction}.  In the terminology of Definition \ref{def:Tame}, we want to show that $A,B, \mb K$ has type (I), type (II), or type (III).

	The additional hypothesis that $A,B$ is irreducible implies
	\begin{equation}\label{eqn:StarTriangle}
	\mu_{*}((A'+B')\triangle(A+B))=0
	\end{equation}
	since otherwise setting $A''=A\cap A'$ and $B''=B\cap B'$ would yield a pair with $\mu_{*}(A''+B'')<\mu(A)+\mu(B)$ and $A\sim_{\mu} A''$, $B\sim_{\mu} B''$.
	
	We consider three cases.

	\noindent \textbf{Case 1: $A'\subseteq A, B'\subseteq B$.}   We consider several subcases based on the type of $A', B', \mb K'$.

	\noindent \textbf{Subcase 1.1:} $A', B', \mb K'$ has type (I).  Here $A'+B'\sim_{\mu}A'+B'+\mb K'$. Then the stabilizer $\mb H$ of $A'+B'+\mb K'$ contains $\mb K'$, which is an internal $\mu$-measurable finite index subgroup of $\mb G$.  It follows that $\mb H$ is also such a subgroup, and $A'+B'\sim_{\mu} A'+B'+\mb H$.  We will derive one of two possibilities: either $A+B\sim_{\mu} A+B+\mb H$, or $A$ and $B$ have quasi-periodic decompositions with respect to $\mb H$ such that $A, B$, and $\mb H$ satisfy conclusion (II) of Proposition \ref{prop:Reduction}.  Note that $A', B'$ and $\mb H$ satisfy all of the hypotheses of Lemma \ref{lem:LongCase}, except possibly the assumption that $A'$ and $B'$ are $\mb H$-solid.  Let $A''\subseteq A'$ and $B''\subseteq B'$ be $\mb H$-solid sets with $\mu(A'')=\mu(A')$ and $\mu(B'')=\mu(B')$.  We will show that $A'', B''$ and $\mb H$ do satisfy the hypothesis of Lemma \ref{lem:LongCase}.  To see this, note that the irreducibility of $A,B$ and the containments $A''\subseteq A$, $B''\subseteq B$ imply $\mu_{*}(A''+B'')=\mu_{*}(A+B)$, and therefore $\mu_{*}(A''+B'')=\mu_{*}(A'+B')$.  It follows that $A''+B''\sim_{\mu} A''+B''+\mb H$, and we have verified the hypotheses of Lemma \ref{lem:LongCase} with $A''$ and $B''$ in place of $A'$ and $B'$.  We may now conclude that $A,B, \mb H$ has type (I) or type (II).

	\noindent \textbf{Subcase 1.2:} $A', B', \mb K'$ has type (II).    In this case $A'+B'\nsim_{\mu} A'+B'+\mb K'$, and there are quasi-periodic decompositions $A'=A_{1}'\cup A_{0}'$, $B'=B_{1}'\cup B_{0}'$  with respect to $\mb K'$ such that $\mu_{*}(A_{0}'+B_{0}')=\mu(A_{0}')+\mu(B_{0}')$, and at least one of $A_{1}', B_{1}'$ is nonempty.  Without loss of generality we assume $A_{1}'\neq \varnothing$. We consider two further subcases.

	\noindent \textbf{Subcase 1.2.1:}  $\mu(A_{0}')>0$ and $\mu(B_{0}')>0$.  Here we will show that $A, B, \mb K'$ has type (II).  We may assume $A', B'$ is nonextendible, as otherwise Lemma \ref{lem:Extendible} implies $A'+B'\sim_{\mu} A'+B'+\mb H$ for some $\mu$-measurable finite index subgroup $\mb H$ and we could argue as in Subcase 1.1.   Define $A_{0}:=A\cap (A_{0}'+\mb K')$ and $A_{1}:=A\setminus A_{0}$.  We will show that $A_{1}\sim_{\mu} A_{1}+\mb K'$.  To do so, we will show that for all $a\in A_{1}$, $a+\mb K'\subset_{\mu} A_{1}$. Fix $a\in A_{1}$.  First observe that irreducibility of $A, B$ implies $a+B'\subset_{\mu} A'+B'$.
	\begin{claim}
		With the preceding choice of $a$, we have $a+B'\subset_{\mu} (A_{1}'+B')\cup (A'+B_{1}')$.
	\end{claim}
	To prove the Claim, assume otherwise. Then $(a+B')\cap (A_{0}'+B_{0}'+\mb K')\neq \varnothing$, so there is a $b\in B'$ such that $a+b\in A_{0}+B_{0}+\mb K'$.   Since $a\notin A_{0}+\mb K'$, $b\notin B_{0}'$, so $b\in B_{1}'$.  Since $B_{1}'\sim_{\mu} B_{1}'+\mb K'$, we then have $A_{0}'+B_{0}'+\mb K'\subset_{\mu} a+B'$, and the observation preceding the Claim then implies $A_{0}'+B_{0}'+\mb K'\subset_{\mu} A'+B'$.  But then $A'+B'\sim_{\mu} A'+B'+\mb K'$, contradicting our present hypothesis on $A',B'$ and establishing the Claim.
	
	The similarities $A_{1}'\sim_{\mu} A_{1}'+\mb K'$ and $B_{1}'\sim_{\mu} B_{1}+\mb K'$, the $\mb K'$-solidity of $B'$, and the Claim yield the essential containment $a+B'+\mb K'\subset_{\mu} A'+B'$.  Then $a+\mb K'\subset_{\mu} A'$ by the nonextendibility of $A', B'$: otherwise setting $C:=A'\cup (a+\mb K')$ we would have $\mu(C)>\mu(A')$ and $\mu_{*}(C+B')=\mu_{*}(A'+B')$.
	
	We have shown that for all $a\in A_{1}$, $a+\mb K'\subset_{\mu} A'$.  The similarity $A\sim_{\mu} A'$ then implies $A_{1}\sim_{\mu} A_{1}+\mb K'$.  Likewise we obtain $B_{1}\sim_{\mu} B_{1}+\mb K'$.
	
	We have shown that $A_{1}\cup A_{0}$ and $B_{1}\cup B_{0}$ are quasi-periodic decompositions of $A$ and $B$ with respect to $\mb K'$.  To show $\mu_{*}(A_{0}+B_{0})=\mu(A_{0})+\mu(B_{0})$, observe that $A_{0}'+B_{0}'$ is disjoint from $(A_{1}'+B')\cup (A'+B_{1}')$; otherwise we would have $A'+B'\sim_{\mu} A'+B'+\mb K'$. Irreducibility of $A,B$ implies $\mu_{*}(A+B)=\mu_{*}(A'+B')$, so the desired equation follows from the similarities $A_{0}\sim_{\mu} A_{0}'$, $B_{0}\sim_{\mu} B_{0}'$ and the assumption that $\mu_{*}(A_{0}'+B_{0}') = \mu(A_{0}')+\mu(B_{0}')$.

	\noindent \textbf{Subcase 1.2.2:}  $\mu(A_{0}')=0$ or $\mu(B_{0}')=0$.   Assume, without loss of generality, that $\mu(A_{0}')=0$.  In this case let $A'':=A'\setminus A_{0}'$.   Then $A''\sim_{\mu} A''+\mb K'$, so $A''+B'\sim_{\mu}A''+B'+\mb K'$. Irreducibility of $A, B$ implies $\mu_{*}(A''+B')=\mu_{*}(A+B)=\mu(A'')+\mu(B')$, so  Subcase 1.1 applies with the pair $A'', B'$ in place of the pair $A' ,B'$.  We conclude that $A, B$ has type (I) or type (II).

	\noindent \textbf{Subcase 1.3:} $A', B', \mb K'$ has type (III).  By translating $A'$ and $B'$, we may assume $A'$ and $B'$ are contained in parallel Bohr intervals $\tilde{I}, \tilde{J}\subseteq \mb K'$, respectively, with $\mu(\tilde{I})=\mu(A')$, $\mu(\tilde{J})=\mu(B')$.  Thus $A\sim_{\mu} \tilde{I}$ and $B\sim_{\mu} \tilde{J}$, whence Corollary \ref{cor:EssentialMoveInterval} implies $A, B$ has type (III).

	\noindent \textbf{Case 2: $A\subseteq A', B\subseteq B'$.}     In this case, if $A', B', \mb K'$ has type (I) or type (III), then we immediately conclude that $A, B$ also has type (I) or type (III).  If $A', B', \mb K'$ has type (II), the containments $A+B\subseteq A'+B'\subseteq A'+B'+\mb K'$ and the similarities $A\sim_{\mu} A'$, $B\sim_{\mu} B'$ easily imply that $A, B$ has type (I) or type (II).

	\noindent \textbf{Case 3:} the general case.   Let $A'':= A\cap A'$, $B'':=B\cap B'$, so that $A\sim_{\mu} A''$, $B\sim_{\mu} B''$.  If $\mu_{*}(A''+B'')<\mu(A'')+\mu(B'')$, then $A, B$ is reducible, contrary to our assumptions.  If $\mu_{*}(A''+B'')=\mu(A'')+\mu(B'')$, then we may apply the result from Case 2 with $A'', B''$ in place of $A, B$ to conclude that $A'', B''$ is a tame pair.  Then Case 1 with $A'', B''$ in place of $A', B'$ implies  $A, B$ is a tame pair. \end{proof}

\begin{proof}[Proof of Proposition \ref{prop:Reduction}]  With the terminology introduced in Definition \ref{def:Tame}, Proposition \ref{prop:Reduction} can be stated as ``If $A, B\subseteq \mb G$ satisfy $\mu_{*}(A+B)= \mu(A)+\mu(B)$, then $A, B$ is a tame pair''.
	
	Assume $A, B\subseteq \mb G$ satisfy $\mu(A), \mu(B)>0$ and $\mu_{*}(A+B)=\mu(A)+\mu(B)$. By Lemmas \ref{lem:Extendible} and \ref{lem:Reducible}, we may assume that $A,B$ is nonextendible and irreducible.  Apply Corollary \ref{cor:ProjectAndPullBack} to find a compact metrizable quotient $G$ of $\mb G$ with $\mu$-measure preserving quotient map $\pi:\mb G\to G$ and Borel sets $C, D\subseteq G$ such that $C+D$ is Borel while $A':=\pi^{-1}(C)$, $B':=\pi^{-1}(D)$ satisfy $A\subset_{\mu} A'$, $B\subset_{\mu} B'$, and $A'+B'\subset_{\mu} A+B$. Let $A''=A\cap A'$ and $B''=B\cap B'$, so that $A''\sim_{\mu} A$ and $B''\sim_{\mu} B$.  The irreducibility of $A,B$ then implies $\mu_{*}(A''+B'')=\mu_{*}(A+B)$, and hence $\mu_{*}(A''+B'')=\mu(A'')+\mu(B'')$. We now consider two cases.
	
	\noindent \textbf{Case 1:} $A''\nsim_{\mu} A'$ or $ B''\nsim_{\mu} B'$. In this case $A'', B''$ is extendible, as $A''\subseteq A'$, $B''\subseteq B'$, and $\mu(A'+B')\leq \mu_{*}(A+B)=\mu_{*}(A''+B'')$. Then Lemma \ref{lem:Extendible} implies $A'',B''$ is a tame pair.  Lemma \ref{lem:TameSimilar} and the similarities $A''\sim_{\mu} A$, $B''\sim_{\mu} B$ then imply $A, B$ is tame.
	
	\noindent \textbf{Case 2:}  $A''\sim_{\mu}A'$ and $B''\sim_{\mu}B'$.  In this case we also have $A\sim_{\mu} A'$ and $B\sim_{\mu} B'$.  Now the containments $A''+B''\subseteq A'+B'\subset_{\mu} A+B$ and irreducibility of $A,B$ imply $\mu(A'+B') = \mu_{*}(A+B)=\mu(A)+\mu(B)$, and therefore $\mu(A'+B')=\mu(A')+\mu(B').$  Then $m(C+D)=m(C)+m(D)$, and Theorem \ref{thm:LCAInverse} implies $C,D$ is a tame pair.  Lemma \ref{lem:LiftTame} now implies $A', B'$ is a tame pair, so we may apply Lemma \ref{lem:TameSimilar} to conclude that $A, B$ is a tame pair. \end{proof}

The equation $\mu_{*}(A_{0}+B_{0})=\mu(A_{0})+\mu(B_{0})$ in conclusion (II) of Proposition \ref{prop:Reduction}  allows one to apply Proposition \ref{prop:Reduction} to $A_{0}, B_{0}$ and obtain Corollary \ref{cor:Iterate}.  The next lemma is required to perform the iteration.  We use the term ``type'' as in Definition \ref{def:Tame}.

\begin{lemma}\label{lem:IterateQP}
	Suppose $\mu_{*}(A+B)=\mu(A)+\mu(B)$ and conclusion (II) holds in Proposition \ref{prop:Reduction} with some subgroup $\mb K\leq \mb G$ and quasi-periodic decompositions $A=A_{1}\cup A_{0}$, $B=B_{1}\cup B_{0}$.  Suppose further that $A_{0}$ and $B_{0}$ are $\mu$-nonnull, and that $A_{0},B_{0}$ has type (II) with respect to a subgroup $\mb K'\leq \mb G$, with quasi-periodic decompositions $A_{0}=A_{1}'\cup A_{0}'$, $B_{0}= B_{1}'\cup B_{0}'$.  Then $A, B, \mb K'$ has type (II), with quasi-periodic decompositions $A_{1}''\cup A_{0}''$, where $A_{0}''=A_{0}'$, and $A_{1}'':= A\setminus A_{0}'$, and similarly for $B$.
\end{lemma}

\begin{proof}
	Since $A_{0}$ and $B_{0}$, and therefore $A_{0}'$ and $B_{0}'$, are each contained in cosets of $\mb K$, and $A_{1}'\sim_{\mu} A_{1}'+\mb K'$, $B_{1}'\sim_{\mu} B_{1}'+\mb K'$, we have $\mb K'\subseteq \mb K$.  It follows that $A_{1}\sim_{\mu} A_{1}+\mb K'$, $B_{1}\sim_{\mu} B_{1}+\mb K'$.  Consequently, $A_{1}''$ and $B_{1}''$, as defined in the statement of the lemma, satisfy the same similarities.  To see that $A+B\nsim_{\mu} A+B+\mb K'$, note that our hypothesis implies $A_{0}+B_{0}\nsim_{\mu} A_{0}+B_{0}+\mb K'$.  Remark \ref{rem:UEE} implies $A+B$ can be written as $C\cup (A_{0}+B_{0})$, where $C=(A_{1}+B)\cup(A+B_{1})$ is disjoint from $A_{0}+B_{0}+\mb K$, so the last non-similarity implies $A+B\nsim_{\mu} A+B+\mb K'$.
\end{proof}

We now iterate Proposition \ref{prop:Reduction} and obtain a more detailed conclusion with the same hypothesis.

\begin{corollary}\label{cor:Iterate}
	If $A, B\subseteq \mb G$ are $\mu$-measurable sets such that $\mu(A)>0$, $\mu(B)>0$, and $\mu_{*}(A+B)=\mu(A)+\mu(B)$, then at least one of the following is true.
	\begin{enumerate}
		\item[(a)]  $A+B\sim_{\mu} A+B+\mb K$ for some $\mu$-measurable finite index internal subgroup $\mb K\leq \mb G$.
		
		\smallskip
		
		\item[(b)] $A, B$ has type (III).
		
		\smallskip
		
		\item[(c)] For all $\eta>0$, there is a $\mu$-measurable finite index internal subgroup $\mb K\leq \mb G$ with $\mu(\mb K)<\eta$ such that $A, B,\mb K$ has type (II).
		
		\smallskip
		
		\item[(d)] $A, B$ has type (II) with quasi-periodic decompositions $A=A_{1}\cup A_{0}$, $B=B_{1}\cup B_{0}$ where at least one of $\mu(A_{0})=0$ or $\mu(B_{0})=0$.
		
		\smallskip
		
		\item[(e)]  $A, B$  has type (II) with quasi-periodic decompositions $A=A_{1}\cup A_{0}$, $B=B_{1}\cup B_{0}$ and  $A_{0}, B_{0}$ has type (III) with $\mu(A_{0})>0$ and $\mu(B_{0})>0$. \end{enumerate}
\end{corollary}

\begin{proof}
	Assuming $A$ and $B$ are as in the hypothesis, we apply Proposition \ref{prop:Reduction} and consider the possible conclusions.  In conclusions (I) or (III) of the proposition, we have conclusion (a) or (b), respectively, of the present corollary.  We therefore assume $A,B$ satisfies  (II) in Proposition \ref{prop:Reduction}.   Let $\eta_{0}$ be the infimum of all the numbers $\mu(\mb K)$ such that $A, B$ and $\mb K$ satisfy (II) in Proposition \ref{prop:Reduction}.  If $\eta_{0}=0$, then conclusion (c) holds here.  If $\eta_{0}>0$ we choose a subgroup $\mb K\leq \mb G$ with $\mu(\mb K)<2\eta_{0}$ such that $A,B,\mb K$ has type (II) and fix the corresponding  quasi-periodic decompositions $A=A_{1}\cup A_{0}$ and $B=B_{1}\cup B_{0}$, where $\mu_{*}(A_{0}+B_{0})=\mu(A_{0})+\mu(B_{0})$.    If one of $A_{0}$ or $B_{0}$ has measure $0$, then $A, B$ satisfies conclusion (d) here. Otherwise, we apply Proposition \ref{prop:Reduction} to $A_{0},B_{0}$.  If this pair satisfies conclusion (I) or (III) in the proposition, then $A,B$ satisfies conclusion (a) or (e), respectively, in the present corollary.  If $A_{0}, B_{0}$ has type (II), then we will derive a contradiction: in this case Lemma \ref{lem:IterateQP} implies $A,B,\mb K'$ has type (II), with $\mb K'$ a proper subgroup of $\mb K$.  Thus $\mu(\mb K')\leq \frac{1}{2}\mu(\mb K)<\eta_{0}$, contradicting the definition of $\eta_{0}$.
\end{proof}

\section{Proof of Theorem \ref{thm:Precise}}\label{sec:PreciseProof}

Suppose, to get a contradiction, that Theorem \ref{thm:Precise} fails for a given $\varepsilon>0$.  Then for all  $n\in \mathbb N$, there exists a compact abelian group $G_{n}$ with Haar probability measure $m_{n}$, inner Haar measure $m_{n*}$, and sets $A_{n}, B_{n} \subseteq G_{n}$ such that $m_{n}(A_{n}), m_{n}(B_{n})>\varepsilon$,
\begin{equation}\label{eqn:OneOverN}
m_{n*}(A_{n}+B_{n})\leq m_{n}(A_{n})+m_{n}(B_{n})+\tfrac{1}{n},
\end{equation}
while the pair $A_{n},B_{n}$ satisfies none of the conclusions (I)-(III) of Theorem \ref{thm:Precise} with a subgroup $K=K_{n}\leq G_{n}$ having index at most $n$.   We will derive a contradiction by showing that for some $n$ (in fact for infinitely many $n$) the pair $A_{n}, B_{n}$ satisfies the conclusion of Theorem \ref{thm:Precise} with a subgroup $K=K_{n}$ having index less than $n$.

Let $\mathcal U$ be a nonprincipal ultrafilter on $\mathbb N$. Form the ultraproduct $\mb G:=\prod_{n\to \mathcal U} G_{n}$ and the corresponding Loeb measure $\mu$ and consider the internal sets $\mb A:=\prod_{n\to \mathcal U} A_{n}$, $\mb B:=\prod_{n\to \mathcal U} B_{n}$ with $\mu(\mb A)=\lim_{n\to \mathcal U} m_{n}(A_{n})\geq \varepsilon$, $\mu(\mb B)=\lim_{n\to\mathcal U} m_{n}(B_{n})\geq \varepsilon$, and $\mu_{*}(\mb A+\mb B)=\lim_{n\to \mathcal U} m_{n*}(A_{n}+B_{n})$ by Lemma \ref{lem:InnerLoeb}.  Together with Inequality (\ref{eqn:OneOverN}) this implies $\mu_{*}(\mb A+\mb B) \leq \mu(\mb A)+\mu(\mb B)$.  We consider several cases based on the structure of $\mb A$ and $\mb B$.

\noindent \textbf{Case 1: $\mu(\mb A+\mb B)<\mu(\mb A)+\mu(\mb B)$.}   In this case Proposition \ref{prop:LiftSubcritical} implies that the stabilizer of $\mb A+\mb B$  is an internal $\mu$-measurable finite index subgroup $\mb K = \prod_{n\to \mathcal U} K_{n} \leq \mb G$ satisfying $\mb A+\mb B=\mb A+\mb B+\mb K$ and
\begin{equation}\label{eqn:2.1K}
\mu(\mb A+\mb B)=\mu(\mb A+\mb K)+\mu(\mb B+\mb K)-\mu(\mb K).
\end{equation}
Then Lemma \ref{lem:StabilizeInLimit} implies $A_{n}+B_{n}=A_{n}+B_{n}+K_{n}$ and \[m_{n}(A_{n}+B_{n})=m_{n}(A_{n}+K_{n})+m_{n}(B_{n}+K_{n})-m_{n}(K_{n}) \quad \text {for } \mathcal U\text{-many }n,\] so $A_{n}, B_{n}, K_{n}$ satisfies conclusion (I) of Theorem \ref{thm:Precise} for $\mathcal U$-many $n$, where the index of $K_{n}$ is $1/\mu(\mb K)<n$.

\noindent \textbf{Case 2: $\mu(\mb A+\mb B)=\mu(\mb A)+\mu(\mb B)$.}  In this case apply Proposition \ref{prop:Reduction} to the pair $\mb A,\mb B$ and fix the finite index internal subgroup $\mb K=\prod_{n\to \mathcal U} K_{n}\leq \mb G$ therein. We now consider subcases based on the type of $\mb A$, $\mb B$, $\mb K$ (Definition \ref{def:Tame}).  In each subcase we will find that for $\mathcal U$-many $n$, the triple $A_{n}, B_{n}, K_{n}$ satisfies the conclusion of Theorem \ref{thm:Precise}.  Note that by Corollary \ref{cor:FiniteIndexInternal}, $K_{n}$ has index $1/\mu(\mb K)$ in $G_{n}$, for $\mathcal U$-many $n$.

\noindent \textbf{Subcase 2.1:}  $\mb A,\mb B, \mb K$ has type (I), meaning $\mb A+\mb B\sim_{\mu} \mb A+\mb B+\mb K$.  We will show that $A_{n}, B_{n}, K_{n}$ satisfy conclusion (I) of Theorem \ref{thm:Precise} for $\mathcal U$-many $n$.  In this case, Lemma \ref{lem:FiniteIsomorphism} implies that for $\mathcal U$-many $n$, $m_{n}(A+B+K_{n})- m_{n*}(A_{n}+B_{n})<(1-\varepsilon)m_{n}(K_{n})$, meaning $A_{n}+B_{n}$ is $\varepsilon$-periodic. Then the assumption $\mu_{*}(\mb A+\mb B)=\mu(\mb A)+\mu(\mb B)$ implies $\mu(\mb A+\mb B+\mb K)\leq \mu(\mb A+\mb K)+\mu(\mb B+\mb K)$, so Lemma \ref{lem:FiniteIsomorphism} implies \[m_{n}(A_{n}+B_{n}+K_{n})\leq m_{n}(A_{n}+K_{n})+m_{n}(B_{n}+K_{n})\] for $\mathcal U$-many $n$.

\noindent \textbf{Subcase 2.2:}  $\mb A, \mb B, \mb K$ has type (II). In this case we will prove that for $\mathcal U$-many $n$, either $A_{n}, B_{n}, K_{n}$ satisfies conclusion (I) or conclusion (II) in Theorem \ref{thm:Precise}.  Since $\mb A, \mb B, \mb K$ has type (II), Corollary \ref{cor:StabilizerUltra} implies \[\mu(\mb A+\mb B+\mb K)=\mu(\mb A+\mb K)+\mu(\mb B+\mb K)-\mu(\mb K),\] so in particular $\mu(\mb A+\mb B+\mb K)<\mu(\mb A+\mb K)+\mu(\mb B+\mb K)$. Then Lemma \ref{lem:FiniteIsomorphism} implies $m_{n}(A_{n}+B_{n}+K_{n})<m_{n}(A_{n}+K_{n})+m_{n}(B_{n}+K_{n})$ for $\mathcal U$-many $n$.  If $A_{n}+B_{n}$ is $\varepsilon$-periodic with respect to $K_{n}$ for $\mathcal U$-many $n$, then the last inequality implies $A_{n}, B_{n}$, and $K_{n}$ satisfy  conclusion (I) of Theorem \ref{thm:Precise} for $\mathcal U$-many $n$.  Otherwise, Lemmas \ref{lem:FiniteIsomorphism} and \ref{lem:PullQPdown} imply $A_{n}, B_{n}$, and $K_{n}$ satisfy conclusion (II) of Theorem \ref{thm:Precise} for $\mathcal U$-many $n$.

\noindent \textbf{Subcase 2.3:}  $\mb A, \mb B$ has type (III), meaning there are Bohr intervals $A', B'\subseteq \mb K$, and $\mb a, \mb b\in G$ such that $\mu(A')=\mu(\mb A)$, $\mu(B')=\mu(\mb B)$, and $\mb A\subseteq \mb a+A'$, $\mb B\subseteq \mb b+ B'$.   Then for $\mathcal U$-many $n$, we have that $A_{n}$ and $B_{n}$ are each contained in a coset of $K_{n}$. If for $\mathcal U$-many $n$ we have $m_{n*}(A_{n}+B_{n})\geq (1-\varepsilon)m_{n}(K_{n})$, then $A_{n}, B_{n}$, and $K_{n}$ satisfy conclusion (I) of Theorem \ref{thm:Precise} for these $n$. Otherwise, we apply Corollary \ref{cor:PullPreintervalsDown} to the sets $\mb A-\mb a, \mb B-\mb b$ (considered as subsets of $\mb K$), and find that $A_{n}, B_{n}$, and $K_{n}$ satisfy conclusion (III) of Theorem \ref{thm:Precise} for $\mathcal U$-many $n$.

In each case, we have shown that for infinitely many $n$, the sets $A_{n}, B_{n}$ satisfy the conclusion of Theorem \ref{thm:Precise} with a subgroup $K_{n}\leq G_{n}$ having index $1/\mu(\mb K)<n$ for $\mathcal U$-many $n$, contradicting our assumption to the contrary.  This completes the proof of Theorem \ref{thm:Precise}. \hfill $\square$

\section{Proof of Theorem \ref{thm:Popular}}\label{sec:PopularProof}

Theorem \ref{thm:Popular} is proved immediately after the proof of Lemma \ref{lem:FirstHalf}.

For the lemmas in this section, let $\mathcal U$ be a nonprincipal ultrafilter on $\mathbb N$, $(G_{n})_{n\in \mathbb N}$ a sequence of compact abelian groups with Haar probability measure $m_{n}$, and $\mb G=\prod_{n\to \mathcal U} G_{n}$ the ultraproduct with corresponding Loeb measure $\mu$.

\begin{lemma}\label{lem:NonAtomic} If $\mb G$ has infinite cardinality and $A\subseteq \mb G$ is $\mu$-measurable with $\mu(A)>0$, then for all $c$ such that $0\leq c<\mu(A)$, there is an internal set $\mb C\subseteq A$ such that $\mu(\mb C)=c$.
\end{lemma}

\begin{proof}
	Apply Proposition \ref{prop:LoebApprox} to find an internal set $\mb A'=\prod_{n\to \mathcal U} A_{n}'\subseteq \mb G$ such that $\mb A'\subseteq A$ and $\mu(\mb A')>c$. Since $\mb G$ has infinite cardinality, then for all $N\in \mathbb N$ and $\mathcal U$-many $n$ the group $G_{n}$ is infinite, or $G_{n}$ is finite and $|G_{n}|>N$.  In either case we may choose, for each $N\in \mathbb N$, an element $U_{N}\in \mathcal U$ and for each $n\in U_{N}$, compact sets $C_{N,n}'\subseteq A_{n}'$ such that $|m_{n}(C_{N,n}')-c|< \frac{1}{N}$. Choosing $U_{N}$ so that $U_{N+1}\subseteq U_{N}$ for each $N$, we may define $C_{n}=C_{N,n}'$ if $n\in U_{N}\setminus U_{N+1}$, and  $C_{n}=\varnothing$ if $n\notin U_{1}$.  Then $\mb C=\prod_{n\to \mathcal U}C_{n}$ is an internal subset of $A$ with $\mu(\mb C)=c$.
\end{proof}

\begin{corollary}\label{cor:Atomless}
	If $\mb G$ has infinite cardinality and $A\subseteq \mb B\subseteq \mb G$ are $\mu$-measurable sets, $\mb B$ is an internal set, and $\mu(A)<c<\mu(\mb B)$, then there is an internal set $\mb C$ such that $A\subseteq \mb C\subseteq \mb B$ and $\mu(\mb C)=c$.
\end{corollary}

\begin{proof}
	Choose, by Proposition \ref{prop:LoebApprox}, an internal set $\mb A'\supseteq A$ such that $\mu(\mb A')<c$.  Then $\mb A'':=\mb A'\cap \mb B$ is an internal set containing $A$ and having $\mu(\mb A'')<c$.   Use Lemma \ref{lem:NonAtomic} to choose an internal set $\mb C'\subseteq \mb B\setminus \mb A''$ such that $\mu(\mb C')=c-\mu(\mb A'')$.  Then $\mb C:=\mb C'\cup \mb A''$ is the desired set.
\end{proof}

\begin{lemma}\label{lem:FirstHalf}
	Assume $A, B\subseteq \mb G$ have $\mu(A), \mu(B)>0$, and $\mu_{*}(A+B)\leq \mu(A)+\mu(B)$.  Then for all $\varepsilon>0$ there are $\mu$-measurable internal sets $\mb S=\prod_{n\to \mathcal U}S_{n}$ and $\mb T=\prod_{n\to \mathcal U} T_{n}$ such that $\mu(A \triangle \mb S)+\mu(B \triangle \mb T)<\varepsilon$ and for $\mathcal U$-many $n$, $S_{n}+T_{n}$ is $m_{n}$-measurable and $m_{n}(S_{n}+T_{n})\leq m_{n}(S_{n})+m_{n}(T_{n})$.
\end{lemma}

\begin{proof}  First observe that the conclusion of the lemma is trivial when $\mb G$ is a finite group, as every subset of such a group $\mb G$ is internal and represented by a sequence of sets constant on some element of $\mathcal U$.   We therefore assume that $\mb G$ has infinite cardinality.
	
	Note that if we find $S_{n}, T_{n}$ satisfying $m_{n*}(S_{n}+T_{n})\leq m_{n}(S_{n})+m_{n}(T_{n})$, we may obtain the $m_{n}$-measurability of $S_{n}+T_{n}$ by replacing $S_{n}$ and $T_{n}$ with countable unions of compact subsets of $S_{n}$ and $T_{n}$ having the same $m_{n}$-measure as $S_{n}$ and $T_{n}$.  So we will not address $m_{n}$-measurability further.
	
	We will use the following observation repeatedly.
	
	\begin{observation}\label{obs:Done} If there are internal sets $\mb S=\prod_{n\to \mathcal U} S_{n}$ and $ \mb T=\prod_{n\to \mathcal U}T_{n}$ such that $\mu_{*}(\mb S+\mb T)<\mu(\mb S)+\mu(\mb T)$ and $\mu(A \triangle \mb S)+\mu(B \triangle \mb T)<\varepsilon$, then $\mb S, \mb T$ satisfy the conclusion of Lemma \ref{lem:FirstHalf}: under this assumption Lemma \ref{lem:InnerLoeb}  implies $m_{n*}(S_{n}+T_{n})<m_{n}(S_{n})+m_{n}(T_{n})$ for $\mathcal U$-many $n$.
	\end{observation}
	
	Now we prove Lemma \ref{lem:FirstHalf} in the case where $\mu_{*}(A+B)<\mu(A)+\mu(B)$.  Fix $\varepsilon>0$ such that $\mu_{*}(A+B)+\varepsilon<\mu(A)+\mu(B)$, and apply Proposition \ref{prop:LoebApprox} to find internal sets $\mb S\subseteq A$, $\mb T\subseteq B$ so that $\mu(\mb S)>\mu(A)-\frac{\varepsilon}{2}$ and $\mu(\mb T)>\mu(B)-\frac{\varepsilon}{2}$.  The containments $\mb S\subseteq A$, $\mb T\subseteq B$ and our choice of $\varepsilon$ then imply $\mu_{*}(\mb S+\mb T)<\mu(\mb S)+\mu(\mb T)$, so we are done by Observation \ref{obs:Done}.
	
	To prove Lemma \ref{lem:FirstHalf} in the case where $\mu_{*}(A+B)=\mu(A)+\mu(B)$, we appeal to Corollary \ref{cor:Iterate} and consider the five possibilities therein.
	
	\noindent \textbf{Case 1:} conclusion (a) holds in Corollary \ref{cor:Iterate}.   In this case, $A+B\sim_{\mu} A+B+\mb K$ for an internal $\mu$-measurable finite index subgroup $\mb K=\prod_{n\to \mathcal U} K_{n} \leq \mb G$.  If $A\sim_{\mu} A+\mb K$ and $B\sim_{\mu} B+\mb K$, we let $\mb S=A+\mb K$ and $\mb T=B+\mb K$.  Then $\mb S$ and $\mb T$ are internal, and Lemma \ref{lem:FiniteIsomorphism} implies \[m_{n}(S_{n}+T_{n})=\mu(A+B+\mb K)\leq \mu(A+\mb K)+\mu(B+\mb K)\leq m_{n}(S_{n})+m_{n}(T_{n}),\]
	so $m_{n}(S_{n}+T_{n})\leq m_{n}(S_{n})+m_{n}(T_{n})$ for $\mathcal U$-many $n$. If $A\nsim_{\mu} A+\mb K$ or $B\nsim_{\mu} B+\mb K$, Corollary \ref{cor:Atomless} provides internal sets $\mb S$, $\mb T$ with $A\subseteq \mb S\subseteq A+\mb K$ and $B\subseteq \mb T\subseteq B+\mb K$ such that $0<\mu(\mb S\setminus A)+\mu(\mb T\setminus B)<\varepsilon$.  Then $\mb S+\mb T\subset_{\mu} A+B+\mb K$, so $\mb S+\mb T\subset_{\mu} A+B$, and $\mu(\mb S+\mb T)<\mu(\mb S)+\mu(\mb T)$.  We are done by Observation \ref{obs:Done}.
	
	\noindent \textbf{Case 2:}  Conclusion (b) holds in Corollary \ref{cor:Iterate}.  In this case there are parallel Bohr intervals $\tilde{I}, \tilde{J}\subseteq \mb G$ such that $A\subseteq \tilde{I}, B\subseteq \tilde{J}$, and $\mu(A)=\mu(\tilde{I})$, $\mu(B)=\mu(\tilde{J})$.  Lemma \ref{lem:InternalIntervals} then implies that for $\mathcal U$-many $n$, there are sets $S_{n}\supseteq A_{n}$, $T_{n}\supseteq B_{n}$, such that $\mb S:=\prod_{n\to \mathcal U}S_{n}$ and $\mb T:=\prod_{n\to \mathcal U} T_{n}$ satisfy $\mb S\sim_{\mu} \tilde{I}$, $\mb T\sim_{\mu} \tilde{J}$, and $S_{n}$ and $T_{n}$ are either parallel Bohr intervals or parallel $N$-cyclic progressions for some $N$. We then have $m_{n}(S_{n}+T_{n})\leq m_{n}(S_{n})+m_{n}(T_{n})$ for $\mathcal U$-many $n$, while $\mb S\sim_{\mu}A$ and $\mb T\sim_{\mu} B$.

	\noindent \textbf{Case 3:}  $A, B$ satisfies conclusion (c) of Corollary \ref{cor:Iterate}.  In this case  choose an internal $\mu$-measurable finite index subgroup $\mb K=\prod_{n\to \mathcal U} K_{n}\leq \mb G$ such that $\mu(\mb K)<\frac{\varepsilon}{2}$ and $A$, $B$ have quasi-periodic decompositions $A= A_{1}\cup A_{0}$, $B=B_{1}\cup  B_{0}$, with respect to $\mb K$.   Let $\mb S=A+\mb K$ and $\mb T=B+\mb K$, so that $\mu(\mb S\setminus A)\leq \mu(\mb K)-\mu(A_{0})<\frac{\varepsilon}{2}$ and similarly $\mu(\mb T\setminus B)< \frac{\varepsilon}{2}$.  Since $A, B,\mb K$ has type (II), Corollary \ref{cor:StabilizerUltra} implies \[\mu(A+B+\mb K)=\mu(A+\mb K)+\mu(B+\mb K)-\mu(\mb K).\] Thus $\mu(\mb S+\mb T)<\mu(\mb S)+\mu(\mb T)$, and we are done by Observation \ref{obs:Done}.
	
	\noindent \textbf{Case 4:}  $A, B$ satisfies conclusion (d) of Corollary \ref{cor:Iterate}.   Suppose, without loss of generality, that $A$ has a quasi-periodic decomposition $A=A_{1}\cup A_{0}$ with respect to $\mb K$ such that $\mu(A_{0})=0$.  Now $A\sim_{\mu} A_{1}\sim_{\mu} A_{1}+\mb K$ and $A_{1}+B\sim_{\mu} A_{1}+B+\mb K$.  We can then argue as in Case 1 with $A_{1}$ in place of $A$, and this will suffice as $A_{1}\sim_{\mu} A$.

	\begin{observation}\label{obs:III} We have already proved the present lemma in the case where $A, B$ has type (III).  In the remainder of the proof, we may apply the lemma to any pair having type (III).\end{observation}
	
	\noindent\textbf{Case 5:}  $A, B$ satisfies conclusion (e) of Corollary \ref{cor:Iterate}.    In this case $A$ and $B$ have quasi-periodic decompositions $A=A_{1}\cup A_{0}$, $B=B_{1}\cup B_{0}$ with respect to an internal $\mu$-measurable finite index subgroup $\mb K=\prod_{n\to \mathcal U} K_{n}$, and $A+B\nsim_{\mu} A+B+\mb K$.  Furthermore $\mu(A_{0}), \mu(B_{0})>0$, $\mu_{*}(A_{0}+B_{0})=\mu(A_{0})+\mu(B_{0})$, and $A_{0}, B_{0}$ satisfies conclusion (III) in Proposition \ref{prop:Reduction}. By Observation $\ref{obs:III}$, there are internal sets $\mb S_{0}=\prod_{n\to \mathcal U} S_{n,0}$, $\mb T_{0}=\prod_{n\to \mathcal U} T_{n,0}$ such that $\mu(A_{0}\triangle \mb S_{0})+\mu(B_{0}\triangle \mb T_{0})<\varepsilon$ and
	\begin{align}\label{eqn:SnTncritical}
	m_{n}(S_{n,0}+T_{n,0})\leq m_{n}(S_{n,0})+m_{n}(T_{n,0}) && \text{for } \mathcal U\text{-many } n.
	\end{align} We also have $A_{0}+\mb K=\mb S_{0}+\mb K$ and $B_{0}+\mb K=\mb S_{0}+\mb K$, as $A_{0}$ and $B_{0}$ are each contained in a single coset of $\mb K$.
	Setting $\mb S_{1}=A_{1}+\mb K$, $\mb T_{1}=B_{1}+\mb K$ and $\mb S=\mb S_{1}\cup \mb S_{0}$, $\mb T=\mb T_{1}\cup \mb T_{0}$, we then have that $\mb S$ and $\mb T$ are internal, and $\mu(A\triangle \mb S)+\mu(B\triangle \mb T)<\varepsilon$.  We will show that $\mb S$ and $\mb T$ satisfy the conclusion of the present lemma.
	
	Writing $\mb S_{1}=\prod_{n\to \mathcal U} S_{n,1}$, Lemma \ref{lem:StabilizeInLimit} implies $S_{n,1}=S_{n,1}+K_{n}$ for $\mathcal U$-many $n$. Note that $A+\mb K=\mb S+\mb K$ and $B+\mb K=\mb T+\mb K$, and Corollary \ref{cor:StabilizerUltra} implies $\mu(A+B+\mb K)=\mu(A+\mb K)+\mu(B+\mb K)-\mu(\mb K)$, so the same identity holds with $\mb S, \mb T$ in place of $A,B$.  Lemma \ref{lem:StabilizeInLimit} then implies \begin{align}\label{eqn:SnTnKneser}
	m_{n}(S_{n}+T_{n}+K_{n})=m_{n}(S_{n}+K_{n})+m_{n}(T_{n}+K_{n})-m_{n}(K_{n}) && \text{ for } \mathcal U\text{-many } n.
	\end{align}Lemma \ref{lem:PullQPdown} shows that the decompositions $S_{n}=S_{n,1}\cup S_{n,0}$ and $T_{n}=T_{n,1}\cup T_{n,0}$ are indeed quasi-periodic decompositions with respect to $K_{n}$.   As observed in Remark \ref{rem:UEE}, $A_{0}+B_{0}+\mb K$ is a unique expression element of $A+B+\mb K$ in $\mb G/\mb K$, so $\mb S_{0}+\mb T_{0}+\mb K$ is as well. The isomorphism $\phi_{n}: G_{n}/K_{n}\to \mb G/\mb K$ provided by Lemma \ref{lem:FiniteIsomorphism} shows that $S_{n,0}+T_{n,0}+K_{n}$ is a unique expression element of $S_{n}+T_{n}+K_{n}$ in $G_{n}/K_{n}$ for $\mathcal U$-many $n$. Now (\ref{eqn:SnTncritical}),  (\ref{eqn:SnTnKneser}), and Lemma \ref{lem:BuildQP} imply $m_{n}(S_{n}+T_{n})\leq m_{n}(S_{n})+m_{n}(T_{n})$ for $\mathcal U$-many $n$, as desired. \end{proof}

\begin{proof}[Proof of Theorem \ref{thm:Popular}]  Suppose, to get a contradiction, that Theorem \ref{thm:Popular} is false. Then for all sufficiently large $n\in \mathbb N$ there is a compact abelian group $G_{n}$ with Haar probability measure $m_{n}$ and sets $A_{n}, B_{n}\subseteq G_{n}$ having $m_{n}(A_{n})>\varepsilon$, $ m_{n}(B_{n})>\varepsilon$, and
	\begin{equation}\label{eqn:SmallPopular}
	m_{n}(A_{n}+_{\frac{1}{n}} B_{n})\leq m_{n}(A_{n})+m_{n}(B_{n})+\tfrac{1}{n},
	\end{equation} while
	\begin{equation}\label{eqn:LastContradiction}  m_{n}(A_{n} \triangle S_{n}) + m_{n}(B_{n}\triangle T_{n})\geq \varepsilon
	\end{equation}
	for every pair $S_{n}, T_{n}\subseteq G_{n}$ satisfying $m_{n}(S_{n}+T_{n})\leq m_{n}(S_{n})+m_{n}(T_{n})$.
	
	Let $\mathcal U$ be a nonprincipal ultrafilter on $\mathbb N$, and let $\mb G=\prod_{n\to \mathcal U} G_{n}$, with $\mu$ the Loeb measure on $\mb G$ corresponding to $m_{n}$.  Consider the following sets and functions.
	
	Let $f_{n}:=1_{A_{n}}$, $g_{n}:=1_{B_{n}}$, and $h_{n}=f_{n}*g_{n}$. Write $C_{n}$ for  $A_{n}+_{\frac{1}{n}}B_{n}:=\{t\in G_{n}:f_{n}*g_{n}(t)>\frac{1}{n}\}$, the $\frac{1}{n}$-popular sumset of $A_{n}$ and $B_{n}$.
	Let $\mb A:=\prod_{n\to \mathcal U} A_{n}$, $\mb B:=\prod_{n\to \mathcal U} B_{n}$, and $\mb C:=\prod_{n\to \mathcal U} C_{n}$. The assumption (\ref{eqn:SmallPopular}) implies
	\begin{equation}\label{eqn:CAB}
	\mu(\mb C)\leq \mu(\mb A)+\mu(\mb B).
	\end{equation}

	Let $f:=\lim_{n\to \mathcal U} f_{n}$, $g:=\lim_{n\to \mathcal U} g_{n}$, and $h:=\lim_{n\to \mathcal U} f_{n}*g_{n}$. Then $h=f*g$, by Lemma \ref{lem:Ergodicity}.

	\begin{claim}
		$\{\mb x\in \mb G: f*g(\mb x)>0\}\subseteq \mb C$.
	\end{claim}
	\begin{proof}[Proof of Claim.]
		If $h(\mb x)>0$ and $\mb x \sim_{\mathcal U} (x_{n})_{n\in \mathbb N}$, then for some $s>0$, $h_{n}(x_{n})>s$ for $\mathcal U$-many $n$.  Thus $x_{n}\in A_{n}+_{\frac{1}{n}}B_{n}$ for $\mathcal U$-many $n$, so $\mb x\in \mb C$. \end{proof}
	Apply Lemma \ref{lem:ProjectAndPullBack} to find $\mu$-measurable sets $A', B' \subseteq \mb G$ such that $\mb A\subset_{\mu} A'$, $\mb B\subset_{\mu} B'$, and $A'+B'\subset_{\mu} \mb \{\mb x: f*g(\mb x)>0\}$, so that $\mu(A'+B')\leq \mu(\mb C)$.    Let $A''=\mb A\cap A'$, $B''=\mb B\cap B'$, so that $A''\sim_{\mu} \mb A$, $B''\sim_{\mu} \mb B$. Then
	\begin{align*}
	\mu_{*}(A''+B'') &\leq \mu(A'+B')  && \text{since $A''\subseteq A'$ and $B''\subseteq B'$}\\
	& \leq \mu(\mb C) && \text{since $A'+B'\subseteq \mb C$}\\
	&\leq \mu(\mb A)+\mu(\mb B) && \text{by (\ref{eqn:CAB})}\\
	&= \mu(A'')+\mu(B''), && \text{since $A''\sim_{\mu}\mb A, B''\sim_{\mu} \mb B$}
	\end{align*}
	so $\mu_{*}(A''+B'')\leq \mu(A'')+\mu(B'')$. We may therefore apply Lemma \ref{lem:FirstHalf} to $A''$ and $B''$ to find internal sets $\mb S=\prod_{n\to \mathcal U} S_{n}$ and $\mb T=\prod_{n\to \mathcal U} T_{n}$ such that $\mu(A''\triangle \mb S)+\mu(B''\triangle \mb T)<\varepsilon$, while
	\begin{equation}\label{eqn:STsmall}
	m_{n}(S_{n}+T_{n})\leq m_{n}(S_{n})+m_{n}(T_{n}) \quad \text{for } \mathcal U\text{-many } n.
	\end{equation}  Since $\mb A\sim_{\mu} A''$ and $\mb B\sim_{\mu} B''$, we also have $\mu(\mb A\triangle \mb S)+\mu(\mb B\triangle \mb T)<\varepsilon$, meaning
	\begin{equation}\label{eqn:STcloseAB}
	m_{n}(A_{n}\triangle S_{n})+m_{n}(B_{n}\triangle T_{n})<\varepsilon \quad \text{for } \mathcal U\text{-many } n.
	\end{equation}  Thus for $\mathcal U$-many $n$, both (\ref{eqn:STsmall}) and (\ref{eqn:STcloseAB}) hold, contradicting our assumptions on $A_{n}$ and $B_{n}$. This completes the proof of Theorem \ref{thm:Popular} \end{proof}

\section{Questions}\label{sec:Questions and remarks}

In \S\ref{sec:Niveau} we list some classical examples related to the hypothesis of Theorem \ref{thm:Precise}.  These motivate some of the questions posed in \S\ref{sec:Problems}.

\subsection{Possible lack of structure in \texorpdfstring{$A+B$}{A+B}}\label{sec:Niveau}   Conclusion (I) of Theorem \ref{thm:Precise} permits the following: $A+B$ is contained in a coset of $K$ and $m_{*}(A+B)\approx m(K)$.  This imposes almost no structure on $A$, $B$, or $A+B$, raising the question of whether additional detail can be obtained under these circumstances.   Here are a few known examples showing that some natural candidates  for stronger conclusions in Theorem \ref{thm:Precise} are impossible.

\subsubsection{Cyclic groups }\label{sec:ZpExample}  The techniques of \cite{Ruzsa91} show that for all $\varepsilon>0$ and every sufficiently large prime $p$ there is a set $A\subseteq \mathbb Z/p\mathbb Z$ such that $|A|\approx \frac{p}{2}$ while $A-A$ does not contain an arithmetic progression of length at least $\exp((\log p)^{\frac{2}{3}+\varepsilon})$.

\subsubsection{Tori}\label{sec:T2Example} For all  $d\in \mathbb N$ and $\varepsilon>0$, there is a set $A\subseteq \mathbb T^{d}$ having $m(A)>\frac{1}{2}-\varepsilon$, such that $A=-A$ and $A+A$ does not contain a Bohr interval.  In fact $A$ may be constructed so that $A+A$ does not contain coset of a nontrivial connected subgroup of $\mathbb T^{d}$. This is a folklore theorem, sometimes attributed to J. Bourgain answering a question of Y. Katznelson.

\subsubsection{Finite vector spaces}\label{sec:VecExample} For all $k\in \mathbb N$ and $N$ sufficiently large depending on $k$ there is a set $A\subseteq (\mathbb Z/2\mathbb Z)^{N}$ with $|A|>\bigl(\frac{1}{2}-\frac{1}{k}\bigr)2^{N}$ such that $A-A$ does not contain a coset of a subgroup of index less than $k$.  See Theorem 9.4 of \cite{GreenFiniteField} for such an example\footnote{The example in \cite{GreenFiniteField} produces $A$ with $|A| \approx \frac{1}{4}2^{N}$; the argument is easily adapted to produce densities approaching $\frac{1}{2}$ by replacing the quantity $\sqrt{n}$ therein by $n^{1/3}$, for example.} and \cite{SandersSumsetProblem} for discussion.  A similar construction for finite and countably infinite vector spaces over fields of odd order appears in \cite{GriesmerBohrDifference}.

\medskip

These constructions all rely on the same basic idea, sometimes called ``niveau sets'', introduced independently in \cite{Kriz} and \cite{Ruzsa87}. A very clear exposition appears in \cite{WolfPopular}, which constructs sets $A$ in $\mathbb Z/N\mathbb Z$ and $(\mathbb Z/2\mathbb Z)^{N}$ whose popular difference sets $A+_{\delta}(-A)$ are remarkably lacking in structure.

\subsection{Open problems}\label{sec:Problems}

\subsubsection{When \texorpdfstring{$A+B$}{A+B} nearly covers \texorpdfstring{$G$}{G}}  The examples listed in \S\ref{sec:Niveau} lead to the following open-ended problem: given a compact abelian group $G$ and small $\varepsilon, \delta>0$, find a complete description of all pairs of sets $A, B\subseteq G$ satisfying  $m(A), m(B)>\varepsilon$ and $m(A)+m(B)>1-\varepsilon$ such that $A+B$ does not contain any of the following: a coset of a subgroup of index at most $\delta^{-1}$, a relative Bohr interval of measure at least $\delta$, or a relative $N$-cyclic progression of measure at least $\delta$.  As mentioned in \S\ref{sec:Niveau}, all known examples of such pairs rely on the idea of niveau sets.  It would be very interesting to prove that all such examples must be based on this idea, but we do not  have a precise conjecture to that effect.  It would be equally interesting to find a fundamentally different example.  Progress in this direction could lead to  resolution of some open problems on difference sets $A-A$ where $A$ is a subset of a countable abelian group having positive upper Banach density, cf.~\cite{BergelsonRuzsa,GrPreprint,HegyvariRuzsa}.

\subsubsection{Quantifying the dependence of \texorpdfstring{$\delta$}{d} on \texorpdfstring{$\varepsilon$}{e}}\label{sec:QuantitativeDependence}  We would prefer to state a specific $\delta$ in terms of $\varepsilon$ in Theorems \ref{thm:Popular} and \ref{thm:Precise}, but our present methods will not yield any quantitative dependence. With additional hypotheses on the ambient group $G$, and in some cases on $m(A)$ and $m(B)$, such quantitative results are provided in  \cite{Bilu,CandelaDeRoton,CandelaSerraSpiegel,DeshouillersFrieman,MazurZp,MazurInZ,SerraZemor}, and Chapters 19 and 21 of \cite{GrynkBook}.  A prototypical example of such a result is \cite[Theorem 2.11]{NathansonBook} (originally in \cite{Freiman1}). An optimistic hope is that for all $\eta<1$ and $\varepsilon>0$, the hypothesis \[m_{*}(A+B)\leq m(A)+m(B)+\eta\min\{m(A),m(B)\}<1-\varepsilon\] implies a conclusion resembling that of Theorem \ref{thm:Precise}.  Even in the group $\mathbb T$ with $\eta=\frac{1}{3}$ this is too optimistic, as shown by Example 2.17 of \cite{CandelaDeRoton}.  That example is based on a similar one for $\mathbb Z/p\mathbb Z$ in \cite{SerraZemor}.

\subsubsection{The hypothesis \texorpdfstring{$m(A), m(B)>\varepsilon$}{m(A), m(B)>e}.}  In Theorem \ref{thm:Precise}, can we change the hypothesis ``$m_{*}(A+B)\leq m(A)+m(B)+\delta$'' to  ``$m_{*}(A+B)\leq m(A)+m(B)+\delta \min\{m(A),m(B)\}$'' and replace  ``$m(A),m(B)>\varepsilon$'' with ``$m(A), m(B)>0$'', without affecting the conclusion? Of course the bound on the index of $K$ would still depend on $m(A)$ and $m(B)$. Perhaps a method of rectification like those of \cite{GreenRuzsaRectification} can be combined with our techniques to  yield such a result.

\subsubsection{Semicontinuity in general}  Given $\varepsilon, \delta>0$, which pairs of subsets $A,B$ of a compact abelian group $G$ have the property that there are sets $A',B'\subseteq G$ such that
\[m(A\triangle A')+m(B\triangle B')<\varepsilon\] and $m_{*}(A'+B')-(m(A')+m(B'))\leq m_{*}(A+B)-(m(A)+m(B))-\delta$?

\subsubsection{Hope for a simpler proof of Theorem \ref{thm:Popular}.}  Is there a proof of Theorem \ref{thm:Popular} that does not use the detailed structural information contained in Theorem \ref{thm:LCAInverse}?  A simpler proof might generalize more easily to the nonabelian setting.

\subsubsection{The exceptional set in Proposition \ref{prop:Kronecker}}\label{sec:Exceptional}

As mentioned in Remark \ref{rem:AEnotE}, the equation $f*g \equiv_{\mu} (\tilde{f}*\tilde{g})\circ \pi$ in Proposition \ref{prop:Kronecker} cannot be improved to assert genuine equality. It may be interesting to characterise exactly which subsets of $\mb G$ can be the set of $\mb x$ where $f*g(\mb x) \neq (\tilde{f}*\tilde{g})\circ \pi(\mb x)$. Here is an example of $f$, $g$ where the exceptional set has some interesting structure, based on the sets mentioned in \S\ref{sec:VecExample}. Let $G_{n}=(\mathbb Z/2\mathbb Z)^{n}$, and let $\mb G=\prod_{n\to \mathcal U} G_{n}$.   Let $d_{n}$ be a sequence of natural numbers slowly tending to $\infty$, and choose $A_{n}\subseteq  G_{n}$ so that $|A_{n}|\geq (\frac{1}{2}-\frac{1}{d_{n}})|G_{n}|$ while $A_{n}+A_{n}$ does not contain a coset of a subgroup of index $d_{n}$.  The standard construction of such a set $A_{n}$ will also guarantee that $\lim_{n\to \mathcal U}\sup_{\chi \in \widehat{G}_{n}\setminus \{1\}} |\widehat{1}_{A_{n}}(\chi)|=0$. Letting $\mb A=\prod_{n\to \mathcal U}A_{n}$ and $f=1_{\mb A}$, we then have that $\mb G\setminus (\mb A+\mb A)$ has nonempty intersection with every coset of every internal finite index subgroup of $\mb G$.  Furthermore $f*f(\mb x)=0$ for all $\mb x\in \mb G\setminus (\mb A+\mb A)$, while $\tilde{f}\equiv_{m}\frac{1}{2}$, so $(\tilde{f}*\tilde{f})\circ \pi(\mb x)=\frac{1}{4}$ for all $\mb x\in \mb G$.  This construction guarantees that $\{\mb x: (f*f)(\mb x)\neq (\tilde{f}*\tilde{f})\circ \pi(\mb x)\}$  has nonempty intersection with every coset of every $\mu$-measurable finite index subgroup of $\mb G$.

\subsubsection{Nonabelian groups}

Kemperman \cite{KempermanCompact} partially extended Theorem \ref{thm:Satz1} to arbitrary locally compact (not necessarily abelian) groups: if $G$ is a locally compact group and $A, B\subseteq G$ satisfy $m(AB)<m(A)+m(B)$, then $AB$ is a union of cosets of a compact open subgroup of $G$.  No analogue to Equation (\ref{eqn:Kneser}) is proved there. Our current inability to extend Theorems \ref{thm:Popular} and \ref{thm:Precise} to non-abelian groups is partly due to the failure of Equation (\ref{eqn:Kneser}) in that setting, as shown by an example in \S5 of \cite{OlsonSum}.

For compact nonabelian groups with nontrivial abelian identity component, pairs $A,B$ satisfying $m(AB)=m(A)+m(B)$ and an additional hypothesis (called ``spread out'') are classified in \cite{BjorklundCompact} (here $AB$ denotes the product set $\{ab: a\in A, b\in B\}$).  It may be possible to extend the results of \cite{BjorklundCompact} to handle the hypothesis $m(AB)\leq m(A)+m(B)+\delta$ for very small $\delta$, perhaps with an additional hypothesis on the measure of $AB\cap cK$ for open subgroups $K$, in the following way: extend the ultraproduct machinery we use to nonabelian groups, and use the results of \cite{BjorklundCompact} in place of Theorem \ref{thm:LCAInverse}.

\appendix

\section{Ultraproducts and Loeb measure}\label{sec:Ultraproducts}

We summarize the ultraproduct construction and  properties of Loeb measure.  Our framework is basically identical to the one developed in \cite{SzegedyHOFA,SzegedyLimits}, and similar to that of \cite{BergelsonTaoQuasirandom} and \cite{TserunyanMixing}. The undergraduate thesis \cite{WarnerUltraproducts} is a very clear introduction to these topics.   The presentation here is to fix terminology and to state some important results from the literature.  We will assume familiarity with the basic properties of an ultrafilter on $\mathbb N$.

\subsection{Limits along ultrafilters}\label{Ulimits}
If $\mathcal U$ is an  ultrafilter on $\mathbb N$ and $P$ is a statement depending on $n\in \mathbb N$, we write ``$P(n)$ is true for $\mathcal U$-many $n$'' if $\{n\in \mathbb N: P(n) \text{ is true}\}\in \mathcal U$.

If $X$ is a topological space and $(x_{n})_{n\in \mathbb N}$ is a sequence of elements of $X$, we say that $\lim_{n\to \mathcal U} x_{n}=x$ if for every neighborhood $V$ of $x$, $\{n\in \mathbb N:x_{n}\in V\}\in \mathcal U$. With the above terminology we may say that $\lim_{n\to \mathcal U} x_{n}=x$ if for every neighborhood $V$ of $x$, $x_{n}\in V$ for $\mathcal U$-many $n$. When $X$ is a compact Hausdorff space, there is always a unique $x\in X$ such that $\lim_{n\to \mathcal U} x_{n}=x$.

\subsection{The ultraproduct construction}\label{sec:UPconstruction} Fix a nonprincipal ultrafilter $\mathcal U$ on $\mathbb N$.  If $(S_{n})_{n\in \mathbb N}$ is a sequence of sets, we form an equivalence relation $\sim_{\mathcal U}$ on the set of sequences $(s_{n})_{n\in \mathbb N}$ where $s_{n}\in S_{n}$, defined by $(s_{n})_{n\in \mathbb N} \sim_{\mathcal U} (t_{n})_{n\in \mathbb N}$ iff $\{n: s_{n}=t_{n}\}\in \mathcal U$.  The \emph{ultraproduct} $\prod_{n\to \mathcal U} S_{n}$, denoted $\mathbf S$, is the set of equivalence classes under the relation $\sim_{\mathcal U}$.  If $\mb s\in \mb S$ and $(a_{n})_{n\in \mathbb N}$ is a member of the equivalence class $\mb s$, we say that $\mb s$ is \emph{represented by} $(a_{n})_{n\in \mathbb N}$, and by abuse of notation we write $(a_{n})_{n\in \mathbb N}\sim_{\mathcal U} \mb s$.

Given a constant sequence of sets, $S_{n}=S$ for all $n$, we write $\leftidx^{*}S$ for the ultraproduct $\prod_{n\to \mathcal U} S$, which is then called an \emph{ultrapower} of the underlying set.

A word of warning: we use the symbols ``$\prod_{n\to \mathcal U}$'' and ``$\lim_{n\to \mathcal U}$'' to denote different operations.  Some articles use the symbol ``$\lim_{n\to \mathcal U}$'' to denote what we call ``$\prod_{n\to \mathcal U}$''.

\subsubsection{Internal sets} Given a sequence of sets $S_{n}$ and the corresponding ultraproduct $\mathbf S=\prod_{n\to\mathcal U} S_{n}$, a sequence of subsets $A_{n}\subseteq S_{n}$ determines an ultraproduct $\mathbf A=\prod_{n\to \mathcal U} A_{n}$ which is itself a subset of $\mathbf S$.  We call a set $\mb A\subseteq \mathbf S$ \emph{internal} if it arises in this way, i.e.~if there is a sequence of sets $A_{n}\subseteq S_{n}$ such that $\mb A=\prod_{n\to \mathcal U} A_{n}$.

We will always use boldface letters such as $\mb S$, $\mb A$, etc., to denote ultraproducts and internal sets.

\subsubsection{Standard elements} We call an element of an ultrapower $\leftidx^{*}S$ \emph{standard} if it has a constant representative $(s_{n})_{n\in \mathbb N}$, meaning there is an $s$ such that $s_{n}=s$ for all $n$.  The standard elements of $\leftidx^{*}S$ are naturally identified with elements of $S$.

\subsubsection{Infinitesimals and bounded elements}\label{sec:inf} The set of complex numbers $\mathbb C$ may be identified with the set of standard elements of the ultrapower $\leftidx^{*}\mathbb C$.  An element $\mb a$ of $\leftidx^{*}\mathbb C$ is called

\begin{enumerate}
	\item[$\cdot$] \emph{infinitesimal} if $\mb a$ has a representative $(a_{n})_{n\in \mathbb N}$  such that $\lim_{n\to \mathcal U} a_{n}=0$.
	
	\item[$\cdot$] \emph{bounded} if there is a $c \in \mathbb R$ and a representative $(a_{n})_{n\in \mathbb N}$ of $\mb a$ such that $|a_{n}|\leq c$ for $\mathcal U$-many $n$.
\end{enumerate}
For every bounded $a\in \leftidx^{*}\mathbb C$, there is a unique standard element $b\in \leftidx^{*}\mathbb C$ such that $a-b$ is infinitesimal.  Identifying $b$ with an element of $\mathbb C$, we write $\st(a)$ for $b$ and call $\st(a)$ the \emph{standard part} of $a$.

\subsubsection{Functions from \texorpdfstring{$\mb S$}{S} into compact spaces}  Given a compact Hausdorff topological space $Y$ and a sequence of functions $f_{n}:S_{n}\to Y$, consider the ultraproduct $\mb S=\prod_{n\to \mathcal U} S_{n}$ and let $\lim_{n\to \mathcal U} f_{n}$ denote the function $f:\mb S\to Y$ given by $f(\mb s)=\lim_{n\to \mathcal U} f_{n}(s_{n})$, where $(s_{n})_{n\in \mathbb N}$ represents $\mb s$.

\subsubsection{Internal functions}\label{sec:InternalFunctions}  If $(S_{n})_{n\in \mathbb N}$, $(T_{n})_{n\in \mathbb N}$ are sequences of sets with corresponding ultraproducts $\mb S$ and $\mb T$, and $f_{n}:S_{n}\to T_{n}$ is a function for each $n$, we may define a function $\mathbf f: \mb S\to \mb T$ as follows: $\mathbf f(\mathbf s)$ is the element of $\mb T$ represented by the sequence $(f_{n}(s_{n}))_{n\in \mathbb N}$.  Functions from $\mb S$ to $\mb T$ arising this way are called \emph{internal} functions.\footnote{It is routine to check that this definition of $\mb f$ agrees with the definition of $\mb f$ as the ultraproduct $\prod_{n\to \mathcal U} f_n$, where each $f_n$ is considered as a subset of $S_n\times T_n$.  Thus the terminology ``internal'' is appropriate.}
If $f_{n}:S_{n}\to \mathbb C$ for each $n$ and the $f_{n}$ are uniformly bounded, then $\mathbf f$ is a bounded function from $\mathbf S$ into $\leftidx^{*}\mathbb C$, and the function $\leftidx^{\circ}\mathbf f$ defined by $\leftidx^{\circ}\mathbf f(\mathbf s) := \st(\mathbf f(\mathbf s))$ is a function from $\mathbf S$ to $\mathbb C$.

When $\mb f: \mb S\to \leftidx^{*}\mathbb C$ is bounded, there is an another way to define $\lcirc f$.  If $(f_{n})_{n\in \mathbb N}$ is a sequence of uniformly bounded functions from $S_{n}$ to $\mathbb C$ and $\mb f=\prod_{n\to \mathcal U} f_{n}$, then for each $\mathbf s\in \mathbf S$ and each representative $(s_{n})_{n\in \mathbb N}$ of $\mathbf s$ we have $\lcirc f(\mathbf s) = \lim_{n\to \mathcal U} f_{n}(s_{n})$.  Thus $\lcirc f=\lim_{n\to \mathcal U} f_{n}$.

\subsection{Loeb measure}  We summarize the development of Loeb measure as used in \cite{SzegedyHOFA}, \cite{SzegedyLimits}, and \cite{BergelsonTaoQuasirandom}, for example.  See Chapter 1 of \cite{CutlandLoebMeasure} for an overview of the construction and references.

For each $n$, let $X_{n}$ be a set, $\mathscr X_{n}$ an algebra of subsets of $X_{n}$, and $\mu_{n}$ a finitely additive measure defined on $\mathscr X_{n}$.  Consider the ultraproduct $\mathbf X = \prod_{n\to \mathcal U} X_{n}$ and the
$\sigma$-algebra $\mathscr X$ of subsets of $\mb X$ generated by the internal sets of the form
$\prod_{n\to \mathcal U} A_{n}$, where $A_{n}\in \mathscr X_{n}$ for $\mathcal U$-many $n$.
Then there is a unique countably additive measure $\mu$ on $(\mb X,\mathscr X)$ such that for every internal set $\mathbf A=\prod_{n\to \mathcal U} A_{n}$ where $A_{n}\in \mathscr X_{n}$ for $\mathcal U$-many $n$, we have
\begin{equation}\label{eqn:Loeb}
\mu(\mathbf A) = \lim_{n\to \mathcal U} \mu_{n}(A_{n}).
\end{equation}
Extending $\mathscr X$ to its completion $\mathscr X'$ under $\mu$, the measure $\mu$ is called the \emph{Loeb measure} associated to $(X_{n},\mathscr X_{n},\mu_{n})$, and $(\mb X, \mathscr X',\mu)$ is the \emph{Loeb measure space} associated to $(X_{n},\mathscr X_{n},\mu_{n}) $. Given such a measure $\mu$ and a function $f:\mb X\to \mathbb C$, we say that $f$ is \emph{$\mu$-measurable} (or \emph{Loeb measurable}, if there is no ambiguity) if $f$ is measurable with respect to $\mathscr X'$.  We say that a set $A\subseteq \mathbf X$ is \emph{$\mu$-measurable} if $A\in \mathscr X'$.

We need some natural approximation results for Loeb measurable functions, summarized in the following proposition.  Here $D\subseteq \mathbb C$ is the closed unit disk.

\begin{proposition}\label{prop:LoebApprox}  Let $(X_{n},\mathscr X_{n},\mu_{n})$ be a sequence of probability measure spaces and let $(\mb X, \mathscr X,\mu)$ be the associated Loeb measure space.
	\begin{enumerate}
		\item[(i)] \cite[Theorem 1]{LoebMeasure} If $A\subseteq \mb X$ is $\mu$-measurable then for all $\varepsilon>0$ there are internal sets $\mb A'$, $\mb A''$ such that $\mb A'\subseteq A\subseteq \mb A''$ and $\mu(\mb A''\setminus \mb A')<\varepsilon$.
		\item[(ii)]  \cite[Theorem 3]{LoebMeasure} If $f_{n}:X_{n}\to D$ is $\mathscr X_{n}$-measurable for $\mathcal U$-many $n$ and  $f=\lim_{n\to \mathcal U} f_{n}$, then $f$ is $\mu$-measurable and $\int f \,d\mu = \lim_{n\to \mathcal U} \int  f_{n}\, d\mu_{n}$.
	\end{enumerate}
\end{proposition}

\subsection{Ultraproducts of compact abelian groups}

If $(G_{n})_{n\in \mathbb N}$ is a sequence of abelian groups, the ultraproduct $\mb G=\prod_{n\to \mathcal U} G_{n}$ is an abelian group with addition defined by $\mb x+\mb y \sim_{\mathcal U} (x_{n}+y_{n})_{n\in \mathbb N}$, where $\mb x\sim_{\mathcal U} (x_{n})_{n\in \mathbb N}$ and $\mb y\sim_{\mathcal U} (y_{n})_{n\in \mathbb N}$.  The inverse and identity are defined similarly.  If each $G_{n}$ is a compact abelian group with Haar probability  measure $m_{n}$, then the Loeb measure $\mu$ corresponding to $m_{n}$ is translation invariant: $\mu(A+\mb t)=\mu(A)$ for all $\mu$-measurable $A\subseteq \mb G$ and all $\mb t\in \mb G$.  We will consider inner Haar measure $m_{n*}$ on $G_{n}$, so we need to relate $m_{n*}$ to the inner measure $\mu_{*}$ associated to $\mu$.  The next lemma says that the expected relationship holds.

\begin{lemma}\label{lem:InnerLoeb}
	If $\mb G$ and $\mu$ are as in the preceding paragraph and $A\subseteq \mb G$ is any set, then $\mu_{*}(A)=\sup\{\mu(\mb C):\mb C\in \mathcal S, \mb C\subseteq A\}$, where $\mathcal S$ is the class of internal sets of the form $\prod_{n\to \mathcal U} C_{n}\subseteq \mb G$ such that $C_{n}$ is compact for all $n$.  Consequently, if $\mb A=\prod_{n\to \mathcal U} A_{n}$ is an internal (not necessarily $\mu$-measurable) subset of $\mb G$, then $\mu_{*}(\mb A)=\lim_{n\to \mathcal U} m_{n*}(A_{n})$.
\end{lemma}

\begin{proof}
	Let $A\subseteq \mb G$ be an arbitrary set, let $\varepsilon>0$, and choose a $\mu$-measurable set $A'\subseteq A$ such that $\mu(A')>\mu_{*}(A)-\varepsilon$.  By Proposition \ref{prop:LoebApprox} choose an internal set $\mb B=\prod_{n\to \mathcal U} B_{n}\subseteq A'$ such that each $B_{n}$ is $m_{n}$-measurable and $\mu(\mb B)>\mu_{*}(A)-\varepsilon$.  For each $n$, choose a compact set $C_{n}\subseteq B_{n}$ such that $m_{n}(C_{n})>m_{n}(B_{n})-\frac{1}{n}$.  Let $\mb C=\prod_{n\to \mathcal U} C_{n}$, so that $\mb C\in \mathcal S$, $\mb C\subseteq A$, and $\mu(\mb C)=\mu(\mb B)>\mu_{*}(A)-\varepsilon$. This shows that $\mu_{*}(A)\leq \sup\{\mu(\mb C): \mb C\in \mathcal S, \mb C\subseteq A\}$, and the reverse inequality is obvious.
\end{proof}

\subsection{Convolutions on ultraproducts}\label{sec:Convolutions}

The ultraproduct $\mb G$ is not (in any nontrivial case) a locally compact group, so we need to prove some properties of convolutions on $\mb G$ which are inherited from the constituent groups.  We assume familiarity with basic properties of convolution on compact groups, see \cite{FollandAbstract} or \cite{RudinFourier}, for example. Recall the notation $f_{t}$ from Definition \ref{def:Translate}: $f_{t}(x):=f(x-t)$.

\begin{lemma}\label{lem:Ergodicity}
	Let $\mathcal U$ be a nonprincipal ultrafilter on $\mathbb N$, $(G_{n})_{n\in \mathbb N}$ a sequence of compact abelian groups with Haar probability measure $m_{n}$, and $\mb G=\prod_{n\to \mathcal U} G_{n}$ the ultraproduct with corresponding Loeb measure $\mu$. Let $f_{n}, g_{n}: G_{n}\to \mathbb C$ be uniformly bounded functions, and consider the internal functions $\mathbf f=\prod_{n\to \mathcal U} f_{n}$, $\mathbf g=\prod_{n\to \mathcal U} g_{n}$. Then
	\begin{enumerate}
		\item[(i)] $(\lcirc f)_{\mb t}= \leftidx^\circ(\mathbf f_{\mb t})$ for all $\mb t\in \mb G$.
		
		\smallskip
		
		\item[(ii)] $\lcirc f *_{\mu} \lcirc g=\lim_{n\to \mathcal U} f_{n}*_{m_{n}} g_{n}$.
		
	\end{enumerate}

\end{lemma}

\begin{proof}
	Part (i) follows from the relevant definitions.   To prove Part (ii) let $\mb x\in \mb G$ and choose $x_{n}\in G_{n}$ so that $\mb x\sim_{\mathcal U} (x_{n})_{n\in \mathbb N}$. Then $\lcirc f *_{\mu} \lcirc g(\mb x)=\int \lcirc f(\mb t)\lcirc g(\mb x-\mb t) d\mu(\mb t)$, which equals $\lim_{n\to \mathcal U} \int f_{n}(t)g(x_{n}-t)\, dm_{n}(t)$ by Part (ii) of Proposition \ref{prop:LoebApprox} and Part (i) of the present lemma.  The last integral can be rewritten as $\lim_{n\to \mathcal U} (f_{n}*_{m_{n}}g_{n})(x_{n})$, as desired. \end{proof}

\providecommand{\bysame}{\leavevmode\hbox to3em{\hrulefill}\thinspace}
\providecommand{\href}[2]{#2}


\begin{dajauthors}
\begin{authorinfo}[john]
John T. Griesmer\\
Department of Applied Mathematics and Statistics\\
 Colorado School of Mines\\
 Golden, CO,  USA\\
  jtgriesmer\imageat{}gmail\imagedot{}com
\end{authorinfo}
\end{dajauthors}

\end{document}